\documentclass[a4paper]{article}
\usepackage[utf8]{inputenc}

\usepackage{amsthm}
\usepackage{amsmath}
\usepackage{amssymb}
\usepackage{amsfonts}
\usepackage{latexsym}
\usepackage{mathrsfs}
\usepackage{mathtools}
\usepackage{url}
\usepackage{enumerate}
\usepackage{enumitem}
\usepackage[margin=25mm]{geometry}
\usepackage{authblk}
\usepackage{stmaryrd} 
\allowdisplaybreaks

\mathtoolsset{showonlyrefs}

\usepackage[dvipsnames]{xcolor}

\usepackage{hyperref}
\hypersetup{
    colorlinks=true,
    linkcolor={red!80!black},
    citecolor={blue!90!black},
    urlcolor=magenta,
    linktocpage=true,
}

\newtheoremstyle{non}
  {\topsep} 
  {\topsep} 
  {} 
  {} 
  {\bfseries} 
  {.} 
  {.5em} 
  {} 

\theoremstyle{non} 
\newtheorem*{thm*}{Theorem 1.1}

\theoremstyle{plain}
\newtheorem{Thm}{Theorem}[section]
\newtheorem{Lem}[Thm]{Lemma}
\newtheorem{Prop}[Thm]{Proposition}
\newtheorem{Cor}[Thm]{Corollary}
\theoremstyle{definition}
\newtheorem{Def}[Thm]{Definition}

\newtheorem{Rem}[Thm]{Remark}
\theoremstyle{remark}

\theoremstyle{plain}
\newtheorem*{Thm*}{Theorem}
\newtheorem*{Lem*}{Lemma}
\newtheorem*{Prop*}{Proposition}
\newtheorem*{Cor*}{Corollary}
\theoremstyle{definition}
\newtheorem*{Def*}{Definition}
\newtheorem*{Rem*}{Remark}
\theoremstyle{remark}
\newtheorem*{Eg*}{Example}

\numberwithin{equation}{section}

\DeclareMathOperator\supp{supp}

\DeclarePairedDelimiter{\norm}{\lVert}{\rVert} 
\DeclarePairedDelimiter{\abra}{\langle}{\rangle} 

\newcommand{\E}{\mathbb{E}}
\newcommand{\R}{\mathbb{R}}
\newcommand{\C}{\mathbb{C}}
\newcommand{\F}{\mathcal{F}}
\newcommand{\ind}{\mathbf{1}}

\title{Central limit theorems for stochastic wave equations \\ in high dimensions}
\author{Masahisa Ebina} 

\affil{Department of Mathematics, Graduate School of Science, Kyoto University, Kitashirakawa-Oiwakecho, Sakyo-ku, Kyoto 606-8502, Japan}
\affil{E-mail address: \texttt{ebina.masahisa.83m@st.kyoto-u.ac.jp}}
\date{}

\begin{document}

\maketitle
\begin{abstract}
We consider stochastic wave equations in spatial dimensions $d \geq 4$.
We assume that the driving noise is given by a Gaussian noise that is white in time and has some spatial correlation.
When the spatial correlation is given by the Riesz kernel, we establish that the spatial integral of the solution with proper normalization converges to the standard normal distribution under the Wasserstein distance. 
The convergence is obtained by first constructing the approximation sequence to the solution and then applying Malliavin-Stein's method to the normalized spatial integral of the sequence.
The corresponding functional central limit theorem is presented as well.
\end{abstract}

\noindent
\textbf{Keywords:} Stochastic wave equation, Central limit theorem, Malliavin calculus, Stein's method.\\
\textbf{2020 Mathematics Subject Classification:} 60F05, 60G15, 60H07, 60H15.

\tableofcontents

\section{Introduction}
\label{section Introduction}
We consider the following stochastic wave equation 
\begin{align}
\label{SPDE}
    \begin{cases}
        \displaystyle \frac{\partial^2U }{\partial t^2}(t,x) = \Delta U(t,x) + \sigma(U(t,x))\dot{W}(t,x), \quad (t,x) \in [0,T] \times \mathbb{R}^d,\\
        \displaystyle U(0,x) = 1, \quad  \frac{\partial U}{\partial t}(0,x) = 0,
    \end{cases}
\end{align}
where $T > 0$ is fixed, $d \geq 4$, $\sigma\colon \mathbb{R} \to \mathbb{R}$ is a continuously differentiable Lipschitz function, and $\dot{W}(t,x)$ is the formal notation of centered Gaussian noise defined on a complete probability space $(\Omega, \mathscr{F}, P)$ with the covariance structure
\begin{equation}
\label{dot W covariance}
\mathbb{E}[\dot{W}(t,x)\dot{W}(s,y)] = \delta_0(t-s)\gamma(x-y).
\end{equation}
Here $\delta_0$ is the Dirac delta function, and $\gamma$ denotes the spatial correlation of the noise given by a nonnegative function on $\mathbb{R}^d$ such that $\gamma(x)dx$ is a nonnegative definite tempered measure on $\R^d$.
Throughout the paper, we always assume that this $\gamma$ satisfies 
\begin{equation}
\label{Dalang's condition gamma}
    \int_{|x| \leq 1}\frac{\gamma(x)dx}{|x|^{d-2}} < \infty,
\end{equation}
where $|x|$ denotes the usual Euclidean norm of $x$ on $\mathbb{R}^d$.
Note that \eqref{Dalang's condition gamma} is equivalent to Dalang's condition \eqref{Dalang's condition} in our case (see Section \ref{section Preliminaries} for details).

We adopt Dalang-Walsh's random field approach to study this equation, and we consider a real-valued predictable process $\{U(t,x) \mid [0,T] \times \R^d\}$ satisfying 
\begin{equation*}
    U(t,x) = 1 + \int_0^t\int_{\R^d} G(t-s,x-y)\sigma(U(s,y))W(ds,dy), \quad \text{a.s.}
\end{equation*}
as a solution to the equation. 
Here $G$ denotes the fundamental solution of the $d$-dimensional wave equation, and the stochastic integral on the right-hand side is interpreted as Dalang's integral.
Under \eqref{Dalang's condition gamma} and the condition that $\sigma$ is a Lipschitz function, it is shown in \cite{MR2399293} that a solution to \eqref{SPDE} exists, and that uniqueness holds within a proper class of predictable processes.
See Sections \ref{section Preliminaries} and \ref{section Stochastic wave equations} for the detailed definitions of terminologies and the solution.

The solution process $\{U(t,x) \mid x \in \R^d \}$ of \eqref{SPDE} is known to be strictly stationary in any dimension $d \geq 1$ (see \cite{MR2399293, MR1684157}). 
In order to understand the spatial behavior of the solution, a natural problem is determining the conditions under which the solution process is ergodic with respect to spatial shifts.
In \cite{MR4346664}, this problem is addressed for stochastic heat equations in arbitrary spatial dimensions, and several sufficient conditions are provided for a stationary solution to be ergodic.
Moreover, using the same approach as in \cite{MR4346664}, the same sufficient conditions for stochastic wave equations in the case $d \leq 3$ are given in \cite{MR4187721}.
Immediately after \cite{MR4187721}, as a more general abstract result, the same sufficient conditions for stationary processes consisting of square-integrable Gaussian functionals are provided in \cite{MR4563698}.
From this general result, it follows that under Dalang's condition, the solution process $\{U(t,x) \mid x \in \R^d \}$ of \eqref{SPDE} is ergodic if 
\begin{equation}
\label{suff ergo}
    \lim_{R \to \infty} |\mathbb{B}_R|^{-1}\int_{\mathbb{B}_R}\gamma(x)dx = 0,
\end{equation}
where $\mathbb{B}_R \coloneqq \{y \in \R^d \mid |y| \leq R\}$ and $|\mathbb{B}_R|$ denotes the Lebesgue measure of $\mathbb{B}_R$.
Note that the condition \eqref{suff ergo} is proved in \cite{MR4346664} to be equivalent to $\mu(\{0\}) = 0$, where $\mu$ denotes the spectral measure of $\gamma$ (see Section \ref{section Preliminaries} for the details on $\mu$).
A straightforward consequence of spatial ergodicity is the law of large numbers for the solution: under condition \eqref{suff ergo}, we can deduce 
\begin{equation*}
    |\mathbb{B}_R|^{-1} \int_{\mathbb{B}_R}(U(t,x) - \E[U(t,x)])dx \xrightarrow[R \to \infty]{} 0, \quad \text{a.s. and in $L^p(\Omega)$ for any $p \in [1,\infty)$}
\end{equation*}
by applying the ergodic theorem (\textit{cf.} \cite[Theorems $\mathrm{I}''$ and $\mathrm{II}''$]{MR1546100} and \cite[Theorem 25.14]{Kallenberg}).
As a result, it becomes natural and interesting to investigate the fluctuations of the spatial average of the solution next.

The goal of the present paper is to prove the central limit theorems for stochastic wave equations in spatial dimensions $d \geq 4$.
We will show in Theorems \ref{Thm m1} and \ref{Thm m2} below that when the spatial correlation is given by $\gamma(x) = |x|^{-\beta}$ for some $0 < \beta < 2$, the centered spatial integral of the solution
\begin{equation}
\label{F_R(t)}
    F_R(t) \coloneqq \int_{\mathbb{B}_R}(U(t,x) - \E[U(t,x)])dx
\end{equation}
with appropriate normalization converges in distribution to the standard normal distribution for each $t >0$ and to a Gaussian process in $C([0,T])$ as $R \to \infty$.
Here $C([0,T])$ is the space of continuous functions on $[0,T]$.

This kind of asymptotic behavior of spatial averages for stochastic partial differential equations has been studied extensively in recent years.
This line of research was initiated by Huang, Nualart, and Viitasaari in \cite{MR4167203}, where they established quantitative central limit theorem and its functional version for one-dimensional nonlinear stochastic heat equations driven by space-time white noise using Malliavin-Stein's method (\textit{cf.} \cite{nourdin_peccati_2012}). 
After this result, similar central limit theorems for stochastic partial differential equations have been established under various settings.
Following the same approach as in \cite{MR4167203}, the authors of \cite{MR4098872} shows that similar central limit theorems hold for stochastic heat equations in arbitrary dimensions when the spatial correlation of the noise is given by $|x|^{-\beta}$ with $0 < \beta < \min\{2,d\}$.  
For other similar results for stochastic heat equations, see \textit{e.g.} \cite{MR4385408, CLTforHEwithtimeindependentnoise, MR4563698, MR4242879, MR4337703, MR4479916, averagingGaussianfunctionals, MR4391725} and the reference therein.

The same approach as in \cite{MR4167203} also led to similar results for stochastic wave equations. 
The one-dimensional case where the driving noise is a fractional Gaussian noise with Hurst parameter $H \in [1/2,1)$ is treated in \cite{MR4164864}, while the two-dimensional case where the spatial correlation of the noise is $|x|^{-\beta}$ with $0<\beta<2$ is considered in \cite{MR4290504}.
Moreover, \cite{MR4439987} shows similar results for integrable spatial correlations in one- and two-dimensional cases.
Among others, we refer to \cite{BalanZhengLevyHAM, MR4491503, BalanYuanHAM-with-time-independent-noise} for other related works.
However, in contrast to the results for stochastic heat equations in all spatial dimensions presented in \cite{MR4098872}, these results for stochastic wave equations are restricted to one- and two-dimensional cases. 
The approach used in \cite{MR4167203} cannot be directly applied to stochastic wave equations in dimensions three or higher due to a technical problem.
To effectively apply Malliavin-Stein’s method, we need some control over the moments of the Malliavin derivative of the solution, but obtaining such control proves challenging when $d \geq 3$.
The key estimate used in, \textit{e.g.}, \cite{MR4290504, MR4164864, MR4439987} is the following (almost) pointwise evaluation:
\begin{equation}
\label{DU estimate}
    \lVert D_{s,y}U(t,x) \rVert_{L^p(\Omega)} \lesssim G(t-s,x-y), \quad \text{for almost all $(s,y) \in [0,t] \times \R^d$},
\end{equation}
where $D$ is the Malliavin derivative operator (see Section \ref{subsection Malliavin calculus and Malliavin-Stein method}).
In the case of stochastic heat equations, the same kind of estimate holds in all spatial dimensions and is used in \cite{MR4167203, MR4098872}. 
However, such estimates cannot be expected to hold for stochastic wave equations when $d \geq 3$, as $G$ becomes a distribution rather than a function.
The difficulty in obtaining a precise bound for the Malliavin derivative of the solution directly leads to the difficulty in proving the central limit theorems via Malliavin-Stein's method.

For three-dimensional stochastic wave equations, it is shown in \cite{ebina2023central} that this technical problem can be avoided by approximation and that the same type of central limit theorems hold though the convergence rate is not provided. 
More precisely, when $d=3$ and $\gamma(x) = |x|^{-\beta}$ with $0 < \beta < 2$, we can construct a sequence $(u_n(t,x))_{n=0}^{\infty}$ that converges to the solution $U(t,x)$ in $L^p(\Omega)$ as $n \to \infty$ and satisfies
\begin{equation}
\label{int powise est}
    \lVert D_{s,y}u_n(t,x) \rVert_{L^p(\Omega)} \leq C_{n,p,T} \ind_{\mathbb{B}_{t-s+1}}(x-y), \qquad 0 \leq s \leq t \leq T
\end{equation}
with some constant $C_{n,p,T}$ depending on $n$, $p$, and $T$.
This estimate can be used as a substitute for \eqref{DU estimate}, and in view of Malliavin-Stein's method, it can show 
\begin{equation}
\label{int lim R}
    \lim_{R \to \infty}\mathrm{dist}\left(\frac{F_{n,R}(t)}{\sigma_{n,R}(t)}, \mathcal{N}(0,1) \right) = 0
\end{equation}
for each fixed $n$. 
Here $\mathrm{dist}(\cdot, \cdot)$ is some distance between probability measures, $F_{n,R}(t)$ is defined by \eqref{F_R(t)} with $U(t,x)$ replaced by $u_n(t,x)$, $\sigma_{n,R}(t) \coloneqq \sqrt{\mathrm{Var}(F_{n,R}(t))}$, and $\mathcal{N}(0,1)$ denotes the standard normal distribution.
With a little more effort, it can also be shown that
\begin{equation}
\label{int lim n}
    \lim_{n \to \infty}\sup_{R}\left\lVert \frac{F_{R}(t)}{\sigma_R(t)} - \frac{F_{n,R}(t)}{\sigma_{n,R}(t)} \right\rVert_{L^2(\Omega)} = 0,
\end{equation}
where $\sigma_R(t) \coloneqq \sqrt{\mathrm{Var}(F_{R}(t))}$.
Thus, if we choose $\mathrm{dist}(\cdot, \cdot)$ as the Wasserstein distance $d_{\mathrm{W}}$, which can be dominated by $L^2(\Omega)$ distance, then we obtain
\begin{align*}
    \lim_{R \to \infty}d_{\mathrm{W}}\left(\frac{F_{R}(t)}{\sigma_{R}(t)}, \mathcal{N}(0,1) \right)
    &\leq \sup_{R}d_{\mathrm{W}}\left(\frac{F_{R}(t)}{\sigma_R(t)}, \frac{F_{n,R}(t)}{\sigma_{n,R}(t)} \right) + \lim_{R \to \infty}d_{\mathrm{W}}\left(\frac{F_{n,R}(t)}{\sigma_{n,R}(t)}, \mathcal{N}(0,1) \right)\\
    &\leq \varepsilon
\end{align*}
for any given $\varepsilon \ll 1$ by letting $n$ large enough, which proves the aforementioned result.
Unfortunately, this argument still cannot apply directly to the high-dimensional case $d \geq 4$.
The obstacle here is that $G(t)$ is generally not a nonnegvative distribution when $d \geq 4$ and that some evaluations to obtain \eqref{int lim R} and \eqref{int lim n} fail to work (see Remark \ref{Rem why high dimensions difficult} for more details).
The lack of nonnegativity makes the study of the stochastic wave equation in four or higher dimensions more difficult than in three or lower dimensions (see \cite{MR2399293, MR1930613}).

In this paper, we prove the central limit theorems for stochastic wave equations when $d\geq 4$ by modifying the argument above used in \cite{ebina2023central}. 
The key fact used for modification is that the Fourier transform of $G(t)$ is a function dominated by the Fourier transform of a slightly modified Bessel kernel of order one $b_1$, which is a nonnegative integrable function (see Section \ref{subsection Some technical tools} for the definition of $b_1$).
More precisely, we have
\begin{equation}
\label{int FG bound}
    |\F G(t)(\xi)| \lesssim_T (1+|\xi|^2)^{-\frac{1}{2}} = \F b_1 (\xi), \quad t \in [0,T].
\end{equation}
Thus, by first performing the Fourier transform and then applying \eqref{int FG bound}, we can replace $G(t)$ with the nonnegative integrable function $b_1$ provided that we can justify the Fourier inversion.
Once this can be done, the situation becomes the same as in the three-dimensional case, and then the arguments of \cite{ebina2023central} can be applied.
For the role of this replacement argument, see Remarks \ref{Rem why high dimensions difficult} and \ref{Rem rep2}.

In order to state the main results precisely, we first fix some notations.
We define $F_R(t)$ by \eqref{F_R(t)} and set 
\begin{equation*}
    \sigma_R(t) = \sqrt{\mathrm{Var}(F_R(t))} \quad \text{and} \quad \tau_{\beta} = \int_{\mathbb{B}_1^2}|x-y|^{-\beta}dxdy.
\end{equation*}
Note that $\tau_{\beta}$ is finite since $\beta < 2 < d$.
Let $\mathscr{H} \coloneqq \{h \colon \mathbb{R} \to \mathbb{R} \; | \; \lVert h\rVert_{\textrm{Lip}} \leq 1\}$, where 
\begin{equation*}
    \lVert h \rVert_{\mathrm{Lip}} = \sup_{\substack{x \neq y \\ x,y \in \mathbb{R}}} \frac{|h(x)-h(y)|}{|x-y|}.
\end{equation*}
The Wasserstein distance $d_{\mathrm{W}}$, which is also called 1-Wasserstein distance, between the laws of two integrable real-valued random variables $X$ and $Y$ is given by
\begin{equation}
\label{Wasserstein distance def}
    d_{\mathrm{W}}(X,Y) = \sup_{h \in \mathscr{H}}|\mathbb{E}[h(X)]-\mathbb{E}[h(Y)]|.
\end{equation}
It is known (see \textit{e.g.} \cite[Appendix C]{nourdin_peccati_2012}) that the convergence in $d_{\mathrm{W}}$ implies the convergence in distribution.
As before, the Wasserstein distance between the law of $X$ and the standard normal distribution is denoted by $d_{\mathrm{W}}(X,\mathcal{N}(0,1))$.

The following theorems provide the first results of the central limit theorem and its functional version for stochastic wave equations in the case $d \geq 4$.

\begin{Thm}
\label{Thm m1}
Assume that $\sigma(1) \neq 0$ and $\gamma(x) = |x|^{-\beta}$ for some $0 < \beta < 2$.
For any $t \in [0,T]$, it holds that
\begin{equation*}
    \lim_{R \to \infty} R^{\beta-2d}\sigma_R^2(t) = \tau_{\beta}\int_0^t(t-s)^2\E[\sigma(U(s,0))^2]ds.
\end{equation*}
If $t \in (0,T]$, then we have
\begin{equation}
\label{Thm m1 conv}
    \lim_{R \to \infty} d_{\mathrm{W}}\left(\frac{F_R(t)}{\sigma_R(t)}, \mathcal{N}(0,1)\right) = 0.
\end{equation}
\end{Thm}

\begin{Thm}
\label{Thm m2}
Assume that $\sigma(1) \neq 0$ and $\gamma(x) = |x|^{-\beta}$ for some $0 < \beta < 2$.
The process $\{R^{\frac{\beta}{2} - d}F_R(t) \mid t \in [0,T]\}$ converges in distribution in $C([0,T])$, and the limit is a centered Gaussian process $\{\mathcal{G}(t) \mid t \in [0,T]\}$ with covariance function
\begin{equation*}
    \E[\mathcal{G}(t)\mathcal{G}(s)] = \tau_{\beta}\int_0^{t \land s} (t-r)(s-r)\E[\sigma(U(r,0))]^2 dr,
\end{equation*}
where $t \land s \coloneqq \min \{t,s\}$.
\end{Thm}

\begin{Rem}
\begin{enumerate}[wide, labelindent=0pt]
    \item[(1)] Assumption $\sigma(1) \neq 0$ excludes the trivial case where $U(t,x) \equiv 1$ for every $(t,x) \in [0,T] \times \R^d$ and, consequently, $F_R(t) \equiv 0$.
    \item[(2)] The function $\gamma(x) = |x|^{-\beta}$ with $0< \beta < d$, called the Riesz kernel, satisfies \eqref{suff ergo}. The restriction $0 < \beta < 2$ comes from \eqref{Dalang's condition gamma}.
    \item[(3)] The convergence \eqref{Thm m1 conv} is also valid when $d_{\mathrm{W}}$ is replaced by the Kolmogorov distance $d_{\mathrm{Kol}}$ or the Fortet-Mourier distance $d_{\mathrm{FM}}$.
    Indeed, $d_{\mathrm{FM}}$ is dominated by $d_{\mathrm{W}}$ and it holds that
    \begin{equation*}
       d_{\mathrm{Kol}}(X,\mathcal{N}(0,1)) \leq 2 \sqrt{d_{\mathrm{W}}(X,\mathcal{N}(0,1))}
    \end{equation*}
    for any integrable real-valued random variables $X$ (\textit{cf.} \cite[Appendix C]{nourdin_peccati_2012}). 
    However, we do not know whether the convergence holds under the total variation distance $d_{\mathrm{TV}}$ as in the previous works \cite{MR4290504, MR4164864, MR4439987}.
    \item[(4)] Although we mainly consider the case $d\geq 4$, the approach in this paper can also be applied to all spatial dimensions $d \geq 1$, and Theorems \ref{Thm m1} and \ref{Thm m2} hold in these cases as well if the spatial correlation of the noise is replaced by $|x|^{-\beta}, \ \  0< \beta < \min\{2,d\}$. 
    Even so, the argument can be made simpler in the case $d\leq 3$ thanks to the nonnegativity of $G(t)$.
\end{enumerate}
\end{Rem}

\noindent
\textbf{The organization of the paper.}
The rest of the paper is organized as follows. 
Below we introduce notations and function spaces used throughout the paper.
In Section \ref{section Preliminaries}, we collect some preliminary results and recall the basics of stochastic integrals and the Malliavin calculus used in the paper.
In Section \ref{section Stochastic wave equations}, we revisit the high-dimensional stochastic wave equations and show, in a slightly different way than in \cite{MR2399293}, the existence of a unique solution.
The construction of an approximation sequence to the solution is carried out in Section \ref{section Approximation to the solution}.
The Malliavin differentiability of the sequence and the properties of its derivative are investigated in Section \ref{section Malliavin derivative of the approximation sequence}.
In Section \ref{section Limit of the covariance functions}, we determine the limits of the normalized covariance functions of slightly modified $F_R(t)$ and the spatial integral of the sequence.
Finally, Section \ref{section Proof of Theorem m1} and Section \ref{section Proof of Theorem m2} are devoted to the proofs of Theorem \ref{Thm m1} and Theorem \ref{Thm m2}, respectively.

\vskip\baselineskip
\noindent
\textbf{Notations and function spaces.}
For $x, y \in \R^d$, $|x|$ and $x \cdot y$ denote the usual Euclidean norm and inner product on $\R^d$, respectively. 
For $z \in \C$, $\overline{z}$ is its complex conjugate.
The notations $\langle x \rangle \coloneqq \sqrt{1+|x|^2}$, $a \land b \coloneqq \min\{a,b\}$, and $a \lor b \coloneqq \max\{a,b\}$ for $x \in \R^d$ and $a,b\in \R$ are also used.
When we write $a \lesssim b$, it means that $a \leq Cb$ for some constant $C > 0$.
Sometimes we write $a \lesssim_{Q} b$ for some quantity $Q$ to emphasize that the constant $C$ depends on $Q$.

Set $\mathbb{B}_r(x) \coloneqq \{y \in \R^d \mid |y-x| \leq r\}$ and $\mathbb{B}_r \coloneqq \mathbb{B}_r(0)$.
Let $\mathcal{B}([0,T])$ and $\mathcal{B}(\R^d)$ denote the usual Borel $\sigma$-algebra, and we write $\mathcal{B}_{\mathrm{b}}(\R^d)$ for the set of all bounded Borel sets of $\R^d$.

Let $\F$ be the Fourier transform operator on $\mathbb{R}^d$. 
The Fourier transform of an integrable function $f \colon \R^d \to \C$ is defined by 
\begin{equation*}
    \F f(\xi) = \int_{\mathbb{R}^d}e^{-2\pi \sqrt{-1} \xi \cdot x} f(x)dx.
\end{equation*}

We usually consider real-valued functions and real vector spaces in this paper, but complex-valued functions and complex vector spaces are sometimes needed due to the Fourier transform. In such cases, we use $\C$ as a subscript.
Let $C_{\mathrm{c}}^{\infty}(\mathbb{R}^d)$ and $\mathcal{S}(\mathbb{R}^d)$ be the real-valued smooth functions on $\R^d$ with compact support and the Schwartz space of real-valued rapidly decreasing functions on $\R^d$, respectively.
The Schwartz space of complex-valued rapidly decreasing functions on $\R^d$ is denoted by $\mathcal{S}_{\C}(\R^d)$, and the complex vector space of tempered distributions on $\R^d$ is denoted by $\mathcal{S}_{\C}'(\R^d)$.

Let $(\Omega, \mathscr{F}, P)$ be a complete probability space on which the Gaussian noise is defined, and let $\E$ be the expectation operator.
For $p \in [1,\infty)$ and a Hilbert space $H$, we write $L^p(\Omega)$ (resp. $L^p(\Omega ; H)$) for the usual Lebesgue space (resp. Bochner space) of $p$-integrable real-valued (resp. $H$-valued) random variables on $(\Omega, \mathscr{F}, P)$. 
The $L^p(\Omega)$-norm of a real-valued random variable $X$ is always denoted by $\lVert X \rVert_p \coloneqq \mathbb{E}[|X|^p]^{\frac{1}{p}}$.

\vskip\baselineskip
\noindent
\textbf{Acknowledgment.}
The author would like to thank his advisor, Professor Seiichiro Kusuoka, for his encouragement and several valuable comments on an earlier version of this work.
He thanks the anonymous referees for their valuable comments and suggestions, particularly for pointing out the abstract ergodicity result in \cite{MR4563698}, which significantly enhanced the manuscript.
This work was supported by the Japan Society for the Promotion of Science (JSPS), KAKENHI Grant Number JP22J21604.

\section{Preliminaries}
\label{section Preliminaries}
We assume $d \geq 1$ throughout Section \ref{section Preliminaries}.
Recall that the spatial correlation of the noise $\gamma$ is given by a nonnegative function on $\mathbb{R}^d$ such that $\gamma(x)dx$ is a nonnegative definite tempered measure on $\R^d$ and satisfies \eqref{Dalang's condition gamma}.
The nonnegative definite tempered measure means that 
\begin{equation*}
    \int_{\R^d}(\varphi \ast \varphi^{\natural})(x)\gamma(x)dx \geq 0 \quad \text{for every $\varphi \in \mathcal{S}_{\mathbb{C}}(\R^d)$}, 
\end{equation*}
and that there exists $k > 0$ such that 
\begin{equation}
\label{gamma tempered cond}
    \int_{\R^d}\langle x \rangle^{-k}\gamma(x)dx < \infty.
\end{equation}
Here $\ast$ denotes the convolution and $\varphi^{\natural}(x) \coloneqq \overline{\varphi(-x)}$.
These conditions imply that $\gamma(-x) = \gamma(x)$ and that $\gamma$ is locally integrable.
By the Bochner-Schwartz theorem (see \cite[Chapitre VII, Th\'{e}or\`{e}me XVIII]{Schwartz}), there exists a nonnegative tempered measure $\mu$, which is called the spectral measure of $\gamma$, such that $\gamma  = \mathcal{F}\mu$ in $\mathcal{S}_{\mathbb{C}}'(\mathbb{R}^d)$.
Note that since $\gamma(x) = \gamma(-x)$, $\mu$ is also symmetric.
When $d \geq 3$, the condition \eqref{Dalang's condition gamma} is known to equivalent to Dalang's condition:
\begin{equation}
\label{Dalang's condition}
    \int_{\R^d}\langle x \rangle^{-2}\mu(dx) < \infty.
\end{equation}
For details on this equivalence, see \textit{e.g.} \cite[Theorems 1.1 and 1.2]{KarczewskaZabczyk}.

\begin{Rem}
The Riesz kernel $\gamma(x) = |x|^{-\beta}, \ \  0< \beta < d$ is an example that satisfies the above conditions when $0 < \beta < 2 \land d$.
Its spectral measure is given by $\mu(dx) = c_{\beta}|x|^{\beta-d}dx$, where $c_{\beta}$ is a constant depending on $\beta$.
Note that this $\mu$ satisfies Dalang's condition \eqref{Dalang's condition} if and only if $0 < \beta < 2 \land d$. 
\end{Rem}

\subsection{Hilbert spaces associated with the noise}
Let $\mathcal{H}$ be the separable real Hilbert space obtained by completing the space $\mathcal{S}(\mathbb{R}^d)$ with respect to the norm $\lVert \cdot \rVert_{\mathcal{H}}$ induced by the real inner product
\begin{equation*}
    \langle \varphi,\psi \rangle_{\mathcal{H}} \coloneqq \int_{\mathbb{R}^{2d}}\varphi(x) \psi(y)\gamma(x-y)dxdy = \int_{\mathbb{R}^d}\mathcal{F}\varphi(\xi)\overline{\mathcal{F}\psi(\xi)}\mu(d\xi), \quad \varphi, \psi \in \mathcal{S}(\mathbb{R}^d).
\end{equation*}
Here two functions $\varphi, \psi \in \mathcal{S}(\mathbb{R}^d)$ are identified if $\lVert \varphi - \psi \rVert_{\mathcal{H}} = 0$.

For our analysis, it is convenient to identify elements of $\mathcal{H}$.
To do this, we introduce $\mathcal{I}$ as the set of functions $f \colon \R^d \to \R$ such that $f$ is $\mathcal{B}(\R^d)/\mathcal{B}(\R)$-measurable and satisfies 
\begin{equation*}
    \int_{\R^{2d}} |f(x)||f(y)|\gamma(x-y)dxdy < \infty.
\end{equation*}
Note that for any $f, g \in \mathcal{I}$, we have
\begin{align}
    &\int_{\R^{2d}}|f(x)||g(y)|\gamma(x-y)dxdy \nonumber\\
    &\leq \left(\int_{\R^{2d}}|f(x)||f(y)|\gamma(x-y)dxdy \right)^{\frac{1}{2}}\left(\int_{\R^{2d}}|g(x)||g(y)|\gamma(x-y)dxdy \right)^{\frac{1}{2}} < \infty. \label{I cs ineq}
\end{align}
Indeed, since
\begin{equation*}
    \int_{\R^{2d}}\varphi(x)\psi(y)\gamma(x-y)dxdy = \int_{\R^d} \mathcal{F}\varphi(\xi)\overline{\mathcal{F}\psi(\xi)} \mu(d\xi)
\end{equation*}
holds for any bounded measurable functions $\varphi, \psi$ having compact supports, approximating $f,g \in \mathcal{I}$ by such functions and applying the Cauchy-Schwarz inequality yield \eqref{I cs ineq}. 
It follows that $\mathcal{I}$ is a real vector space. 
If we define the bilinear form $\langle \cdot, \cdot \rangle_{\mathcal{I}} \colon \mathcal{I} \times \mathcal{I} \to \R$ and the associated norm $\lVert \cdot \rVert_{\mathcal{I}}$ as
\begin{equation*}
    \langle f,g \rangle_{\mathcal{I}} = \int_{\R^{2d}} f(x)g(y)\gamma(x-y)dxdy \quad \text{and} \quad \lVert f \rVert_{\mathcal{I}} = \sqrt{\langle f,f \rangle_{\mathcal{I}}},
\end{equation*}
then $(\mathcal{I}, \langle \cdot, \cdot \rangle_{\mathcal{I}})$ becomes a real pre-Hilbert space, provided that we identify $f,g \in \mathcal{I}$ such that $\lVert f-g \rVert_{\mathcal{I}} = 0$.

\begin{Lem}
\label{Lem I H}
There exists a linear isometry $\iota \colon \mathcal{I} \to \mathcal{H}$.
\end{Lem}

\begin{proof}
It is easily seen that $\mathcal{I}$ contains $\mathcal{S}(\R^d)$ as a dense subset.
Thus, we obtain the desired isometry by extending the linear operator $\iota_0 \colon \mathcal{I} \to \mathcal{H}$ with domain $\mathcal{S}(\R^d)$ given by $\iota_0(f) = f$.
\end{proof}

To simplify notations, we identify $f\in \mathcal{I}$ with $\iota f \in \mathcal{H}$ and treat $\mathcal{I}$ as a subspace of $\mathcal{H}$.
Thus, if $f,g \in \mathcal{I}$, then $f,g \in \mathcal{H}$ and we have
\begin{equation*}
    \langle f,g \rangle_{\mathcal{H}} = \int_{\R^{2d}}f(x)g(y)\gamma(x-y)dxdy.
\end{equation*}

Let $([0,T], \mathcal{L}([0,T]), dt)$ be a Lebesgue measure space, where $dt$ is the usual Lebesgue measure.
On this complete measure space, we next consider the space $\mathcal{H}_T \coloneqq L^2([0,T]; \mathcal{H})$.

\begin{Lem}
\label{Lem H_T}
Let $f \colon [0,T] \times \R^d \to \R$ be a $\mathcal{L}([0,T]) \times \mathcal{B}(\R^d)/\mathcal{B}(\R)$-measurable function such that
\begin{equation*}
    \int_0^T\int_{\R^{2d}}|f(t,x)f(t,y)|\gamma(x-y)dxdydt < \infty.
\end{equation*}
Then $f(t,\cdot) \in \mathcal{H}$ for almost every $t\in [0,T]$, and the almost-everywhere-defined function $f \colon [0,T] \to \mathcal{H}$ belongs to $\mathcal{H}_T$.
If $g$ satisfies the same conditions as $f$, then it holds that
\begin{equation*}
    \langle f, g \rangle_{\mathcal{H}_T} 
    = \int_0^T\int_{\R^{2d}}f(t,x)g(t,y)\gamma(x-y)dxdydt.
\end{equation*}
\end{Lem}

\begin{proof}
By Lemma \ref{Lem I H}, we have for almost every $t\in [0,T]$ that $f(t,\cdot) \in \mathcal{H}$ and 
\begin{equation*}
    \langle f(t,\cdot), \varphi \rangle_{\mathcal{H}} = \int_{\R^{2d}}f(t,x)\varphi(y)\gamma(x-y)dxdy, \quad \varphi \in \mathcal{S}(\R^d).
\end{equation*}
Let $\Tilde{f}\colon [0,T] \to \mathcal{H}$ be any modification of $f$, and we will show that $\Tilde{f} \in \mathcal{H}_T$. 
Fix $\varphi \in \mathcal{S}(\R^d)$.
It holds that for almost every $t$,
\begin{equation}
\label{lb1}
    \langle \Tilde{f}(t,\cdot), \varphi \rangle_{\mathcal{H}} = \int_{\R^{2d}}f(t,x)\varphi(y)\gamma(x-y)dxdy.
\end{equation}
Because the right-hand side of \eqref{lb1} is $\mathcal{L}([0,T])$-measurable, the function $\langle \Tilde{f}(t,\cdot), \varphi \rangle_{\mathcal{H}}$ is also $\mathcal{L}([0,T])$-measurable, thanks to the completeness of the measure space.
Taking into account that $\mathcal{S}(\R^d)$ is dense in $\mathcal{H}$, we conclude from the Pettis measurability theorem (\textit{cf.} \cite[Theorem 1.1.6]{MR3617205}) that $\Tilde{f}\colon [0,T] \to \mathcal{H}$ is strongly measurable, where $\mathcal{H}$ is endowed with the Borel $\sigma$-algebra.
Moreover, by Lemma \ref{Lem I H}, we have
\begin{equation*}
    \int_0^T \lVert \Tilde{f}(t) \rVert^2_{\mathcal{H}}dt = \int_0^T\int_{\R^{2d}} f(t,x)f(t,y)\gamma(x-y)dxdydt < \infty.
\end{equation*}
This gives $\tilde{f} \in \mathcal{H}_T$, which also implies $f \in \mathcal{H}_T$.  
The last claim follows from a similar argument.
\end{proof}

The following result is also known. See \textit{e.g.} \cite[Lemma 2.4]{DalangQuel} for the proof.

\begin{Lem}
\label{Lem dense subspace in H_T}
$C_{\mathrm{c}}^{\infty}([0,T] \times\R^d)$ is dense in $(\mathcal{H}_T, \lVert \cdot \rVert_{\mathcal{H}_T})$.
\end{Lem}

Recall that we are considering the complete probability space $(\Omega, \mathscr{F},P)$.
The following result will be needed in Section \ref{section Malliavin derivative of the approximation sequence}.

\begin{Prop}
\label{Prop L^2(Omega H_T)}
Let $p \in [2,\infty)$ and let $f\colon [0,T] \times \R^d \times \Omega \to \R$ be a $\mathcal{L}([0,T]) \times \mathcal{B}(\R^d) \times \mathscr{F}/\mathcal{B}(\R)$-measurable function such that
\begin{equation*}
    \E\left[ \int_0^T \int_{\R^{2d}}|f(t,x)f(t,y)|\gamma(x-y)dxdydt \right] < \infty.
\end{equation*}
Then $f(\cdot, \star, \omega) \in \mathcal{H}_T$ for $P$-a.s. $\omega \in \Omega$, and the almost-surely-defined function $f\colon \Omega \to \mathcal{H}_T$ belongs to $L^2(\Omega;\mathcal{H}_T)$.
If $g$ satisfies the same conditions as $f$, then it holds that
\begin{equation*}
    \langle f, g \rangle_{L^2(\Omega;\mathcal{H}_T)} 
    = \E\left[\int_0^T\int_{\R^{2d}}f(t,x)g(t,y)\gamma(x-y)dxdydt\right] .
\end{equation*}
If $f$ additionally satisfies 
\begin{equation}
\label{lc1}
    \E\left[ \left(\int_0^T \int_{\R^{2d}}f(t,x)f(t,y)\gamma(x-y)dxdydt \right)^{\frac{p}{2}} \right] < \infty,
\end{equation}
then we have $f \in L^p(\Omega;\mathcal{H}_T)$ and $\lVert f \rVert_{L^p(\Omega; \mathcal{H}_T)}^p$ equals the left-hand side of \eqref{lc1}.
\end{Prop}

\begin{proof}
A similar argument as in the proof of Lemma \ref{Lem H_T}, using Lemmas \ref{Lem H_T} and \ref{Lem dense subspace in H_T} and the Pettis measurability theorem, shows that the almost-surely-defined function $f\colon \Omega \to \mathcal{H}_T$ belongs to $L^2(\Omega; \mathcal{H}_T)$.
The rest of the proof is clear.
\end{proof}

\subsection{Stochastic integrals}
\label{subsection Stochastic integrals}
In this section, we introduce some basic notions of the stochastic integrals used in this paper. 
The reader is referred to \cite{MR1684157, DalangQuel} for more information.

Let $W = \{W(\varphi) \mid \varphi \in C_{\mathrm{c}}^{\infty}([0,T] \times \R^d)\}$ be a centered Gaussian process, with covariance function
\begin{equation*}
    \E[W(\varphi)W(\psi)] = \langle \varphi, \psi \rangle_{\mathcal{H}_T}, 
\end{equation*}
defined on the complete probability space $(\Omega, \mathscr{F},P)$.
We assume that $\mathscr{F}$ is generated by $W$.
Because $\varphi \mapsto W(\varphi)$ is a linear isometry from the dense subspace of $(\mathcal{H}_T, \lVert \cdot \rVert_{\mathcal{H}_T})$ to $L^2(\Omega, \mathscr{F},P)$ by Lemma \ref{Lem dense subspace in H_T}, $W(\varphi)$ can be defined for every $\varphi \in \mathcal{H}_T$ by extending the isometry.

Let $W_t(A) \coloneqq W(\ind_{[0,t]}\ind_A)$ for any $t \geq 0$ and $ A \in \mathcal{B}_{\mathrm{b}}(\R^d)$. 
Note that $\ind_{[0,t]}\ind_A \in \mathcal{H}_T$ since $\ind_A \in \mathcal{I} \subset \mathcal{H}$.
Let $\mathscr{F}_t^0$ denote $\sigma$-field generated by $\{W_s(A) \mid A \in \mathcal{B}_{\mathrm{b}}(\mathbb{R}^d), 0 \leq s \leq t\}$ and $P$-null sets, and define $\mathscr{F}_t \coloneqq \cap_{s > t} \mathscr{F}_s^0$ for $t \in [0,T)$ and $\mathscr{F}_T \coloneqq \mathscr{F}_T^0$.
Then the process $\{W_t(A), \mathscr{F}_t, t \in [0,T], A \in \mathcal{B}_{\mathrm{b}}(\mathbb{R}^d)\}$ is a worthy martingale measure (\textit{cf.} \cite{Walsh}).
The associated covariance measure $Q$ and dominating measure $K$ are given by 
\begin{equation*}
    Q((s,t] \times A \times B) = K((s,t] \times A \times B) = (t-s)\int_{A}\int_{B}\gamma(y-z)dydz, \quad 0\leq s<t \leq T,\ A,B\in\mathcal{B}_{\mathrm{b}}(\R^d).
\end{equation*}

Let $X = \{ X(t,x) \mid (t,x) \in [0,T] \times \mathbb{R}^d \}$ be a stochastic process on $(\Omega,\mathscr{F},P)$. 
We say that $X$ is $(\mathscr{F}_t)$-adapted if $X(t,x)$ is $\mathscr{F}_t$-measurable for every $(t,x) \in [0,T] \times \mathbb{R}^d$, and that $X$ is stochastically continuous (resp. $L^2(\Omega)$-continuous) if it is continuous in probability (resp. continuous in $L^2(\Omega)$) at any point $(t,x) \in [0,T] \times \mathbb{R}^d$.
We also say that $X$ is measurable if it is measurable with respect to $\mathcal{B}([0,T]) \times \mathcal{B}(\R^d) \times \mathscr{F}$.
A process $X$ is called elementary with respect to the filtration $(\mathscr{F}_t)$ if it has the form 
\begin{equation}
\label{elementary}
    X(t,x,\omega) = Y(\omega)\ind_{(a,b]}(t)\ind_{A}(x),
\end{equation}
where $Y$ is a bounded $\mathscr{F}_a$-measurable random variable, $0 \leq a < b \leq T$, and $A \in \mathcal{B}_{b}(\R^d)$.
A finite sum of elementary processes is called simple, and the set of simple processes is denoted by $\mathfrak{S}$.
The predictable $\sigma$-field $\mathcal{P}$ on $[0,T] \times \mathbb{R}^d \times \Omega$ is the $\sigma$-field generated by $\mathfrak{S}$.
We say that $X$, which is not necessarily simple, is predictable if it is $\mathcal{P}$-measurable.
According to \cite[Proposition B.1]{MR4017124}, a process $X$ has a predictable modification if $X$ is $(\mathscr{F}_t)$-adapted and stochastically continuous.

Let $\mathcal{P}_{+}$ denote the set of all predictable process $X$ such that
\begin{equation*}
    \norm{X}_{+} \coloneqq \E\left[\int_{0}^T\int_{\R^{2d}}|X(t,x)X(t,y)|\gamma(x-y)dxdydt\right]^{\frac{1}{2}} < \infty. 
\end{equation*}
Then $\norm{\cdot}_{+}$ defines the norm on $\mathcal{P}_{+}$ if we identify two functions $f,g \in \mathcal{P}_{+}$ such that $\norm{f-g}_{+} = 0$.
It is known (\textit{cf.} \cite[chapter 2]{Walsh}) that $(\mathcal{P}_{+}, \norm{\cdot}_{+})$ is complete and $\mathfrak{S}$ is dense in $\mathcal{P}_{+}$.
In \cite{Walsh}, Walsh defines the stochastic integral $X\cdot W = \{(X \cdot W)_t \mid t \in [0,T]\}$ for any $X \in \mathcal{P}_{+}$ by first defining the integral for simple processes and then making an approximation in $L^2(\Omega)$ using 
\begin{align*}
    \E[(X\cdot W)_T^2] = \E\left[\int_{0}^T\int_{\R^{2d}}X(t,x)X(t,y)\gamma(x-y)dxdydt\right] \leq \lVert X \rVert_{+}^2
\end{align*}
and the denseness of $\mathfrak{S}$.
We will call this integral $X \cdot W$ the Walsh integral of $X$ (with respect to the martingale measure $W$) and often use the notation
\begin{equation}
\label{Walsh integral notation}
    (X\cdot W)_t = \int_0^t\int_{\R^d}X(s,y)W(ds,dy).
\end{equation}
The Walsh integral $\{(X \cdot W)_t \mid t \in [0,T]\}$ is continuous square-integrable $(\mathscr{F}_t)$-martingale with quadratic variation
\begin{equation*}
    \int_0^t\int_{\R^{2d}}X(s,y)X(s,z)\gamma(y-z)dydzds, \quad 0\leq t \leq T.
\end{equation*}

Define for any $X,Y \in \mathfrak{S}$,
\begin{align*}
    \langle X,Y \rangle_0 = \E\left[\int_{0}^T\int_{\R^{2d}}X(t,x)Y(t,y)\gamma(x-y)dxdydt\right] \quad \text{and} \quad \lVert X \rVert_{0} = \sqrt{\langle X,X \rangle_0}.
\end{align*}
By identifying simple processes $X$ and $Y$ such that $\lVert X-Y\rVert_{0} = 0$, $(\mathfrak{S}, \langle \cdot, \cdot \rangle_0)$ becomes a pre-Hilbert space and we write its completion as $\mathcal{P}_0$.
Clearly, we have $\mathfrak{S} \subset \mathcal{P}_+ \subset \mathcal{P}_{0}$ since $\lVert X \rVert_{0} \leq \lVert X \rVert_+ < \infty$ for $X \in \mathfrak{S}$.
Observe that this time we have
\begin{align*}
    \E[(X\cdot W)_T^2] = \E\left[\int_{0}^T\int_{\R^{2d}}X(t,x)X(t,y)\gamma(x-y)dxdydt\right] = \lVert X \rVert_0^2
\end{align*}
for $X \in \mathfrak{S}$.
In \cite{MR1684157}, Dalang extends the Walsh integral and defines the stochastic integral $X \cdot W$ for any $X \in \mathcal{P}_0$ by extending this isometry. 
We will call this integral the Dalang integral of $X \in \mathcal{P}_0$ (with respect to the martingale measure $W$), and by abuse of notation, we also write it as \eqref{Walsh integral notation}.

\begin{Rem}
\begin{enumerate}[wide, labelindent=0pt] 
    \item[(1)] If $X \in \mathcal{P}_{+} \subset \mathcal{P}_0$, then the Dalang integral of $X$ agrees with its Walsh integral.
    \item[(2)] Because $\mathfrak{S} \subset L^2(\Omega;\mathcal{H}_T)$ and two norms, $\lVert \cdot \rVert_0$ and $\lVert \cdot \rVert_{L^2(\Omega;\mathcal{H}_T)}$, are the same on $\mathfrak{S}$, we can naturally identify $\mathcal{P}_0$ with the closed subspace of $L^2(\Omega;\mathcal{H}_T)$.
\end{enumerate}
\end{Rem}

We now consider more general martingale measures than $W$.
Let $Z = \{Z(t,x) \mid (t,x) \in [0,T]\times \mathbb{R}^d\}$ be a predictable process such that
\begin{equation}
\label{Z uniform L2 bound}
    \sup_{(t,x) \in [0,T] \times \R^d} \lVert Z(t,x) \rVert_2 < \infty.
\end{equation}
Then $W^Z_t$ defined by
\begin{equation*}
    W^Z_t(A) \coloneqq \int_0^t\int_{A} Z(s,y)W(ds,dy), \quad t \in [0,T], A \in \mathcal{B}_{\mathrm{b}}(\R^d)
\end{equation*}
is again the worthy martingale measure, and its covariance measure $Q_Z$ and dominating measure $K_Z$ are given by
\begin{align*}
    Q_Z((s,t] \times A \times B) &= \E\left[\int_s^t\int_{A}\int_{B} Z(r,y)Z(r,z)\gamma(y-z)dydzdr\right],\\ 
    K_Z((s,t] \times A \times B) &= \E\left[\int_s^t\int_{A}\int_{B} |Z(r,y)Z(r,z)|\gamma(y-z)dydzdr\right],
\end{align*}
for $0 \leq s < t\leq T$ and $A,B \in \mathcal{B}_{\mathrm{b}}(\R^d)$.
Similar to $\mathcal{P}_+$ and $\mathcal{P}_0$ for the martingale measure $W$, we can define classes $\mathcal{P}_{+,Z}$ and $\mathcal{P}_{0,Z}$ for $W^Z$ using the norms $\lVert \cdot \rVert_{+,Z}$ and $\lVert \cdot \rVert_{0,Z}$ defined by
\begin{align*}
    \lVert X \rVert_{+,Z} &= \E\left[\int_{0}^T\int_{\R^{2d}}|X(t,x)X(t,y)||Z(t,x)Z(t,y)|\gamma(x-y)dxdydt\right]^{\frac{1}{2}},\\
    \lVert X \rVert_{0,Z} &= \E\left[\int_{0}^T\int_{\R^{2d}}X(t,x)X(t,y)Z(t,x)Z(t,y)\gamma(x-y)dxdydt\right]^{\frac{1}{2}}.
\end{align*}
As before, we have $\mathcal{P}_{+,Z} \subset \mathcal{P}_{0,Z}$.
Walsh integrals and Dalang integrals with respect to $W^Z$ are then defined for any $X \in \mathcal{P}_{+,Z}$ and $X \in \mathcal{P}_{0,Z}$, respectively.
For these integrals, we will use the following notations:
\begin{equation*}
    (X \cdot W^Z)_t  
    = \int_0^t \int_{\R^d}X(s,y)W^Z(ds,dy).
\end{equation*}

\begin{Rem}
\phantomsection
\label{Rem 2.2}
\begin{enumerate}[wide, labelindent=0pt]
    \item [(1)] If $f(t,x)$ is a deterministic measurable function on $[0,T] \times \R^d$ such that $\lVert f \rVert_+ < \infty$, then we have $f \cdot Z \in \mathcal{P}_{+}$ and $f \in \mathcal{P}_{+,Z} \subset \mathcal{P}_{0,Z}$, and it holds that 
    \begin{equation*}
        \int_0^T\int_{\R^d}f(s,y)Z(s,y)W(ds,dy) = \int_0^T\int_{\R^d}f(s,y)W^{Z}(ds,dy).
    \end{equation*}
    Indeed, this equation clearly holds when $f(s,y) = \ind_{(a,b]}(s)\ind_A(y)$, where $A \in \mathcal{B}_{\mathrm{b}}(\R^d)$, and the general case can be shown by approximating $f$ by a deterministic simple process. 
    \item [(2)] Let $Z_1(t,x)$ and $Z_2(t,x)$ be predictable processes satisfying \eqref{Z uniform L2 bound}.
    It is easy to see that
    \begin{equation*}
        X_1 \cdot W^{Z_1} + X_2 \cdot W^{Z_1} = (X_1 + X_2) \cdot W^{Z_1} \quad \text{and} \quad Y \cdot W^{Z_1} + Y \cdot W^{Z_2} = Y \cdot W^{Z_1 + Z_2}
    \end{equation*}
    for any $X_1,X_2 \in \mathcal{P}_{0,Z_1}$ and $Y \in \mathcal{P}_{+,Z_1} \cap \mathcal{P}_{+,Z_2}$.
\end{enumerate}
\end{Rem}

Now we additionally assume that $Z$ has a spatially homogeneous covariance, that is, for every $t \in [0,T]$,
\begin{equation}
\label{Z homogeneous cov}
    \E[Z(t,x)Z(t,y)] = \E[Z(t,x-y)Z(t,0)] \quad \text{for any $x,y \in \R^d$}.
\end{equation}
For such $Z$, set 
\begin{equation}
\label{gamma^Z}
    \gamma^{Z}(t,x) = \gamma(x)\E[Z(t,x)Z(t,0)], \quad t \in [0,T], \ x \in \R^d.
\end{equation}
The following result is used in \cite{MR1684157} to show that certain tempered distributions can be identified with the element of $\mathcal{P}_{0,Z}$.

\begin{Lem}
\label{Lem fourier bochner schwartz}
Let $\{Z(t,x) \mid (t,x) \in [0,T]\times \mathbb{R}^d\}$ be a predictable process satisfying \eqref{Z uniform L2 bound} and \eqref{Z homogeneous cov}.
For every $t \in [0,T]$, $\gamma^Z(t, \cdot)$ is a nonnegative definite tempered distribution, and there is a nonnegative tempered measure $\mu_t^Z$ such that 
\begin{equation}
    \gamma^Z(t, \cdot) = \F \mu_t^Z \quad \text{in $\mathcal{S}_{\C}'(\R^d)$.} \label{mu_t^Z def}
\end{equation}
\end{Lem}

\begin{proof}
We have $\langle \cdot \rangle^{-k}\gamma^Z(t,\cdot) \in L^1(\R^d) \subset \mathcal{S}_{\C}'(\R^d)$ for some $k > 0$ by \eqref{gamma tempered cond} and \eqref{Z uniform L2 bound}, and hence $\gamma^Z(t,\cdot) \in \mathcal{S}_{\C}'(\R^d)$.
Now we show the nonnegative definiteness.
Since $\gamma^Z(t,x) = \gamma^Z(t,-x)$, we only need to show that 
\begin{equation*}
    \int_{\R^{2d}}\varphi(x)\varphi(y)\gamma^Z(t,x-y)dxdy \geq 0
\end{equation*}
for any $\varphi \in \mathcal{S}(\mathbb{R}^d)$.
If we let $Z_N(t,x,\omega) = Z(t,x, \omega) \ind_{\{|Z|\leq N\}}(t,x,\omega)\mathbf{1}_{\mathbb{B}_N}(x)$, then $x \mapsto \varphi(x)Z_N(t,x,\omega)$ belongs to $\mathcal{I}$ and hence to $\mathcal{H}$.
Using the dominated convergence theorem and Fubini's theorem, we have
\begin{align*}
    \int_{\R^{2d}}\varphi(x)\varphi(y)\gamma^Z(t,x-y)dxdy
    &= \lim_{N \to \infty}\int_{\R^{2d}}\varphi(x)\varphi(y)\E[Z_N(t,x)Z_N(t,y)]\gamma(x-y)dxdy\\
    &= \lim_{N \to \infty}\E\left[\int_{\R^{2d}}\varphi(x)Z_N(t,x)\varphi(y)Z_N(t,y)\gamma(x-y)dxdy\right]\\
    &= \lim_{N \to \infty}\E[\lVert \varphi(\cdot)Z_N(t,\cdot) \rVert^2_{\mathcal{H}}] \geq 0.
\end{align*}
The last assertion of the lemma follows from the Bochner-Schwartz theorem.
\end{proof}

\subsection{Malliavin calculus and Malliavin-Stein bound}
\label{subsection Malliavin calculus and Malliavin-Stein method}
Here we recall the basic facts of the Malliavin calculus based on the Gaussian process $\{W(\varphi) \mid \varphi \in \mathcal{H}_T\}$ defined in Section \ref{subsection Stochastic integrals}.
The results of Malliavin-Stein's method, which are necessary to prove the main results, are also presented.
For more details, the reader is referred to \cite{nourdin_peccati_2012, MR2200233}.

Let $C_{\mathrm{p}}^{\infty}(\mathbb{R}^m)$ denote the set of all smooth functions $f\colon \R^m \to \R$ such that $f$ and all its derivatives are at most of the polynomial growth. 
We say that $F\colon \Omega \to \R$ is a smooth random variable if $F$ is of the form $F=f(W(\varphi_1),\ldots W(\varphi_m))$ with $m\geq 1$, $f \in C_{\mathrm{p}}^{\infty}(\mathbb{R}^m)$, and $\varphi_1, \ldots, \varphi_m \in \mathcal{H}_T$.
The set of all smooth random variables is denoted by $\mathscr{S}$.
For a smooth random variable $F$ of the form above, its Malliavin derivative is the $\mathcal{H}_T$-valued random variable defined by
\begin{equation*}
    DF = \sum_{i=1}^m\frac{\partial f}{\partial x_i}(W(\varphi_1),\ldots, W(\varphi_m))\varphi_i.
\end{equation*}
Let $p\in[1,\infty)$ and let $\mathbb{D}^{1,p}$ denote the closure of $\mathscr{S}$ in $L^p(\Omega)$ with respect to the norm 
\begin{equation*}
    \lVert F \rVert_{1,p} \coloneqq ( \mathbb{E}[|F|^p] + \mathbb{E}[\lVert DF \rVert_{\mathcal{H}_T}^p] )^{\frac{1}{p}}.
\end{equation*}
It is known that the space $\mathscr{S}$ is dense in $L^p(\Omega)$ and that the Malliavin derivative operator $D\colon L^p(\Omega) \to L^p(\Omega;\mathcal{H}_T)$ with domain $\mathscr{S}$ is closable.
We will abuse the notation and write the closure of $D$ as $D$ again.
The domain of the closure $D$ in $L^p(\Omega)$ is then given by the Banach space $\mathbb{D}^{1,p}$.
Also, the operator $D$ is known to satisfy the following chain rule: 
If $F \in \mathbb{D}^{1,p}$ and $\psi \colon \R \to \R$ is a continuously differentiable function with bounded derivative, then $\psi(F) \in \mathbb{D}^{1,p}$ and 
\begin{equation*}
    D(\psi(F)) = \psi'(F)DF.
\end{equation*}

The divergence operator $\delta$ is defined as the adjoint operator of $D\colon \mathbb{D}^{1,2} \to L^2(\Omega; \mathcal{H}_T)$, and let $\mathrm{Dom}(\delta)$ be the domain of $\delta$.
The operators $D$ and $\delta$ satisfy the duality relation: 
\begin{equation}
\label{duality relation}
    \mathbb{E}[\delta(u)F] = \mathbb{E}[\langle u, DF \rangle_{\mathcal{H}_T}] \qquad \text{for any $F \in \mathbb{D}^{1,2}$ and $u \in \mathrm{Dom}(\delta)$}.
\end{equation}
Note that $\delta$ is a closed operator and that this relation and $D1 = 0$ imply that $\mathbb{E}[\delta(u)] = 0$ for all $u \in \mathrm{Dom}(\delta)$.
The operator $\delta$ is known to coincide with the Dalang integral (and hence with the Walsh integral) stated in Section \ref{subsection Stochastic integrals}.

\begin{Lem}
\label{Lem Skorokhod Walsh}
For any $X \in \mathcal{P}_0$, we have $X \in \mathrm{Dom}(\delta) \subset L^2(\Omega;\mathcal{H}_T)$ and 
\begin{equation*}
    \delta(X) = \int_0^T\int_{\R^d}X(s,y)W(ds,dy).
\end{equation*}
\end{Lem}

\begin{proof}
If $X$ is an elementary process of the form \eqref{elementary}, then we see from \cite[Lemma 1.3.2]{MR2200233} that $X \in \mathrm{Dom}(\delta)$ and
\begin{equation*}
    \delta(X) = YW(\ind_{(a,b]}\ind_{A}) = \int_0^T\int_{\R^d}X(s,y)W(ds,dy).
\end{equation*}
A simple approximation argument using the closedness of $\delta$ now completes the proof.
\end{proof}

Identifying $\mathcal{P}_0$ with the closed subspace of $L^2(\Omega;\mathcal{H}_T)$, we define $\pi_{\mathcal{P}_0}\colon  L^2(\Omega;\mathcal{H}_T) \to \mathcal{P}_0$ as an orthogonal projection.
The Clark-Ocone formula in the following proposition is of great importance in the Malliavin calculus and will be used in Section \ref{section Malliavin derivative of the approximation sequence} to evaluate some covariance functions.
See \cite[Theorem 6.6]{MR2399277} for the proof (see also \cite[Proposition 6.3]{MR4346664}).

\begin{Prop}
\label{Prop ClarkOcone}
If $F \in \mathbb{D}^{1,2}$ is $\mathscr{F}_T$-measurable, then we have 
\begin{equation*}
    F = \E[F] + (\pi_{\mathcal{P}_0}[DF])\cdot W, \quad \text{a.s.},
\end{equation*}
where the second term on the right-hand side is the Dalang integral of $\pi_{\mathcal{P}_0}[DF] \in \mathcal{P}_0$.
\end{Prop}

\begin{Rem}
The stochastic integration framework used in \cite{MR2399277} differs from that of this paper, but they are actually essentially the same (\textit{cf.} \cite[Section 2]{DalangQuel}).
The space $L^2_{\mathbb{F}}(\Omega; \gamma(L^2(0,T;\mathscr{H}), \R))$, appearing in \cite[Theorem 6.6]{MR2399277} when $p=2$ and $E = \R$, can be identified with $\mathcal{P}_0$ in the setting of this paper (\textit{cf.}  \cite[Collorary 2.2]{MR3078023}).
\end{Rem}

The following two results from Malliavin-Stein's method are powerful tools in proving the central limit theorems.
The first is needed for the proof of Theorem \ref{Thm m1}, and the second is for the proof of Theorem \ref{Thm m2}.
For the proof of these results, see \cite[Propositions 2.2 and 2.3]{MR4167203} and \cite[Theorem 8.2.1]{nualart_nualart_2018}.

\begin{Prop}
\label{Prop Wasserstein bound}
Let $F = \delta(v)$ for some $v \in \mathrm{Dom}(\delta)$. Suppose that $\mathbb{E}[F^2] = 1$ and $F \in \mathbb{D}^{1,2}$. Then we have
\begin{equation*}
     d_{W}(F, \mathcal{N}(0,1)) \leq \sqrt{\frac{2}{\pi}} \sqrt{\mathrm{Var} (\langle DF,v \rangle_{\mathcal{H}_T})}.
\end{equation*}
\end{Prop}

\begin{Prop}
\label{Prop multivariate stein bound}
Fix $m \geq 2$, and let $F = (F_1, \ldots , F_m)$ be a random vector such that for every $i = 1,2,\ldots m$, $F_i = \delta(v_i)$ for some $v_i \in \mathrm{Dom}(\delta)$ and $F_i \in \mathbb{D}^{1,2}$. Let $Z$ be an m-dimensional centered Gaussian vector with covariance matrix $(C_{i,j})_{1 \leq i,j \leq m}$.
Then, for any twice continuously differentiable function $h\colon  \mathbb{R}^m \to \mathbb{R}$ with bounded second partial derivatives, we have
\begin{equation*}
    |\mathbb{E}[h(F)] - \mathbb{E}[h(Z)]| \leq \frac{m}{2}\max_{1 \leq k,l \leq m} \sup_{x \in \mathbb{R}^m} \left| \frac{\partial^2h}{\partial x_k \partial x_l}(x) \right|\sqrt{\sum_{i,j=1}^{m}\mathbb{E}[(C_{i,j}-\langle DF_i, v_j \rangle_{\mathcal{H}_T})^2]} .
\end{equation*}
\end{Prop}

\subsection{Basic estimates on the fundamental solution}
The fundamental solution of the wave operator is well known (\textit{cf.} \cite[Chapter 4]{Mizohata}) and is given by 
\begin{align}
\label{fundamental solution}
    G(t)_x = 
    \begin{cases}
        \displaystyle \frac{1}{2}\ind_{\{|x| < t\}}, &\text{when $d =1$},\\
        \displaystyle (2\pi)^{-\frac{d+1}{2}}\pi\left(\frac{1}{t}\frac{d}{dt}\right)^{\frac{d-3}{2}}\frac{\sigma_t}{t}(dx),  &\text{when $d \geq 3$ and $d$ is odd},\\
        \displaystyle (2\pi)^{-\frac{d}{2}}\left(\frac{1}{t}\frac{\partial}{\partial t}\right)^{\frac{d-2}{2}}(t^2 - |x|^2)^{-\frac{1}{2}}\ind_{\{|x| < t\}}, &\text{when $d$ is even},
    \end{cases}
\end{align}
where $t>0$ and $\sigma_t(dx)$ is the surface measure on $\partial \mathbb{B}_t$. 
For convenience, we let $G(t) \equiv 0$ for $t \leq 0$.
Also, the Fourier transform of $G(t)$ is known to be
\begin{equation}
\label{Fourier transform of G}
    \mathcal{F}G(t)(\xi) = \frac{\sin(2\pi t |\xi|)}{2\pi |\xi|}, \quad t \geq 0.
\end{equation}
From this, the following lemma is easy to check, and we state without the proof.

\begin{Lem}
\label{Lem property of FG}
For every $t,s \in [0,T]$ and $\xi \in \mathbb{R}^d$, the Fourier transform $\mathcal{F}G(t)$ satisfies the following properties:
\begin{enumerate}
    \item [\normalfont(i)] $\displaystyle|\mathcal{F}G(t)(\xi)| \leq t$.
    \item [\normalfont(ii)] $\displaystyle|\mathcal{F}G(t)(\xi)|^2 \leq (1 + 2T^2)\langle \xi \rangle^{-2}$.
    \item [\normalfont(iii)] $\displaystyle|\mathcal{F}G(t)(\xi) - \mathcal{F}G(s)(\xi)| \leq 2\left| \mathcal{F}G\left( \frac{|t-s|}{2} \right) (\xi) \right| \leq |t-s|.$
\end{enumerate}
\end{Lem}

We next introduce the regularization of $G(t)$.
Fix a radial function $\Lambda \in C^{\infty}_0(\mathbb{R}^d)$ such that $\Lambda \geq 0$, $\supp \Lambda \subset \mathbb{B}_1$, and $\int_{\mathbb{R}^d}\Lambda(x)dx = 1$.
Let $(a_n)_{n=1}^{\infty}$ be a fixed monotone increasing positive sequence such that $\sum_{n = 1}^{\infty}\frac{1}{a_n} =1$. For $n \geq 1$, $t \in [0,T]$, and $x \in \R^d$, we define 
\begin{equation}
\label{Lambda G_n definition}
    \Lambda_n(x) = (a_n)^d\Lambda(a_n x) \quad \text{and} \quad G_n(t,x) = (G(t) \ast \Lambda_n)(x). 
\end{equation}
Note that the function $x \mapsto G_n(t,x)$ is nonnegative if $d \leq 3$, but it is generally not in the case $d \geq 4$.

\begin{Lem}
\label{Lem property of G_k}
For all $t \in [0,T]$, we have $G_n(t,\cdot) \in C_0^{\infty}(\R^d)$ and $\supp{G_n(t,\cdot)} \subset \mathbb{B}_{t+1/a_{n}}$. 
Moreover, $G_n(t,x)$ is uniformly continuous on $[0,T] \times \mathbb{R}^d$, and $\Theta(n,T) \coloneqq \sup_{t \in [0,T]}\norm{G_n(t)}_{\infty} < \infty$.
In particular, we have
\begin{equation*} 
    |G_n(t,x)| \leq \Theta(n,T)\ind_{\mathbb{B}_{t+\frac{1}{a_n}}}(x).
\end{equation*}
\end{Lem}

\begin{proof}
Since $G(t)$ is a rapidly decreasing distribution on $\R^d$ and $\supp{G(t)} \subset \mathbb{B}_t$, it follows that $G_n(t) \in \mathcal{S}(\mathbb{R}^d)$ and $\supp{G_n(t)} \subset \mathbb{B}_{t+1/a_{n}}$, and so $G_n(t,\cdot) \in C_0^{\infty}(\R^d)$.
The continuity and the boundedness are obvious when $d=1$. 
Let $d \geq 3$ and $d$ be odd.
From \eqref{fundamental solution}, we have
\begin{equation*}
    (G(t)\ast \Lambda_n)(x) = (2\pi)^{-\frac{d+1}{2}}\pi\left(\frac{1}{r}\frac{\partial}{\partial r}\right)^{\frac{d-3}{2}}\left( r^{d-2}\int_{\partial \mathbb{B}_1}\Lambda_n(x+ry)\sigma_1(dy) \right)\bigg|_{r=t}.
\end{equation*}
It follows that $(G(t)\ast \Lambda_n)(x)$ is a finite sum of products of a positive power of $t$ and the integral of the same form as above with $\Lambda_n(x+ry)$ replaced by its derivative (of some order less than or equal to $\frac{d-3}{2}$) at $r=t$. 
Consequently, it is easy to see that $G_n(t,x)$ is continuous on $[0,T] \times \R^d$.
From this and the fact that $\supp G_n \subset [0,T] \times \mathbb{B}_{T+1}$, the uniform continuity of $G_n$ follows.
Since $\Lambda_n \in C_0^{\infty}(\R^d)$, the integral of derivatives of $\Lambda_n$ with respect to the measure $\sigma_1(dy)$ is bounded uniformly in $t$, which implies that $\Theta(n,T) < \infty$.
The same proof also works when $d$ is even, and the proof is complete.
\end{proof}

Applying the Fourier transform, we have 
\begin{equation}
\label{FG_n}
    |\F G_n(t)(\xi)| = |\F \Lambda_n(\xi)| |\F G(t)(\xi)| \leq |\F G(t)(\xi)|,
\end{equation}
and it follows from Lemma \ref{Lem property of FG} and Dalang's condition \eqref{Dalang's condition} that
\begin{equation*}
    \int_{\R^d}|\F G_n(t)(\xi)|^2 \mu(d\xi)
    \leq (1 + 2T^2) \int_{\R^d}\langle \xi \rangle^{-2}\mu(d\xi) < \infty.
\end{equation*}
Moreover, it is easily seen that
\begin{align}
    &\sup_{(t,x) \in [0,T] \times \R^d}\int_{\R^{2d}}|G_{n}(t,x-y)G_{n}(t,x-z)|\gamma(y-z)dydz \nonumber \\
    &\leq \Theta(n,T)^2\int_{\mathbb{B}_{T+1}^2}\gamma(y-z)dydz < \infty. \label{G_nG_ngamma uniform estimate}
\end{align}

\subsection{Some technical tools}
\label{subsection Some technical tools}
Let ${H}_t$ and ${B_{\alpha}}$ denote the heat kernel and the Bessel kernel of order $\alpha > 0$, respectively:
\begin{align*}
    H_t(x) &= (4\pi t)^{-\frac{d}{2}}e^{-\frac{|x|^2}{4t}},\\
    B_{\alpha}(x) &= \frac{1}{(4\pi)^\frac{\alpha}{2}\Gamma(\frac{\alpha}{2})}\int_0^{\infty}e^{-\frac{\pi|x|^2}{t}} e^{-\frac{t}{4\pi}} t^{\frac{\alpha - d - 2}{2}} dt.
\end{align*}
The Fourier transform of $H_t$ is known to be $\mathcal{F}H_t(\xi) = e^{-4\pi^2t|\xi|^2}$. 
It is also known (see e.g. \cite[Proposition 2 of Chapter V]{singularintegrals}) that $B_{\alpha} \in L^1(\R^d)$ for any $\alpha >0$ and its Fourier transform is given by $\mathcal{F} B_{\alpha}(\xi) = (1 + 4\pi^2|\xi|^2)^{-\frac{\alpha}{2}}$.
For later use, let 
\begin{equation*}
    b_{\alpha}(x) = (2\pi)^dB_{\alpha}(2\pi x).
\end{equation*}
It follows that $b_{\alpha} \in L^1(\R^d)$ and $\F b_{\alpha}(\xi) = \abra{\xi}^{-\alpha}$.

The following lemma, which has already been used in \textit{e.g.} \cite{MR2024344}, is key in later sections.

\begin{Lem}
\label{Lem muZ mu}
Let $\{Z(t,x) \mid (t,x) \in [0,T]\times \mathbb{R}^d\}$ be a predictable process which satisfies \eqref{Z uniform L2 bound} and \eqref{Z homogeneous cov}. 
Let $\mu_t^Z$ be the nonnegative tempered measure satisfying \eqref{mu_t^Z def}, which exists by Lemma \ref{Lem fourier bochner schwartz}.
It holds that
\begin{equation}
    \int_{\R^d}\abra{\xi}^{-2}\mu_t^{Z}(d\xi) 
    \leq \sup_{\eta \in \R^d}\norm{Z(t, \eta)}_2^2\int_{\R^d}\abra{\xi}^{-2}\mu(d\xi). \label{inequality muZ mu}
\end{equation}
\end{Lem}

\begin{proof}
By the monotone convergence theorem, 
\begin{equation}
     \int_{\R^d}\abra{\xi}^{-2}\mu_t^{Z}(d\xi) = \lim_{r \to 0} \int_{\R^d}e^{-4\pi^2 r|\xi|^2}\abra{\xi}^{-2}\mu_t^{Z}(d\xi). \label{a1}
\end{equation}
Since $e^{-4\pi^2 r|\cdot|^2}\abra{\cdot}^{-2} \in \mathcal{S}(\mathbb{R}^d)$ and $e^{-4\pi^2 r|\xi|^2}\abra{\xi}^{-2} = \F (H_r \ast b_{2})(\xi)$, we have
\begin{align}
    \int_{\R^d}e^{-4\pi^2 r|\xi|^2}\abra{\xi}^{-2}\mu_t^{Z}(d\xi)
    &= \int_{\R^d}(H_r\ast b_{2})(x)\gamma(x)\E[Z(t,x)Z(t,0)]dx \nonumber\\
    &\leq \sup_{\eta \in \R^d}\norm{Z(t, \eta)}_2^2\int_{\R^d}(H_r\ast b_{2})(x)\gamma(x)dx. \label{a2}
\end{align}
In the last inequality above, we do not need the absolute value sign in the integral since $H_r\ast b_{2}$ and $\gamma$ are nonnegative. 
Using the Fourier transform again, we have
\begin{equation}
    \int_{\R^d}(H_r\ast b_{2})(x)\gamma(x)dx = \int_{\R^d}e^{-4\pi^2 r|\xi|^2}\abra{\xi}^{-2}\mu(d\xi) \xrightarrow[]{r \to 0} \int_{\R^d}\abra{\xi}^{-2}\mu(d\xi). \label{a3}
\end{equation}
Now \eqref{inequality muZ mu} follows by combining \eqref{a1}, \eqref{a2}, and \eqref{a3}.
\end{proof}

\begin{Rem}\label{Rem why high dimensions difficult}
Lemma \ref{Lem muZ mu} is especially important for analyzing stochastic wave equations in four or higher dimensions in this paper. 
When $d \geq 3$, the corresponding fundamental solution $G(t)$ becomes a distribution (see \eqref{fundamental solution}), which makes it difficult to handle $G(t)$ directly.
A natural approach to address this issue is to approximate $G(t)$ by a sequence of functions.
In this approach, it is essential to obtain uniform estimates for quantities related to the approximation sequence.
Suppose that we want to obtain the estimate for
\begin{align*}
    \mathbf{U} 
    &\coloneqq \sup_{n \in \mathbb{N}} \left \lVert \int_0^T \int_{\R^d}G_n(t,x) W^{Z_n}(dt, dx) \right\rVert_2^2\\
    &= \sup_{n \in \mathbb{N}}\left|\int_0^T\int_{\R^{2d}}G_n(t,y)G_n(t,z)\E[Z_n(t,y)Z_n(t,z)]\gamma(y-z)dydz\right|.
\end{align*}
where, for each $n \in \mathbb{N}$, $\{Z_n(t,x) \mid (t,x) \in [0,T]\times \mathbb{R}^d\}$ is a predictable process satisfying
\begin{equation*}
    \sup_{n\in \mathbb{N}}\sup_{(\tau,\eta) \in [0,T] \times \R^d}\norm{Z_n(\tau, \eta)}_2 < \infty
\end{equation*}
and for every $t \in [0,T]$, 
\begin{equation*}
    \E[Z_n(t,x)Z_n(t,y)] = \E[Z_n(t,x-y)Z_n(t,0)] \qquad \text{for any $x,y \in \R^d$}.
\end{equation*}
A simple estimate (\textit{cf.} \eqref{G_nG_ngamma uniform estimate}) yields 
\begin{align}
    \mathbf{U}
    &\leq  \sup_{n \in \mathbb{N}}\sup_{(\tau,\eta) \in [0,T] \times \R^d}\norm{Z_n(\tau, \eta)}_2^2 \sup_{n \in \mathbb{N}}\int_0^T\int_{\R^{2d}}\left|G_n(t,y)G_n(t,z)\right|\gamma(y-z)dydzdt \label{abso remove}\\
    &\leq \sup_{n \in \mathbb{N}}\sup_{(\tau,\eta) \in [0,T] \times \R^d}\norm{Z_n(\tau, \eta)}_2^2\sup_{n\in \mathbb{N}} \Theta(n,T)^2 T \int_{\mathbb{B}_{T+1}^2}\gamma(y-z)dydz \\
    &= \infty
\end{align}
because $\sup_{n\in \mathbb{N}} \Theta(n,T) = \infty$.
If $G(t)$ is a nonnegative function or nonnegative distribution (\textit{i.e.} when $d \leq 3$), then $G_n$ becomes a nonnegative function, and we can remove the absolute value inside the integral \eqref{abso remove}.
Consequently, using (ii) of Lemma \ref{Lem property of FG}, we can derive that
\begin{align}
    \mathbf{U}
    &\leq  \sup_{n\in \mathbb{N}}\sup_{(\tau,\eta) \in [0,T] \times \R^d}\norm{Z_n(\tau, \eta)}_2^2 \sup_{n\in \mathbb{N}}\int_0^T\int_{\R^{2d}}G_n(t,y)G_n(t,z)\gamma(y-z)dydz dt\\
    &=  \sup_{n\in \mathbb{N}}\sup_{(\tau,\eta) \in [0,T] \times \R^d}\norm{Z_n(\tau, \eta)}_2^2 \sup_{n \in \mathbb{N}}\int_0^T\int_{\R^{d}}|\mathcal{F} G(t)(\xi)|^2|\mathcal{F} \Lambda_n(\xi)|^2\mu(d\xi)dt\\
    &\lesssim_T \sup_{n\in \mathbb{N}}\sup_{(\tau,\eta) \in [0,T] \times \R^d}\norm{Z_n(\tau, \eta)}_2^2 \int_{\R^{d}}\abra{\xi}^{-2}\mu(d\xi) < \infty.
\end{align}
In \cite{ebina2023central}, thanks to the fact that $G(t)$ is a nonnegative distribution (\textit{i.e.}, a measure) when $d = 3$, we can use this type of estimate to obtain uniform moment bounds and uniform convergence results for the approximate sequence $u_n$ (\textit{cf.} Section \ref{section Introduction}), which are crucial for proving the limits \eqref{int lim R} and \eqref{int lim n}. 

However, when $d \geq 4$, we cannot remove the absolute value in \eqref{abso remove} to derive this estimate because $G(t)$ becomes a distribution that is not necessarily nonnegative. 
As a result, simply adopting the argument of the three-dimensional case when $d \geq 4$ fails to establish various uniform estimates, causing most of the reasoning used in \cite{ebina2023central} to break down. 
(Note that as taking absolute value is a nonlinear operation, it is difficult to relate $(\mathcal{F}|G_n(t)|)(\xi)$ and $|\mathcal{F} G_n(t) (\xi)|$.)
Even so, if we first apply the Fourier transform and (ii) of Lemma \ref{Lem property of FG} to obtain
\begin{equation}
    \mathbf{U} 
    = \sup_{n \in \mathbb{N}} \int_0^T \int_{\R^{d}}|\mathcal{F} G(t)(\xi)|^2|\mathcal{F} \Lambda_n(\xi)|^2\mu^{Z_n}_t(d\xi)dt
    \lesssim_T \sup_{n \in \mathbb{N}} \int_0^T\int_{\R^d} \abra{\xi}^{-2}\mu_t^{Z_n}(d\xi)dt,
\end{equation}
then, by applying Lemma \ref{Lem muZ mu}, we can now deduce that
\begin{equation}
    \mathbf{U} 
    \lesssim_T \sup_{n \in \mathbb{N}} \sup_{(\tau,\eta) \in [0,T] \times \R^d}\norm{Z_n(\tau, \eta)}_2^2 \int_{\R^d} \abra{\xi}^{-2}\mu(d\xi) < \infty.
\end{equation}
As mentioned in Section \ref{section Introduction}, the crucial point here is that, by applying inequality \eqref{int FG bound} (or equivallently (ii) of Lemma \ref{Lem property of FG}), we can essentially replace the (not necessarily nonnegative) function $G_n$ with the nonnegative function $b_1$ on the Fourier side.
This replacement enables us to apply Lemma \ref{Lem muZ mu}, which is particularly useful for establishing uniform estimates for various quantities involving $G_n$.
Lemma \ref{Lem muZ mu} will mainly be used in Sections \ref{section Stochastic wave equations} and \ref{section Approximation to the solution} to establish uniform estimates and convergences results for the approximation sequence of solutions to \eqref{SPDE}, which will be essential for proving the main theorems.
\end{Rem}

For completeness, we record two useful tools that will be used in later sections.
The first one is an elementary inequality due to Peetre. 
\begin{Lem}[Peetre's inequality]
\label{Lem Peetre inequality}
For any $k \in \mathbb{R}$ and $x,y \in \mathbb{R}^d$, we have
\begin{equation*}
    \langle x + y \rangle^k \leq 2^{\frac{|k|}{2}}\langle x \rangle^k\langle y \rangle^{|k|}.
\end{equation*}
\end{Lem}

The second one is a slight generalization of the dominated convergence theorem. 
See e.g. \cite[Theorem 2.8.8.]{MR2267655} for the proof.

\begin{Prop}
\label{Prop general DCT}
Let $f$ and $g$ be functions on $\R^d$.
Let $(f_n)_{n=1}^{\infty}, (g_n)_{n=1}^{\infty} \subset L^1(\mathbb{R}^d)$ be two sequences of functions such that for almost every $x \in \mathbb{R}^d$, 
\begin{gather*}
    |f_n(x)| \leq g_n(x), \quad
    \lim_{n \to \infty} f_n(x) = f(x), \quad \lim_{n \to \infty} g_n(x) = g(x).
\end{gather*}
Suppose that $g \in L^1(\mathbb{R}^d)$ and
\begin{equation*}
    \lim_{n \to \infty} \int_{\mathbb{R}^d} g_n(x) dx = \int_{\mathbb{R}^d} g(x) dx.
\end{equation*}
Then $f \in L^1(\mathbb{R}^d)$ and 
\begin{equation*}
    \lim_{n \to \infty} \int_{\mathbb{R}^d} f_n(x) dx = \int_{\mathbb{R}^d} f(x) dx.
\end{equation*}
\end{Prop}

\section{Stochastic wave equations in dimensions \texorpdfstring{$d \geq 4$}{}}
\label{section Stochastic wave equations}
The stochastic wave equation is a fundamental stochastic partial differential equation and has been studied along with the stochastic heat equation for many years.
However, in contrast to the stochastic heat equation, the stochastic wave equation in four or higher dimensions has not been considered much due to technical issues, as mentioned in Introduction.
In \cite{MR1684157}, Dalang extends Walsh's theory of martingale measures and stochastic integrals to include some rapidly decreasing nonnegative tempered distributions as integrands.
This framework of stochastic integrals allows one to study a stochastic wave equation of spatial dimension $d \leq 3$ driven by a spatially homogeneous Gaussian noise.
Unfortunately, one cannot apply this framework for the case $d \geq 4$ since the fundamental solution $G(t)$ is not necessarily nonnegative. 
To study high-dimensional cases, the framework used in \cite{MR1684157} is further extended in \cite{MR2399293}. 
For different approaches, see \textit{e.g.} \cite{MR1961163, MR1930613}.

Let $Z = \{Z(t,x) \mid (t,x) \in [0,T] \times \R^d\}$ be a predictable process satisfying the conditions \eqref{Z uniform L2 bound} and \eqref{Z homogeneous cov}.
Set $g_t(x) = \E[Z(t,x)Z(t,0)]$.
If additionally $x \mapsto Z(t,x)$ is $L^2(\Omega)$-continuous at $x=0$, then we apply Bochner's theorem to see that there is a finite measure $\nu_t^{Z}$ for each $t$ such that
\begin{equation*}
    g_t(x) = \int_{\R^d}e^{-2\pi\sqrt{-1}x\cdot y}\nu_t^{Z}(dy).
\end{equation*}
An extension of stochastic integrals (with respect to the martingale measure $W^Z$) discussed in \cite[Section 3]{MR2399293} is based on the equation
\begin{equation}
\label{conus dalang equation}
    \gamma \cdot g_t = \mathcal{F}\mu \cdot \mathcal{F}\nu_t^{Z} = \mathcal{F}(\mu \ast \nu_t^{Z}).
\end{equation}
Note that the spectral measure $\mu$ of $\gamma$ is generally not finite, so the equation holds in $\mathcal{S}'_{\C}(\R^d)$.
This equation allows $\mu$ and $\nu_t^{Z}$ to be treated separately instead of $\mu_t^{Z}$, which is useful when dealing with tempered distributions that are not necessarily nonnegative.
However, some care is needed because the convolution of two tempered distributions is generally not well-defined, and the convolution theorem for the Fourier transform does not always hold for such convolutions.

In this section, assuming that $d \geq 4$ and $\sigma\colon \R \to \R$ is a Lipschitz function that is not necessarily continuously differentiable, we consider the stochastic wave equation \eqref{SPDE}.
The purpose of this section is to show that \eqref{SPDE} can be handled by using Lemma \ref{Lem muZ mu} and exploiting the nonnegativity of the Bessel kernel instead of using \eqref{conus dalang equation}.
Without using \eqref{conus dalang equation}, we will show that $G$ belongs to $\mathcal{P}_{0,Z}$ and that a unique solution to \eqref{SPDE} exists by following the argument in \cite{MR1684157}.

Recall that the spectral measure $\mu$ of $\gamma$ is assumed to satisfy Dalang's condition \eqref{Dalang's condition} and that $\mu_t^Z$ is the spectral measure of $\gamma^Z(t,\cdot)$ defined by \eqref{gamma^Z}.
Let $Z(t,x)$ be a predictable process satisfying \eqref{Z uniform L2 bound} and \eqref{Z homogeneous cov}.
Let $\overline{\mathcal{P}}^Z$ be the set of deterministic functions $[0,T] \ni t \mapsto S(t) \in \mathcal{S}'(\R^d)$ such that $\F S(t)$ is a function and 
\begin{equation*}
    \lVert S \rVert_{\overline{\mathcal{P}}^Z}^2 \coloneqq \int_0^T\int_{\R^d}|\F S(t)(\xi)|^2\mu_t^Z(d\xi)dt < \infty.
\end{equation*}
Note that $\overline{\mathcal{P}}^Z$ is a normed space if we identify $S_1, S_2 \in \overline{\mathcal{P}}^Z$ satisfying $\norm{S_1 - S_2}_{\overline{\mathcal{P}}^Z} = 0$.
Also, we define $\mathcal{E}_{0,d}$ to be the subset of $\mathcal{P}_{+}$ consisting of deterministic functions $f(t,x)$ such that $f(t,\cdot) \in \mathcal{S}(\mathbb{R}^d)$ for every $t$.
Clearly, we have $\mathcal{E}_{0,d} \subset \overline{\mathcal{P}}^Z$ and $\mathcal{E}_{0,d} \subset \mathcal{P}_{+,Z} \subset \mathcal{P}_{0,Z}$, and hence $\mathcal{E}_{0,d} \subset \overline{\mathcal{P}}^Z \cap \mathcal{P}_{0,Z}$. 
Since $\lVert f \rVert_{\overline{\mathcal{P}}^Z} = \lVert f \rVert_{0,Z}$ for all $f \in \mathcal{E}_{0,d}$, any $S \in \overline{\mathcal{P}}^Z$ that can be approximated by $(f_n) \subset \mathcal{E}_{0,d}$ in $\lVert \cdot \rVert_{\overline{\mathcal{P}}^Z}$ corresponds to an element of $\mathcal{P}_{0,Z}$ (\textit{cf.} \cite[Section 2]{MR1684157}).

\begin{Prop}
\label{Prop extended SI}
Let $\{Z(t,x) \mid (t,x) \in [0,T]\times \mathbb{R}^d\}$ be a predictable process which satisfies \eqref{Z uniform L2 bound} and \eqref{Z homogeneous cov}. 
Then $G$ belongs to $\mathcal{P}_{0,Z}$ and 
\begin{align}
    \E[(G\cdot W^{Z})_t^2] 
    &= \int_0^t\int_{\R^d}|\F G(s)(\xi)|^2\mu_s^{Z}(d\xi)ds \nonumber\\
    &\lesssim_T \int_0^t \sup_{\eta \in \R^d}\norm{Z(s,\eta)}^2_2 ds\int_{\R^d}\abra{\xi}^{-2}\mu(d\xi). \label{FG muZ inequality}
\end{align}
\end{Prop}
\begin{proof}
The proof is a slight modification of the proof of \cite[Theorem 2]{MR1684157}.
We consider $G_n(t,x)$ defined in \eqref{Lambda G_n definition}.
By Lemma \ref{Lem property of G_k}, $G_{n}(t,x)$ is predictable (as it is a measurable deterministic function) and satisfies 
\begin{align*}
    \norm{G_n}_{+,Z}^2
    &=\int_0^T\int_{\R^{2d}}|G_n(t,x)||G_n(t,y)|\gamma(x-y)\E[|Z(t,x)Z(t,y)|]dxdydt\\
    &\leq T \Theta(n,T)^2 \sup_{(\tau, \eta) \in [0,T] \times \R^d}\norm{Z(\tau,\eta)}^2_2 \int_{\mathbb{B}_{T+1}^2}\gamma(x-y)dxdy < \infty.
\end{align*}
Hence $G_n \in \mathcal{P}_{+,Z} \subset \mathcal{P}_{0,Z}$. Moreover, by Lemma \ref{Lem property of G_k}, we have $G_n \in \mathcal{E}_{0,d}$ for every $n$.

To prove $G \in \mathcal{P}_{0,Z}$, it is sufficient to show that $\norm{G_n - G}_{\overline{\mathcal{P}}^Z} \xrightarrow[]{n \to \infty} 0$.
Because $\lim_{n\to \infty}\F \Lambda_n(\xi) = 1$, $|\F \Lambda_n(\xi) - 1|^2 \leq 4$, and 
\begin{align*}
    \int_0^T\int_{\R^d}|\F G(t)(\xi)|^2\mu_t^Z(d\xi)dt
    &\lesssim_T \int_0^T\int_{\R^d}\abra{\xi}^{-2}\mu_t^Z(d\xi)dt \\
    &\lesssim_T\int_{\R^d}\abra{\xi}^{-2}\mu(d\xi) < \infty
\end{align*}
by Lemmas \ref{Lem property of FG} and \ref{Lem muZ mu}, we can apply the dominated convergence theorem to obtain the convergence
\begin{align*} 
    \norm{G_n -G}_{\overline{\mathcal{P}}^Z}^2
    &= \int_0^T\int_{\R^d}|\F \Lambda_n(\xi)-1|^2|\F G(t)(\xi)|^2\mu_t^Z(d\xi)dt \xrightarrow[]{n \to \infty} 0.
\end{align*}
Finally, \eqref{FG muZ inequality} follows by applying Lemmas \ref{Lem property of FG} and \ref{Lem muZ mu} again.
\end{proof}

For any $x \in \R^d$ and $T \in \mathcal{S}'_{\C}(\R^d)$, let $\kappa_x T$ and $T^c$ denote the tempered distributions defined by
\begin{equation*}
    \langle \kappa_x T, \varphi \rangle = \langle T, \varphi(\cdot + x) \rangle \quad \text{and} \quad
    \langle T^c, \varphi \rangle = \langle T, \varphi(-\cdot) \rangle  
\end{equation*} 
for all $\varphi \in \mathcal{S}_{\C}(\R^d)$.
Here $\langle T,\varphi \rangle$ denotes the dual pairing of $T \in \mathcal{S}'_{\C}(\R^d)$ and $\varphi \in \mathcal{S}_{\C}(\R^d)$.

\begin{Cor}
\label{Cor G(t-s,x-y) in P_0Z}
Let $\{Z(t,x) \mid (t,x) \in [0,T]\times \mathbb{R}^d\}$ be a predictable process which satisfies \eqref{Z uniform L2 bound} and \eqref{Z homogeneous cov}. 
Then, for any $ t \in [0,T]$, $s \mapsto (\kappa_xG(t-s))^c$ belongs to $\mathcal{P}_{0,Z}$ and
\begin{align}
    \E[((\kappa_xG(t-\cdot))^c \cdot W^{Z})_r^2] 
    &= \int_0^{t \land r}\int_{\R^d}|\F G(t-s)(\xi)|^2\mu_s^{Z}(d\xi)ds \nonumber \\
    &\lesssim_T \int_0^{t \land r} \sup_{\eta \in \R^d}\norm{Z(s,\eta)}^2_2 ds\int_{\R^d}\abra{\xi}^{-2}\mu(d\xi). \label{cor FG inequality}
\end{align}
\end{Cor}
\begin{proof}
Note that $G(t-s)$ is a rapidly decreasing tempered distribution, and so is $(\kappa_xG(t-s))^c$.
The claim follows from 
\begin{equation*}
    \F ((\kappa_xG(t-s))^c)(\xi) 
    = e^{2\pi \sqrt{-1}\xi \cdot x}\F G(t-s)(\xi) 
\end{equation*}
and the same argument as in Proposition \ref{Prop extended SI}.
\end{proof}

For simplicity, we will use the notation
\begin{equation*}
    ((\kappa_xG(t-\cdot))^c \cdot W^{Z})_r = \int_0^r \int_{\R^d} G(t-s,x-y)W^Z(ds,dy). 
\end{equation*}
Following \cite{MR2399293, MR1684157}, we recall definitions of the property (S) and the solution.
For any set $A \subset \R^d$ and $z \in \R^d$, let $A + z = \{x + z \mid x \in A \}$.

\begin{Def}
We say that a process $\{Z(t,x) \mid (t,x) \in [0,T] \times \R^d\}$ satisfies the property (S) if, for any $z \in \R^d$ and every integers $i,j \geq 1$, all distributions of random vectors of the form 
\begin{gather*}
    (Z(t_1,x_1 + z), \ldots, Z(t_i,x_i + z),W_{t_{i+1}}(A_1 + z), \ldots, W_{t_{i+j}}(A_j + z)),\\
    \text{where $t_1,\ldots,t_{i+j} \in [0,T]$, $x_1, \ldots, x_i \in \R^d$, and $A_1, \ldots, A_j \in \mathcal{B}_{\mathrm{b}}(\R^d)$},
\end{gather*}
do not depend on $z$.
\end{Def}

\begin{Def}
\label{Def solution}
We say that a process $U = \{U(t,x) \mid (t,x) \in [0,T] \times \R^d\}$ is a solution to the equation \eqref{SPDE} if $U$ is predictable and satisfies the property (S) and 
\begin{equation}
\label{U uniform bound}
    \sup_{(t,x) \in [0,T] \times \R^d}\lVert U(t,x) \rVert_2 < \infty,
\end{equation}
and the following equation holds for any $(t,x) \in [0,T] \times \R^d$:
\begin{equation}
\label{solution integral equation}
    U(t,x) = 1 + \int_0^t\int_{\R^d}G(t-s,x-y)W^{\sigma(U)}(ds,dy), \quad \text{a.s.}.
\end{equation}
\end{Def}

\begin{Rem}
Note that $(\kappa_xG(t-s))^c \in \mathcal{P}_{0,\sigma(U)}$ and the stochastic integral in \eqref{solution integral equation} is well-defined by Corollary \ref{Cor G(t-s,x-y) in P_0Z}.
In Definition \ref{Def solution}, we impose the condition that $U(t,x)$ satisfies the property (S), which is stronger than the condition that it has a spatially homogeneous covariance, to ensure the uniqueness of the solution (see \cite[Theorem 4.8]{MR2399293}).
\end{Rem}

With Proposition \ref{Prop extended SI}, we can now show the existence and uniqueness of the solution.

\begin{Prop}
\label{Prop U exi uni conti}
The equation \eqref{SPDE} has a solution $\{U(t,x) \mid (t,x) \in [0,T] \times \R^d\}$ which is $L^2(\Omega)$-continuous.
Moreover, if $\{V(t,x) \mid (t,x) \in [0,T] \times \R^d\}$ is another solution, then $P(U(t,x) = V(t,x)) = 1$ for any $(t,x) \in [0,T] \times \R^d$. 
\end{Prop}

\begin{proof}
The proof is only a slightly modified application of the arguments in \cite[Section 5]{MR1684157} and \cite[Section 4]{MR2399293}, so we only give a brief outline here.
To construct the solution, we consider the following iteration scheme: 
\begin{equation}
    U_0(t,x) = 1, \quad U_{n+1}(t,x) = 1 + \int_0^t\int_{\mathbb{R}^d}G(t-s,x-y)W^{\sigma(U_n)}(ds,dy). \label{U_(n+1)}
\end{equation}
By an induction argument, we can check that for each $n \geq 0$, $U_n$ is $L^2(\Omega)$-continuous and satisfies the property (S) and
\begin{equation}
    \sup_{(t,x) \in [0,T] \times \R^d}\lVert U_n(t,x) \rVert_2 < \infty \label{U_n bound}.
\end{equation}
Indeed, a simple computation using \eqref{cor FG inequality} shows that \eqref{U_n bound} holds, and the property (S) can be shown by the same argument as in the proof of \cite[Lemma 18]{MR1684157}.
See also Lemma \ref{Lem u_n property S} for a similar argument.
The $L^2(\Omega)$-continuity of $U_n$ can also be checked by using Lemma \ref{Lem property of FG} and \eqref{cor FG inequality}, following the proof of Lemma 19 in \cite{MR1825714}.
From there, we can take a predictable modification of $U_n$ which satisfies the same properties as $U_n$. Hence $U_{n+1}$ is well-defined. 
Moreover, a Gronwall-type argument and \eqref{cor FG inequality} yields (see Proposition \ref{Prop u_n uniform bound} for a similar argument) that
\begin{equation}
\label{U_n uniform bound}
    \sup_{n \geq 0}\sup_{(t,x) \in [0,T] \times \R^d}\lVert U_{n}(t,x) \rVert_2 < \infty.
\end{equation}

To prove that the sequence $(U_n(t,x))_{n=0}^{\infty}$ converges in $L^2(\Omega)$, we first have to check that the process $\{\sigma(U_{n+1}(t,x)) - \sigma(U_n(t,x)) \mid (t,x) \in [0,T] \times \R^d \}$ satisfies the property (S).
This can be done by the same argument in \cite[Lemma 4.5]{MR2399293} (see Lemma \ref{Lem V_n property S} for a similar argument).
Then, using the spectral measure $\mu_t^{\sigma(U_{n+1}) - \sigma(U_n)}$ of $\gamma^{\sigma(U_{n+1}) - \sigma(U_n)}(t,\cdot)$ and \eqref{cor FG inequality} and making a Gronwall-type argument, we deduce that 
\begin{equation*}
    \sum_{n = 0}^{\infty}\sup_{(t,x) \in [0,T]\times \R^d}\lVert U_{n+1}(t,x) - U_n(t,x) \rVert_2^2 < \infty,
\end{equation*}
which implies that $(U_n(t,x))_{n=0}^{\infty}$ is a uniformly Cauchy sequence in $L^2(\Omega)$ with respect to $(t,x) \in [0,T]\times \R^d$.
Therefore, there exists $U(t,x) \in L^2(\Omega)$ for each $(t,x)$ such that
\begin{equation}
\label{U_n U convergence}
    \lim_{n \to \infty} \sup_{(t,x) \in [0,T]\times \R^d}\lVert U_n(t,x) - U(t,x) \rVert_2 = 0.
\end{equation}
From \eqref{U_n uniform bound}, it is clear that $U$ satisfies \eqref{U uniform bound}.
Because $U_n$ is $L^2(\Omega)$-continuous and satisfies the property (S), the same holds for $U$. 
Thus $U$ admits a predictable modification, which we will again write as $U$.

To prove that $U(t,x)$ satisfies \eqref{solution integral equation}, it suffices to check that
\begin{equation}
\label{si conv}
    \int_0^t\int_{\R^d}G(t-s,x-y)W^{\sigma(U_n)}(ds,dy) \xrightarrow[n \to \infty]{} \int_0^t\int_{\R^d}G(t-s,x-y)W^{\sigma(U)}(ds,dy) \quad \text{in $L^2(\Omega)$.}
\end{equation}
We can show that the process $\{\sigma(U_n(t,x)) - \sigma(U(t,x)) \mid (t,x) \in [0,T]\times \R^d \}$ satisfies the property (S) and justify 
\begin{align*}
    &\int_0^t\int_{\R^d}G(t-s,x-y)W^{\sigma(U_n)}(ds,dy) - \int_0^t\int_{\R^d}G(t-s,x-y)W^{\sigma(U)}(ds,dy)\\
    &= \int_0^t\int_{\R^d}G(t-s,x-y)W^{\sigma(U_n) - \sigma(U)}(ds,dy).
\end{align*}
The convergence \eqref{si conv} then follows from \eqref{cor FG inequality} and \eqref{U_n U convergence}.
Finally, the uniqueness follows from the same argument as in \cite[Theorem 4.8]{MR2399293}. 
\end{proof}

\section{Approximation to the solution}
\label{section Approximation to the solution}
In this section, we continue to consider the equation \eqref{SPDE} with the assumption that $\sigma$ is a Lipschitz function but not necessarily continuously differentiable. 
In order to prove the central limit theorems, we construct another approximation sequence given (informally) by
\begin{align}
    u_0(t,x) &= 1, \nonumber \\ 
    u_{n+1}(t,x) &= 1 + \int_0^t\int_{\mathbb{R}^d}G_{n+1}(t-s,x-y)\sigma(u_n(s,y))W(ds,dy). \label{u_(n+1)}
\end{align}
Here the stochastic integral in \eqref{u_(n+1)} is defined as the Walsh integral of $G_{n+1}(t-\cdot,x-\star)\sigma(u_n(\cdot, \star)) \in \mathcal{P}_{+}$. 
We first prove that this sequence is well-defined and show some of its properties. 
Here and subsequently, we will often use the notation $G_{n}^{(t,x)}(s,y) \coloneqq G_{n}(t-s,x-y)$.

\begin{Prop}
\label{Prop u_n property}
For every integer $n \geq 0$, we have the following: 
\begin{enumerate}
    \item [\normalfont(i)] The process $\{u_n(t,x) \mid (t,x) \in [0,T]\times \R^d\}$ has a predictable modification.
    \item [\normalfont(ii)] For any $p \geq 1$, $\sup_{(t,x) \in [0,T]\times\mathbb{R}^d}\lVert u_n(t,x)\rVert_p < \infty$.
    \item [\normalfont(iii)] For any $p \geq 1$, $(t,x) \mapsto u_n(t,x)$ is $L^p(\Omega)$-continuous.
\end{enumerate}
\end{Prop}

\begin{proof}
We proceed by induction on $n$.
Clearly, $u_0 \equiv 1$ satisfies (i)-(iii). 
Assume by induction that $u_n$ satisfies (i) and (ii).
Since $(t,x) \mapsto G_{n+1}(t,x)$ is the deterministic measurable function by Lemma \ref{Lem property of G_k}, we can take a predictable modification $\widetilde{u_n}$ of $u_n$ so that the process $\{G_{n+1}^{(t,x)}(s,y)\sigma(\widetilde{u_n} (s,y)) \mid (s,y) \in [0,T] \times \R^d\}$ is predictable.
Then, (ii) for $u_n$ and the Lipschitz continuity of $\sigma$ imply that $G_{n+1}^{(t,x)}(\cdot,\star)\sigma(\widetilde{u_n}(\cdot,\star)) \in \mathcal{P}_+$, and consequently, $u_{n+1}$ defined by
\begin{equation*}
    u_{n+1}(t,x) = 1 + \int_0^t\int_{\mathbb{R}^d}G_{n+1}^{(t,x)}(s,y)\sigma(\widetilde{u_n}(s,y))W(ds,dy)
\end{equation*}
is well-defined for given $(t,x)$.

We now prove (ii) and (iii) for $u_{n+1}$ for every $p \geq 2$.
The claim for the case $1 \leq p < 2$ automatically follows from the case $p \geq 2$.
Applying the Burkholder-Davis-Gundy inequality and Minkowski's inequality, we have
\begin{align*}
    &\lVert u_{n+1}(t,x) \rVert_p^2\\
    &\leq 2 + 2C_p\left\lVert \int_0^t\int_{\mathbb{R}^{2d}} G_{n+1}^{(t,x)}(s,y)G_{n+1}^{(t,x)}(s,z)\sigma(\widetilde{u_n}(s,y))\sigma(\widetilde{u_n}(s,z))\gamma(y-z)dydzds\right\rVert_{\frac{p}{2}}\\
    &\leq 2 + 2C_p\sup_{(\tau, \eta) \in [0,T] \times \R^d}\lVert \sigma(u_n(\tau, \eta)) \rVert_p^2\int_0^t\int_{\R^{2d}}|G_{n+1}^{(t,x)}(s,y)||G_{n+1}^{(t,x)}(s,z)|\gamma(y-z)dydzds\\
    &\leq 2 + 2C_p\sup_{(\tau, \eta) \in [0,T] \times \R^d}\lVert \sigma(u_n(\tau, \eta)) \rVert_p^2 T \Theta(n + 1,T)^2\int_{\mathbb{B}_{T+1}^2}\gamma(y-z)dydz < \infty,
\end{align*}
where $C_p$ is a constant depending on $p$ and the last inequality follows from \eqref{G_nG_ngamma uniform estimate}.
Hence $u_{n+1}$ satisfies (ii).
By a similar computation, we have for any $(t,x), (s,y) \in [0,T] \times \R^d$ that
\begin{align*}
    &\lVert u_{n+1}(t,x) - u_{n+1}(s,y)\rVert_p^2\\
    &\leq C_p\sup_{(\tau, \eta) \in [0,T] \times \R^d}\lVert \sigma(u_n(\tau, \eta)) \rVert_p^2\int_0^T\int_{\R^{2d}}|G_{n+1}^{(t,x)}(r,z) - G_{n+1}^{(s,y)}(r,z)|\\
    &\qquad \qquad \qquad \qquad \qquad \qquad \qquad \qquad \quad\times|G_{n+1}^{(t,x)}(r,w) - G_{n+1}^{(s,y)}(r,w)|\gamma(z-w)dzdwdr.
\end{align*}
We are going to prove 
\begin{equation}
    \lim_{(s,y) \to (t,x)}\lVert u_{n+1}(t,x) - u_{n+1}(s,y)\rVert_p^2 = 0, \label{4.1.1}
\end{equation}
and it is harmless to assume that $|x - y| < 1$. 
By Lemma \ref{Lem property of G_k}, we see that for every $(r,z) \in [0,T] \times \R^d$,
\begin{gather*}
    \lim_{(s,y) \mapsto (t,x)}|G_{n+1}(t-r, x-z) - G_{n+1}(s-r, y-z)| = 0, \\
    |G_{n+1}(t-r, x-z) - G_{n+1}(s-r, y-z)| \leq 2 \Theta(n+1,T)\ind_{\mathbb{B}_{T+2}}(x-z).
\end{gather*}
From there, we can apply the dominated convergence theorem to obtain \eqref{4.1.1}, and hence $u_{n+1}$ satisfies (iii).

Finally, the adaptedness and $L^p(\Omega)$-continuity of $u_{n+1}$ imply that it has a predictable modification, and therefore $u_{n+1}$ satisfies all three statements. 
This completes the proof.
\end{proof}

We can also show that the process $u_n(t,x)$ has the property (S). 
In what follows, $\widetilde{u_n}$ always denotes a predictable modification of $u_n$.

\begin{Lem}
\label{Lem u_n property S}
For every $n \geq 0$, the process $\{u_{n}(t,x) \mid (t,x) \in [0,T] \times \R^d\}$ satisfies the property (S).
\end{Lem}

\begin{proof}
The proof is by induction on $n$. 
Let us consider the distribution of the random vector of the form
\begin{equation*}
    (u_n(t_1,x_1 + z), \ldots, u_n(t_i,x_i + z),W_{t_{i+1}}(A_1 + z), \ldots, W_{t_{i+j}}(A_j + z)),
\end{equation*}
where $t_1,\ldots,t_{i+j} \in [0,T]$, $x_1, \ldots, x_i \in \R^d$, and $A_1, \ldots, A_j \in \mathcal{B}_{\mathrm{b}}(\R^d)$.
Since $u_0 \equiv 1$, we only need to show for the case $n = 0$ that the distribution of $(W_{t_{i+1}}(A_1 +z), \ldots, W_{t_{i+j}}(A_j +z))$ is independent of $z$. 
Because $W$ is the centered Gaussian process and the covariance functions
\begin{align*}
    \E[W_{t_{i+k}}(A_k +z)W_{t_{i+l}}(A_l + z)] 
    &= t_{i+k}\land t_{i+l}\int_{\R^{2d}}\ind_{A_k +z}(x)\ind_{A_l +z}(y)\gamma(x-y)dxdy\\
    &= t_{i+k}\land t_{i+l}\int_{\R^{2d}}\ind_{A_k}(x)\ind_{A_l}(y)\gamma(x-y)dxdy, \quad (1 \leq k,l \leq j),
\end{align*}
do not depend on $z$, the proof for $n=0$ is completed.

Now we assume by induction that $u_n$ satisfies the property (S).
Let
\begin{equation*}
    I_{n+1}(t,x) \coloneqq \int_0^t \int_{\R^d} G_{n+1}^{(t,x)}(s,y)\sigma(\widetilde{u_n}(s,y))W(ds,dy).
\end{equation*}
We only need to show for the case $n+1$ that the distribution of $(I_{n+1}(t_1,x_1 + z), \ldots, I_{n+1}(t_i,x_i + z), W_{t_{i+1}}(A_1 +z), \ldots, W_{t_{i+j}}(A_j +z))$ is independent of $z$.
Observe that 
\begin{align}
    I_{n+1}(t,x + z) 
    &= \int_0^t \int_{\R^d} G_{n+1}(t-s, x + z -y)\sigma(\widetilde{u_n}(s,y))W(ds,dy) \nonumber\\
    &= \int_0^t \int_{\R^d} G_{n+1}(t-s, x -y) \sigma(\widetilde{u_n}(s,y + z))W^{(z)}(ds,dy) \eqqcolon I_{n+1}^{(z)}(t,x) ,  \quad \text{a.s.}, \label{4.2.0}
\end{align}
where the martingale measure $W^{(z)}$ is defined as $W_t^{(z)}(A) = W_t(A + z)$ for every $A \in \mathcal{B}_{\mathrm{b}}(\R^d)$.
Indeed, this can be easily checked when the integrand is a simple process, and we then make an approximation argument to see that \eqref{4.2.0} is true.
Thus, we are reduced to showing that the distribution of $(I_{n+1}^{(z)}(t_1,x_1), \ldots, I_{n+1}^{(z)}(t_i,x_i),W_{t_{i+1}}(A_1 + z), \ldots, W_{t_{i+j}}(A_j + z))$ is independent of $z$.

Set
\begin{equation*}
    V_{n+1}^{(t,x) (z)}(s,y) \coloneqq G_{n+1}^{(t,x)}(s,y)\sigma(\widetilde{u_n}(s, y+z)).
\end{equation*}
By Lemma \ref{Lem property of G_k} and Proposition \ref{Prop u_n property}, we have $\supp V_{n+1}^{(t,x)(z)} \subset [0,T] \times \mathbb{B}_{T+1}(x)$, and the map $(s,y) \mapsto V_{n+1}^{(t,x)(z)}$ is uniformly $L^p(\Omega)$-continuous on the compact set $[0,T] \times \mathbb{B}_{T+1}(x)$ for any $p \geq 1$.
If we take $z_1, z_2 \in \R^d$, then there exists $\delta_{m} > 0$ such that for any $(s,y), (s',y') \in [0,T] \times \mathbb{B}_{T+1}(x)$ with $|(s,y) - (s',y')| < \delta_m$, 
\begin{equation*}
    \lVert V_{n+1}^{(t,x)(z_l)}(s,y) - V_{n+1}^{(t,x)(z_l)}(s',y')\rVert_2 < 2^{-m}, \quad l = 1,2.
\end{equation*}
For each $m$, let $0 = s_0^m < s_1^m < \ldots <s_{n_m}^m =T$ be a partition of $[0,T]$ such that 
\begin{equation*}
    \max_{0 \leq i \leq n_{m} -1}|s_i^{m} - s_{i+1}^{m}| < \frac{1}{2}\delta_m.
\end{equation*}
Also let $(U_j^m)_{1\leq j \leq k_m}$ be a partition of $\mathbb{B}_{T+1}(x)$ such that 
\begin{equation*}
    U_j^m \in \mathcal{B}(\R^d), \quad \sup \{|y-z| : y,z \in U_j^{m}\}< \frac{1}{2}\delta_m,
\end{equation*}
and $(U_j^{m})$ are mutually disjoint up to sets of Lebesgue measure zero. 
For each $j$, take $y_j^m \in U_j^{m}$ arbitrarily.

Now, for each $l =1,2$, we define
\begin{equation*}
    V_{n+1,m}^{(t,x)(z_l)}(s,y) = \sum_{i=0}^{n_m - 1}\sum_{j=1}^{k_m}V_{n+1}^{(t,x)(z_l)}(s_i^m,y_j^m)\ind_{(s_i^m, s_{i+1}^{m}]}(s)\ind_{U_{j}^m}(y).
\end{equation*}
It follows that 
\begin{equation}
\label{4.2.1}
    \sup_{(s,y) \in [0,T] \times \mathbb{B}_{T+1}(x)} \lVert V_{n+1}^{(t,x)(z_l)}(s,y) - V_{n+1,m}^{(t,x)(z_l)}(s,y)\rVert_2 < 2^{-m}.
\end{equation}
It is also clear that the process $V_{n+1,m}^{(t,x)(z_l)}(s,y)$ is predictable and the Walsh integral 
\begin{align}
    I_{n+1,m}^{(z_l)}(t,x) 
    &\coloneqq \int_0^t\int_{\R^d} V_{n+1,m}^{(t,x)(z_l)}(s,y) W^{(z_l)}(ds,dy) \nonumber\\
    &= \sum_{i=0}^{n_m - 1}\sum_{j=1}^{k_m}V_{n+1}^{(t,x)(z_l)}(s_i^m,y_j^m)\{W_{t \land s_{i+1}^{m}}^{(z_l)}({U_{j}^m}) - W_{t \land s_i^m}^{(z_l)}({U_{j}^m})\} \label{4.2.2}
\end{align}
is well-defined. 
From \eqref{4.2.2}, the distribution of 
\begin{equation*}
    (I_{n+1,m}^{(z_l)}(t_1,x_1), \ldots, I_{n+1,m}^{(z_l)}(t_i,x_i),W_{t_{i+1}}(A_1 + z_l), \ldots, W_{t_{i+j}}(A_j + z_l))
\end{equation*}
does not depend on $l \in \{1,2\}$ by the induction assumption.
Moreover, using \eqref{4.2.1}, we can check that 
\begin{equation}
\label{4.2.4}
    I_{n+1,m}^{(z_l)}(t,x) \xrightarrow[m \to \infty]{}  I_{n+1}^{(z_l)}(t,x) \quad \text{in $L^2(\Omega)$} 
\end{equation}
for each $(t,x) \in [0,T] \times \R^d$ and $l \in \{1,2\}$, and consequently,
\begin{align}
    &(I_{n+1,m}^{(z_l)}(t_1,x_1), \ldots, I_{n+1,m}^{(z_l)}(t_i,x_i),W_{t_{i+1}}(A_1+ z_l), \ldots, W_{t_{i+j}}(A_j + z_l)) \nonumber \\
    &\xrightarrow[m \to \infty]{} (I_{n+1}^{(z_l)}(t_1,x_1), \ldots, I_{n+1}^{(z_l)}(t_i,x_i),W_{t_{i+1}}(A_1 + z_l), \ldots, W_{t_{i+j}}(A_j + z_l)) \label{4.2.3} 
\end{align}
in $L^2(\Omega ; \R^{i + j})$ for each $l \in \{1,2\}$. 
Therefore, the distribution of \eqref{4.2.3} is also independent of $l\in \{1,2\}$.
Since $z_1,z_2$ can be taken arbitrarily, $u_{n+1}$ satisfies the property (S), and the proof is complete.
\end{proof}

Using the property (S) of $u_n(t,x)$, we see that the statement (ii) in Proposition \ref{Prop u_n property} holds uniformly in $n$ when $p = 2$.

\begin{Prop}
\label{Prop u_n uniform bound}
It holds that
\begin{equation*}
    \sup_{n \geq 0}\sup_{(t,x) \in [0,T]\times\mathbb{R}^d}\lVert u_n(t,x)\rVert_2 < \infty.
\end{equation*}
\end{Prop}

\begin{proof}
By Proposition \ref{Prop u_n property} and Lemma \ref{Lem u_n property S}, the process $Z(t,x) \coloneqq \sigma(\widetilde{u_n}(t,x))$ satisfies the conditions \eqref{Z uniform L2 bound} and \eqref{Z homogeneous cov}.
Hence, by Lemma \ref{Lem fourier bochner schwartz}, there is a nonnegative tempered measure $\mu_t^{\sigma(\widetilde{u_n})}$ for each $t$ such that 
\begin{equation*}
    \gamma^{\sigma(\widetilde{u_n})} = \F \mu_t^{\sigma(\widetilde{u_n})} \quad \text{in $\mathcal{S}_{\C}'(\R^d)$}. 
\end{equation*}
We then apply \eqref{FG_n} and Lemmas \ref{Lem property of FG} and \ref{Lem muZ mu} to obtain
\begin{align*}
    \lVert u_{n+1}(t,x) \rVert_2^2 
    &\leq 2+ 2\int_0^t\int_{\R^d}|\mathcal{F}G_{n+1}(t-s)(\xi)|^2\mu_s^{\sigma(\widetilde{u_n})}(d\xi)ds\\
    &\leq 2+ 2C_T \int_{\R^d} \langle \xi \rangle^{-2} \mu(d\xi) \int_0^t \sup_{(r,z) \in [0,s] \times \R^d}\lVert \sigma(u_n(r,z)) \rVert_2^2 ds\\
    &\leq 2+ 4C_T \sigma(0)^2 T \int_{\R^d}\langle \xi \rangle^{-2} \mu(d\xi) + 4C_T \sigma_{\mathrm{Lip}}^2 \int_{\R^d} \langle \xi \rangle^{-2} \mu(d\xi) \int_0^t \sup_{(r,z) \in [0,s] \times \R^d}\lVert u_n(r,z) \rVert_2^2 ds.
\end{align*}
Here $C_T = (1+2T^2)$.
Set $H_n(t) = \sup_{(r,z) \in [0,t] \times \R^d}\lVert u_n(r,z) \rVert_2^2$. 
It follows that 
\begin{equation}
\label{4.3.1}
    H_{n+1}(T) \leq C_1(T) + C_2(T, \sigma_{\mathrm{Lip}})\int_0^T H_n(s) ds,
\end{equation}
where 
\begin{equation*}
    C_1(T) \coloneqq 2+ 4C_T \sigma(0)^2 T \int_{\R^d}\langle \xi \rangle^{-2} \mu(d\xi) \quad \text{and} \quad C_2(T, \sigma_{\mathrm{Lip}}) \coloneqq 4C_T \sigma_{\mathrm{Lip}}^2 \int_{\R^d} \langle \xi \rangle^{-2} \mu(d\xi).
\end{equation*}
Since $C_1(T) \geq 1$, we can apply \eqref{4.3.1} repeatedly to obtain 
\begin{equation*}
    H_n(T) \leq C_1(T)e^{C_2(T, \sigma_{\mathrm{Lip}})T},
\end{equation*}
which is the desired conclusion.
\end{proof}

\begin{Rem}
When $d \leq 3$, we can show (\textit{cf.} \cite[Proposition 3.2]{ebina2023central}) that Proposition \ref{Prop u_n uniform bound} holds for any $p \geq 1$. 
However, in dimensions $d \geq 4$, it is unclear how to obtain the bound uniformly in $n$ for $p >2$ due to the lack of the nonnegativity of $G_n$. 
\end{Rem}

Now we check that the sequence $(u_n(t,x))_{n\geq 0}$ converges in $L^2(\Omega)$ to the solution $U(t,x)$ of \eqref{SPDE} for every $(t,x) \in [0,T] \times \mathbb{R}^d$. 
We first show the following lemma, which is slightly stronger than that the process $\sigma(U(t,x)) - \sigma(u_n(t,x))$ has the property (S). 
Recall that $W_t^{(z)}(A) \coloneqq W_t(z + A)$.

\begin{Lem}
\label{Lem V_n property S}
For any $z \in \R^d$ and every integers $n \geq 0$ and $i,j,k \geq 1$, all distributions of random vectors of the form
\begin{equation}
\label{4.4.1}
    (U(t_1,x_1 + z), \ldots,U(t_i,x_i + z), u_n(t_{i+1},x_{i+1} + z), \ldots, u_n(t_{i+j}, x_{i+j} + z), W_{t_{i+j+1}}^{(z)}(A_1), \ldots, W_{t_{i+j+k}}^{(z)}(A_k) ),
\end{equation}
where $t_1, \ldots, t_{i + j + k} \in [0,T]$, $x_1, \ldots, x_{i+j} \in \R^d$, and $A_1, \ldots, A_k \in \mathcal{B}_{\mathrm{b}}(\R^d)$, are independent of $z$.
In particular, the process $\{\sigma(U(t,x)) - \sigma(u_n(t,x)) \mid (t,x) \in [0,T] \times \R^d\}$ satisfies the property (S).
\end{Lem}

\begin{proof}
The proof is done by induction on $n$.
Because $u_0 \equiv 1$ and the process $U(t,x)$ satisfies the property (S), the statement holds for the case $n = 0$.
Assuming by induction that the statement holds for $n$, we show that it holds for $n+1$.
Let us consider the distribution of \eqref{4.4.1}, with $u_{n}$ replaced by $u_{n+1}$.
Fix $z_1, z_2 \in \R^d$.
By a similar argument in the proof of Lemma \ref{Lem u_n property S}, it suffices to show that the distribution of 
\begin{equation*}
    (U(t_1,x_1 + z_l), \ldots,U(t_i,x_i + z_l), I_{n+1}^{(z_l)}(t_{i+1},x_{i+1}), \ldots, I_{n+1}^{(z_l)}(t_{i+j}, x_{i+j}), W_{t_{i+j+1}}^{(z_l)}(A_1), \ldots, W_{t_{i+j+k}}^{(z_l)}(A_k) )
\end{equation*}
does not depend on $l \in \{1,2\}$, where $I_{n+1}^{(z)}(t,x)$ is defined in \eqref{4.2.0}.
Moreover, using \eqref{4.2.4}, we only need to show that the distribution of 
\begin{equation*}
    (U(t_1,x_1 + z_l), \ldots,U(t_i,x_i + z_l), I_{n+1,m}^{(z_l)}(t_{i+1},x_{i+1}), \ldots, I_{n+1,m}^{(z_l)}(t_{i+j}, x_{i+j}), W_{t_{i+j+1}}^{(z_l)}(A_1), \ldots, W_{t_{i+j+k}}^{(z_l)}(A_k))
\end{equation*}
is independent of $l \in \{1,2\}$ for every $m$. 
Because $I_{n+1,m}^{(z)}(t,x)$ defined in \eqref{4.2.2} is a measurable function of $u_n(\cdot, \cdot + z)$ and $W_{\cdot}(\cdot + z)$, we can apply the induction assumption to obtain the statement for the case $n+1$.
This completes the proof.
\end{proof}

\begin{Prop}
\label{Prop u_n convergence to U}
We have 
\begin{equation}
\label{4.6.0}
    \lim_{n \to \infty}\sup_{(t,x) \in [0,T] \times \mathbb{R}^d}\lVert U(t,x) - u_n(t,x) \rVert_2 = 0.
\end{equation}
\end{Prop}

\begin{proof}
First, observe that we have
\begin{align*}
    U(t,x) - u_{n+1}(t,x)
    &= \int_0^t\int_{\R^d}(G(t-s,x-y) - G_{n+1}(t-s, x-y))W^{\sigma(U)}(ds,dy)\\
    &\quad + \int_0^t\int_{\R^d}G_{n+1}(t-s,x-y)(\sigma(U(s,y)-\sigma(\widetilde{u_n}(s,y)))W(ds,dy),
\end{align*}
where the first stochastic integral is Dalang's integral of $G(t-s,x-y) - G_{n+1}(t-s,x-y) \in \mathcal{P}_{0, \sigma(U)}$ and the second stochastic integral is the Walsh integral of $G_{n+1}(t-s,x-y)(\sigma(U(s,y)) - \sigma(\widetilde{u_n}(s,y))) \in \mathcal{P}_{+}$.
Indeed, this follows from the fact (\textit{cf.} Remark \ref{Rem 2.2}) that 
\begin{equation*}
    \int_0^t\int_{\R^d}G_{n+1}(t-s,x-y)\sigma(U(s,y))W(ds,dy) = \int_0^t\int_{\R^d}G_{n+1}(t-s,x-y)W^{\sigma(U)}(ds,dy).
\end{equation*}
Using the isometry property of stochastic integrals and Lemma \ref{Lem property of FG}, we have 
\begin{align*}
    &\lVert U(t,x) - u_{n+1}(t,x) \rVert_2^2\\
    &\leq 2\left\lVert \int_0^t \int_{\R^d} (G(t-s, x-y) - G_{n+1}(t-s, x-y)) W^{\sigma(U)}(ds,dy)\right\rVert_2^2\\
    &\quad + 2\left\lVert \int_0^t \int_{\R^d} G_{n+1}(t-s, x-y) (\sigma(U(s,y)) - \sigma(\widetilde{u_n}(s,y))) W(ds,dy)\right\rVert_2^2\\
    &= 2\int_0^t \int_{\R^d}|\F G(t-s)(\xi)|^2|1 - \F \Lambda_{n+1}(\xi)|^2 \mu_s^{\sigma(U)}(d\xi) ds \\
    &\quad + 2\int_0^t \int_{\R^{2d}} G_{n+1}(t-s, x-y)G_{n+1}(t-s, x-z)\gamma(y-z) \\
    &\qquad \qquad \qquad \times\E[(\sigma(U(s,y)) - \sigma(\widetilde{u_n}(s,y)))(\sigma(U(s,z)) - \sigma(\widetilde{u_n}(s,z)))] dydzds\\
    &\leq 2(1+2T^2)\int_0^T \int_{\R^d}\langle \xi \rangle^{-2} |1 - \F \Lambda_{n+1}(\xi)|^2 \mu_s^{\sigma(U)}(d\xi) ds\\
    &\quad + 2\int_0^t \int_{\R^{2d}} G_{n+1}(t-s, x-y)G_{n+1}(t-s, x-z)\gamma(y-z) \\
    &\qquad \qquad \qquad \times\E[(\sigma(U(s,y)) - \sigma(\widetilde{u_n}(s,y)))(\sigma(U(s,z)) - \sigma(\widetilde{u_n}(s,z)))] dydzds\\
    &\eqqcolon 2(1+2T^2)\mathbf{I_1} + 2\mathbf{I_2}.
\end{align*}
Because 
\begin{equation*}
    \lim_{n \to \infty}|1 - \F \Lambda_{n}(\xi)|^2 = 0, \quad |1 - \F \Lambda_{n + 1}(\xi)|^2 \leq 4,
\end{equation*}
and 
\begin{equation*}
    \int_{\R^d}\langle \xi \rangle^{-2}\mu_s^{\sigma(U)}(d\xi)
    \leq \sup_{(r,z) \in [0,T] \times \R^d}\lVert \sigma(U(t,x)) \rVert_2^2 \int_{\R^d}\langle \xi \rangle^{-2}\mu(d\xi) < \infty
\end{equation*}
by Lemma \ref{Lem muZ mu}, the dominated convergence theorem yields that for any $s$, 
\begin{equation*}
    \lim_{n \to \infty} \int_{\R^d}\langle \xi \rangle^{-2}|1-\F \Lambda_{n+1}(\xi)|^2 \mu_s^{\sigma(U)}(d\xi) = 0.
\end{equation*}
It follows that $\lim_{n \to \infty}\mathbf{I_1} = 0$.
We next estimate the term $\mathbf{I_2}$.
By Proposition \ref{Prop u_n property} and Lemma \ref{Lem V_n property S}, the process $\sigma(U(t,x)) - \sigma(\widetilde{u_n}(t,x))$ satisfies the conditions \eqref{Z uniform L2 bound} and \eqref{Z homogeneous cov}. Hence, by Lemma \ref{Lem fourier bochner schwartz}, there is a nonnegative tempered measure $\mu_t^{\sigma(U) - \sigma(\widetilde{u_n})}$ such that 
\begin{equation*}
    \gamma^{\sigma(U) - \sigma(\widetilde{u_n})}(t,\cdot) = \F \mu_t^{\sigma(U) - \sigma(\widetilde{u_n})}(d\xi) \quad \text{in $\mathcal{S}_{\C}'(\R^d)$}.
\end{equation*}
Since $G_n(t) \in \mathcal{S}(\mathbb{R}^d)$ for each $t$, we apply Lemmas \ref{Lem property of FG} and \ref{Lem muZ mu} to obtain 
\begin{align*}
    \mathbf{I_2} 
    &= \int_0^t \int_{\R^d} |\F G_{n+1}(t-s)(\xi)|^2 \mu_s^{\sigma(U) - \sigma(\widetilde{u_n})}(d\xi) ds\\
    &\leq (1+2T^2) \int_0^t \int_{\R^d} \langle \xi \rangle^{-2}\mu_s^{\sigma(U) - \sigma(\widetilde{u_n})}(d\xi) ds\\
    &\leq (1+2T^2) \sigma_{\mathrm{Lip}}^2 \int_{\R^d}\langle \xi \rangle^{-2} \mu(d\xi)\int_0^t\sup_{(r,z) \in [0,s] \times \R^d}\lVert U(r,z) - \widetilde{u_n}(r,z) \rVert_2^2 ds.
\end{align*}
If we set $h_n(t) = \sup_{(r,z) \in [0,t] \times \R^d}\lVert U(r,z) - u_n(r,z) \rVert_2^2$, then we get
\begin{equation*}
    h_{n+1}(T) \leq C(n + 1, T) + C(T, \sigma_{\mathrm{Lip}})\int_0^T h_n(s) ds,
\end{equation*}
where
\begin{align*}
    C(n +1 ,T) &\coloneqq  2(1+2T^2)\int_0^T \int_{\R^d}\langle \xi \rangle^{-2} |1 - \F \Lambda_{n+1}(\xi)|^2 \mu_s^{\sigma(U)}(d\xi) ds,\\ 
    C(T, \sigma_{\mathrm{Lip}}) &\coloneqq (1+2T^2) \sigma_{\mathrm{Lip}}^2\int_{\R^d}\langle \xi \rangle^{-2} \mu(d\xi).
\end{align*}
We conclude from this inequality together with $\lim_{n \to \infty}C(n+1, T) = 0$ that $\lim_{n \to \infty} h_n(T) = 0$, which is the desired conclusion.
\end{proof}

\section{Malliavin derivative of the approximation sequence}
\label{section Malliavin derivative of the approximation sequence}
In this section, we require that $\sigma$ is a continuously differentiable Lipschitz function. 
We show in Sections \ref{subsection Differentiability of the approximation sequence} and \ref{subsection Version of the derivative} that $u_n(t,x)$, constructed in Section \ref{section Approximation to the solution}, is Malliavin differentiable, and that its derivative satisfies the pointwise moment bounds \eqref{int powise est}.
We then evaluate the covariances of $\sigma(u_n(t,x))$ using the Clark-Ocone formula in Section \ref{subsection Covariance estimates}.

\subsection{Differentiability of the approximation sequence}
\label{subsection Differentiability of the approximation sequence}

\begin{Prop}
\label{Prop u_n diff'bility}
Let $(e_k)_{k=1}^{\infty}$ be a complete orthonormal system of $\mathcal{H}_T$.
For every $n \geq 1$, we have the following:
\begin{enumerate}
    \item [\normalfont(i)] For any $(t,x) \in [0,T]\times \R^d$ and $p \in [1,\infty)$, we have $u_n(t,x) \in \mathbb{D}^{1,p}$. Moreover, for each $k \in \mathbb{N}$, the process $\langle Du_{n}(t,x), e_k \rangle_{\mathcal{H}_T}$ has a predictable modification, which will be written by $\widetilde{D^{e_k}u_{n}}(t,x)$.
    \item [\normalfont(ii)] The derivative $Du_n(t,x)$ is an $\mathscr{F}_t$-measurable $\mathcal{H}_T$-valued random variable and satisfies 
    \begin{align}
        Du_{n}(t,x) \nonumber
        &= G_{n}^{(t,x)}(\cdot,\star)\sigma(\widetilde{u_{n-1}}(\cdot,\star))\\
        &\quad + \sum_{k=1}^{\infty} \left( \int_0^t \int_{\R^d}G_{n}^{(t,x)}(s,y)\sigma'(\widetilde{u_{n-1}}(s,y)) \widetilde{D^{e_k}u_{n-1}}(s,y) W(ds,dy) \right) e_k, \quad \text{a.s.}. \label{Du_(n)(t,x)}
    \end{align}
    \item [\normalfont(iii)] For any $p \in [1,\infty)$, $\sup_{(t,x) \in [0,T] \times \R^d} \lVert Du_n(t,x) \rVert_{L^p(\Omega;\mathcal{H}_T)} < \infty$.
    \item [\normalfont(iv)] For any $p \in [1,\infty)$, $(t,x) \mapsto Du_n(t,x)$ is $L^p(\Omega;\mathcal{H}_T)$-continuous.
\end{enumerate}
\end{Prop}

\begin{proof}
The proof is by induction on $n$.
We first show the claim when $n=1$.
By \eqref{u_(n+1)}, we see that
\begin{equation*}
    u_1(t,x) 
    = 1 + W(G_1^{(t,x)}(\cdot, \star)\sigma(1)) \in \mathbb{D}^{1,p}
\end{equation*}
for any $(t,x) \in [0,T] \times \R^d$ and $p \geq 1$ and 
\begin{equation*}
    Du_1(t,x) = G_1^{(t,x)}(\cdot, \star) \sigma(1).
\end{equation*}
Since $Du_1(t,x)$ is deterministic and $Du_0(t,x) = 0$, statements (i)-(iii) hold.
Moreover, using Lemma \ref{Lem property of G_k}, we can also check that (iv) holds by making a similar argument that we used to show \eqref{4.1.1}.

Now we assume by induction that all (i)-(iv) hold for $n$, and we will show it for $n+1$.
Fix $(t,x) \in [0,T] \times \R^d$ and $p \geq 2$ and let 
\begin{equation*}
    V_{n+1}^{(t,x)}(s,y) = G_{n+1}^{(t,x)}(s,y)\sigma(\widetilde{u_n}(s,y)).
\end{equation*}
The continuity of $G_{n+1}$ and the $L^p(\Omega)$-continuity of $u_n$ imply that $(s,y) \mapsto V_{n+1}^{(t,x)}(s,y)$ is $L^p(\Omega)$-continuous. 
Thus, $G_{n+1}^{(t,x)}$, $u_n$, $V_{n+1}^{(t,x)}$, and $Du_n$ are all uniformly continuous on the compact set $[0,T] \times \mathbb{B}_{T+1}(x)$, and there exists $\delta_{m} > 0$ such that for any $(s,y), (s',y') \in [0,T] \times \mathbb{B}_{T+1}(x)$ with $|(s,y) - (s',y')| < \delta_m$, 
\begin{align}
    |G_{n+1}^{(t,x)}(s,y) - G_{n+1}^{(t,x)}(s',y')| &< 2^{-m}, \label{G_(n+1) uni cty}\\
    \lVert u_n(s,y) - u_n(s',y') \rVert_p &< 2^{-m}, \label{u_n uni cty}\\
    \lVert V_{n+1}^{(t,x)}(s,y) - V_{n+1}^{(t,x)}(s',y') \rVert_p &< 2^{-m}, \quad \text{and} \label{V_(n+1) uni cty}\\
    \lVert Du_{n}(s,y) - Du_{n}(s',y') \rVert_{L^p(\Omega ; \mathcal{H}_T)} &< 2^{-m}. \label{Du_n uni cty}
\end{align}
Note that $\delta_m$ depends on $(t,x)$ and $p$, but we will not write this dependency explicitly for simplicity of notation.
For each $m$, let $0 = s_0^{m} < s_1^{m} < \ldots <s_{k_m}^{m} =T$ be a partition of $[0,T]$ such that 
\begin{equation*}
    \max_{0 \leq i \leq k_{m} -1}|s_i^{m} - s_{i+1}^{m}| < \frac{1}{2}\delta_{m}.
\end{equation*}
Also let $(U_j^m)_{0\leq j \leq l_m}$ be a partition of $\mathbb{B}_{T+1}(x)$ such that 
\begin{equation*}
    U_j^m \in \mathcal{B}(\R^d), \quad \sup \{|y -z| : y,z \in U_j^{m}\}< \frac{1}{2}\delta_{m},
\end{equation*}
and $(U_j^{m})_{0\leq j \leq l_m}$ are mutually disjoint up to sets of Lebesgue measure zero. 
For each $j$, take $y_j^m \in U_j^{m}$ arbitrarily.
Let us define
\begin{equation*}
    V_{n+1,m}^{(t,x)}(s,y) = \sum_{i=0}^{k_m - 1}\sum_{j=0}^{l_m}V_{n+1}^{(t,x)}(s_i^m,y_j^m)\ind_{(s_i^m, s_{i+1}^{m}]}(s)\ind_{U_{j}^m}(y).
\end{equation*}
Clearly, the process $V_{n+1,m}^{(t,x)}(s,y)$ is predictable, and we see from \eqref{V_(n+1) uni cty} that
\begin{equation}
    \sup_{(s,y) \in [0,T] \times \mathbb{B}_{T+1}(x)} \lVert V_{n+1}^{(t,x)}(s,y) - V_{n+1,m}^{(t,x)}(s,y) \rVert_{p} \leq 2^{-m}. \label{V-Vm}
\end{equation}
Moreover, its Walsh integral is given by
\begin{equation*}
    I_{n+1,m}(t,x) \coloneqq \int_0^T\int_{\R^d}V_{n+1,m}^{(t,x)}(s,y)W(ds,dy) = \sum_{i=0}^{k_m - 1}\sum_{j=0}^{l_m}V_{n+1}^{(t,x)}(s_i^m,y_j^m)W(\ind_{(s_i^m, s_{i+1}^{m}]}\ind_{U_{j}^m}),
\end{equation*}
and it is easily seen from \eqref{V-Vm} that 
\begin{equation}
    \int_0^T\int_{\R^d}V_{n+1,m}^{(t,x)}(s,y)W(ds,dy) \xrightarrow[m \to \infty]{} \int_0^T\int_{\R^d}V_{n+1}^{(t,x)}(s,y)W(ds,dy) \quad \text{in $L^p(\Omega)$}. \label{VmW conv}
\end{equation}

We first show that $u_{n+1}(t,x) \in \mathbb{D}^{1,p}$. 
From \eqref{VmW conv}, it suffices to show (\textit{cf.} \cite[Lemma 1.5.3]{MR2200233}) that $I_{n+1,m}(t,x) \in \mathbb{D}^{1,p}$ and $DI_{n+1,m}(t,x)$ converges in $L^p(\Omega;\mathcal{H}_T)$ as $m \to \infty$.
Because $\sigma$ is a continuously differentiable function with bounded derivative and $u_n(t,x) \in \mathbb{D}^{1,q}$ for any $(t,x)$ and $q \geq 1$, it follows from the chain rule that $V_{n+1}^{(t,x)}(s,y) \in \mathbb{D}^{1,q}$ for any $(s,y)$ and $q \geq 1$.
We then conclude from the Leibniz rule (see \textit{e.g.} \cite[Proposition 1.5.6]{MR2200233}) that $I_{n+1,m}(t,x) \in \mathbb{D}^{1,p}$, and we have 
\begin{equation}
    DI_{n+1,m}(t,x) = V_{n+1,m}^{(t,x)}(\cdot, \star) + \sum_{i=0}^{k_m - 1}\sum_{j=0}^{l_m} DV_{n+1}^{(t,x)}(s_i^m,y_j^m)W(\ind_{(s_i^m, s_{i+1}^{m}]}\ind_{U_{j}^m}). \label{DI_(n+1,m)}
\end{equation}
Using Proposition \ref{Prop L^2(Omega H_T)} and \eqref{V-Vm}, we see that 
\begin{equation}
    V_{n+1,m}^{(t,x)}(\cdot, \star) \xrightarrow[m \to \infty]{} V_{n+1}^{(t,x)}(\cdot, \star) \quad \text{in $L^p(\Omega;\mathcal{H}_T)$}.
\end{equation}
To show the convergence of the second term on the right-hand side of \eqref{DI_(n+1,m)}, we begin by evaluating the $L^p(\Omega;\mathcal{H}_T)$ norm of $DV_{n+1}^{(t,x)}(s,y) - DV_{n+1,m}^{(t,x)}(s,y)$. 
From the chain rule, we have
\begin{align*}
    &DV_{n+1}^{(t,x)}(s,y) - DV_{n+1,m}^{(t,x)}(s,y)\\
    &=\sum_{i=0}^{k_m - 1}\sum_{j=0}^{l_m}\left (G_{n+1}^{(t,x)}(s,y)\sigma'(\widetilde{u_n}(s,y))D\widetilde{u_n}(s,y) - G_{n+1}^{(t,x)}(s_i^m,y_j^m)\sigma'(\widetilde{u_n}(s_i^m,y_j^m))D\widetilde{u_n}(s_i^m,y_j^m)\right)\\
    &\qquad \qquad \qquad \times \ind_{(s_i^m, s_{i+1}^{m}]}(s)\ind_{U_{j}^m}(y),
\end{align*}
and it follows from \eqref{G_(n+1) uni cty} and \eqref{Du_n uni cty} that 
\begin{align}
    &\left \lVert DV_{n+1}^{(t,x)}(s,y) - DV_{n+1,m}^{(t,x)}(s,y) \right \rVert_{L^p(\Omega;\mathcal{H}_T)} \nonumber \\
    &\leq \sum_{i=0}^{k_m - 1}\sum_{j=0}^{l_m}\left|G_{n+1}^{(t,x)}(s,y)\right| \left \lVert \sigma'(u_n(s,y))(Du_n(s,y) - Du_n(s^m_i,y^m_j)) \right\rVert _{L^p(\Omega;\mathcal{H}_T)}\ind_{(s_i^m, s_{i+1}^{m}]}(s)\ind_{U_{j}^m}(y) \nonumber \\
    &\quad + \sum_{i=0}^{k_m - 1}\sum_{j=0}^{l_m}\left|G_{n+1}^{(t,x)}(s,y)\right|\lVert \sigma'(u_n(s,y)) - \sigma'(u_n(s^m_i,y^m_j)) \rVert_{2p}\lVert Du_n(s^m_i, y^m_j) \rVert_{L^{2p}(\Omega;\mathcal{H}_T)}\ind_{(s_i^m, s_{i+1}^{m}]}(s)\ind_{U_{j}^m}(y) \nonumber \\
    &\quad + \sum_{i=0}^{k_m - 1}\sum_{j=0}^{l_m}\left |G_{n+1}^{(t,x)}(s,y) - G_{n+1}^{(t,x)}(s^m_i,y^m_j)\right | \lVert \sigma'(u_n(s^m_i,y^m_j)) \rVert_{2p}\lVert Du_n(s^m_i, y^m_j) \rVert_{L^{2p}(\Omega;\mathcal{H}_T)}\ind_{(s_i^m, s_{i+1}^{m}]}(s)\ind_{U_{j}^m}(y) \nonumber \\
    &\leq \left(\Theta(n+1,T) + \sup_{(\tau, \zeta) \in [0,T] \times \R^d} \lVert Du_n(\tau, \zeta) \rVert_{L^{2p}(\Omega;\mathcal{H}_T)} \right)\sigma_{\mathrm{Lip}} 2^{-m} \sum_{i=0}^{k_m - 1}\sum_{j=0}^{l_m}\ind_{(s_i^m, s_{i+1}^{m}]}(s)\ind_{U_{j}^m}(y) \nonumber \\
    &\quad +  \Biggl \{\Theta(n+1,T)\sup_{(\tau, \zeta) \in [0,T] \times \R^d} \lVert Du_n(\tau, \zeta) \rVert_{L^{2p}(\Omega;\mathcal{H}_T)} \nonumber\\
    &\qquad \qquad \times \sum_{i=0}^{k_m - 1}\sum_{j=0}^{l_m}\lVert \sigma'(u_n(s,y)) - \sigma'(u_n(s^m_i,y^m_j)) \rVert_{2p}\ind_{(s_i^m, s_{i+1}^{m}]}(s)\ind_{U_{j}^m}(y) \Biggr \} \label{DV_(n+1) - DV_(n+1,m)}.
\end{align}
Now observe that almost surely,
\begin{align*}
    \sum_{i=0}^{k_m - 1}\sum_{j=0}^{l_m} DV_{n+1}^{(t,x)}(s_i^m,y_j^m)W(\ind_{(s_i^m, s_{i+1}^{m}]}\ind_{U_{j}^m})
    &= \sum_{k=1}^{\infty}\left(\int_0^t\int_{\R^d}\langle DV_{n+1,m}^{(t,x)}(s,y), e_k \rangle_{\mathcal{H}_T}W(ds,dy)\right) e_k \\
    &\eqqcolon \Xi_{n+1,m}(t,x). 
\end{align*}
We will show that the right-hand side above converges as $m \to \infty$ to 
\begin{equation}
\label{conv 2}
    \Xi_{n+1}(t,x) \coloneqq \sum_{k=1}^{\infty} \left( \int_0^t \int_{\R^d}\langle DV_{n+1}^{(t,x)}(s,y), e_k \rangle_{\mathcal{H}_T} W(ds,dy) \right) e_k,
\end{equation}
where we have chosen a predictable modification of the process $\langle DV_{n+1}^{(t,x)}(s,y), e_k \rangle_{\mathcal{H}_T}$ for each $k \in \mathbb{N}$.
Note that \eqref{conv 2} belongs to $L^p(\Omega;\mathcal{H}_T)$, as can be seen from the induction assumption and \eqref{G_nG_ngamma uniform estimate} using the Burkholder-Davis-Gundy inequality.
Let $\upsilon_{n+1,m}^{(t,x)}(s,y) = DV_{n+1}^{(t,x)}(s,y) - DV_{n+1,m}^{(t,x)}(s,y)$.
A standard computation shows that
\begin{align}
    &\E\left[\lVert \Xi_{n+1}(t,x) - \Xi_{n+1,m}(t,x) \rVert_{\mathcal{H}_T}^p\right]^{\frac{2}{p}} \nonumber \\
    &= \E\left[ \left(\sum_{k=1}^{\infty} \int_0^t\int_{\R^{2d}} \langle \upsilon_{n+1,m}^{(t,x)}(s,y) , e_k \rangle_{\mathcal{H}_T}\langle \upsilon_{n+1,m}^{(t,x)}(s,z), e_k \rangle_{\mathcal{H}_T}\gamma(y-z) dydzds\right )^{\frac{p}{2}} \right]^{\frac{2}{p}} \nonumber \\
    &= \E\left[\left( \int_0^t\int_{\R^{2d}}\langle \upsilon_{n+1,m}^{(t,x)}(s,y) ,\upsilon_{n+1,m}^{(t,x)}(s,z) \rangle_{\mathcal{H}_T} \gamma(y-z)dydzds \right)^{\frac{p}{2}} \right]^{\frac{2}{p}} \nonumber\\
    &\leq \int_0^t\int_{\R^{2d}}\lVert \upsilon_{n+1,m}^{(t,x)}(s,y)  \rVert_{L^p(\Omega;\mathcal{H}_T)}\lVert \upsilon_{n+1,m}^{(t,x)}(s,z) \rVert_{L^p(\Omega;\mathcal{H}_T)}\gamma(y-z)dydzds. \label{ine xi_(n+1,m) - xi_(n+1)}
\end{align}
We then apply \eqref{DV_(n+1) - DV_(n+1,m)} to obtain that 
\begin{align}
    &\E\left[\lVert \Xi_{n+1}(t,x) - \Xi_{n+1,m}(t,x) \rVert_{\mathcal{H}_T}^p\right]^{\frac{2}{p}} \nonumber \\
    &\lesssim_{n,T,\sigma_{\mathrm{Lip}}} (2^{-2m} + 2^{-m})\int_0^T\int_{{\mathbb{B}_{T+1}(x)}^2}\gamma(y-z)dydzds \nonumber \\
    &\quad + \int_0^T\int_{{\mathbb{B}_{T+1}(x)}^2}\lVert \sigma'(u_n(s,y)) - \sigma'(u_{n,m}(s,y)) \rVert_{2p}\lVert \sigma'(u_n(s,z)) - \sigma'(u_{n,m}(s,z))\rVert_{2p} \gamma(y-z)dydzds,  \label{conv xi_(n+1,m) - xi_(n+1)}
\end{align}
where 
\begin{equation*}
    u_{n,m}(s,y) \coloneqq \sum_{i=0}^{k_m - 1}\sum_{j=0}^{l_m} u_n(s_i^m, y_j^m)  \ind_{(s_i^m, s_{i+1}^{m}]}(s)\ind_{U_{j}^m}(y).
\end{equation*}
From \eqref{u_n uni cty}, we have
\begin{equation*}
    u_{n,m}(s,y) \xrightarrow[m \to \infty]{} u_n(s,y), \quad \text{a.s.}
\end{equation*}
for any $(s,y) \in [0,T] \times \mathbb{B}_{T+1}(x)$, and hence the dominated convergence theorem yields that 
\begin{equation*}
    \lim_{m \to \infty}\lVert \sigma'(u_{n,m}(s,y)) - \sigma'(u_{n}(s,y)) \rVert_q = 0
\end{equation*}
for any $(s,y) \in [0,T] \times \mathbb{B}_{T+1}(x)$ and $q \in [1,\infty)$ (as $\sigma'$ is a bounded continuous function).
Thus, using the dominated convergence theorem again, we see that the right-hand side of \eqref{conv xi_(n+1,m) - xi_(n+1)} converges to zero.
Therefore we have shown that $u_{n+1}(t,x) \in \mathbb{D}^{1,p}$ and 
\begin{equation*}
    Du_{n+1}(t,x) = V_{n+1}^{(t,x)}(\cdot, \star) + \Xi_{n+1}(t,x).
\end{equation*}
Since we can take $G_{n+1}^{(t,x)}(s,y)\sigma'(\widetilde{u_{n}}(s,y)) \widetilde{D^{e_k}u_{n}}(s,y)$ as a predictable modification of $\langle DV_{n+1}^{(t,x)}(s,y), e_k \rangle_{\mathcal{H}_T}$, we see that statement (ii) holds for $n+1$.

Next, we show statements (iii) and (iv) for $n+1$.
From Proposition \ref{Prop L^2(Omega H_T)} and \eqref{G_nG_ngamma uniform estimate}, it is easily seen that $\lVert V_{n+1}^{(t,x)} \rVert_{L^p(\Omega;\mathcal{H}_T)}^2$ is dominated by 
\begin{equation*}
    \sup_{(\tau, \zeta) \in [0,T] \times \R^d}\lVert \sigma(u_n(\tau,\zeta)) \rVert_p^2 \Theta(n+1,T)^2\int_{\mathbb{B}_{T+1}^2}\gamma(y-z)dydz,
\end{equation*}
which is finite by Proposition \ref{Prop u_n property}. 
Moreover, a similar computation as in \eqref{ine xi_(n+1,m) - xi_(n+1)} using \eqref{G_nG_ngamma uniform estimate} gives
\begin{equation*}
    \lVert \Xi_{n+1}(t,x) \rVert_{L^p(\Omega;\mathcal{H}_T)}^2 
    \leq \sigma_{\mathrm{Lip}}^2 \sup_{(\tau, \zeta) \in [0,T] \times \R^d}\lVert Du_n(\tau,\zeta) \rVert_{L^p(\Omega;\mathcal{H}_T)}^2 \Theta(n+1,T)^2\int_{\mathbb{B}_{T+1}^2}\gamma(y-z)dydz < \infty.
\end{equation*}
We thus obtain $\sup_{(t,x) \in [0,T] \times \R^d} \lVert Du_{n+1}(t,x) \rVert_{L^p(\Omega;\mathcal{H}_T)} < \infty$.
We now turn to the $L^p(\Omega;\mathcal{H}_T)$-continuity of $Du_{n+1}(t,x)$. 
Using Propositions \ref{Prop L^2(Omega H_T)} and \ref{Prop u_n property} again, we see that
\begin{align*}
    &\lVert V_{n+1}^{(t,x)} - V_{n+1}^{(s,y)} \rVert_{L^p(\Omega;\mathcal{H}_T)}^2\\
    &\leq \int_0^T\int_{\R^{2d}} \lVert V_{n+1}^{(t,x)}(r,z) - V_{n+1}^{(s,y)}(r,z) \rVert_p \lVert V_{n+1}^{(t,x)}(r,z') - V_{n+1}^{(s,y)}(r,z') \rVert_p \gamma(z-z')dzdz'dr\\
    &\lesssim_n \int_0^T\int_{\R^{2d}}|G_{n+1}^{(t,x)}(r,z) - G_{n+1}^{(s,y)}(r,z)||G_{n+1}^{(t,x)}(r,z') - G_{n+1}^{(s,y)}(r,z')|\gamma(z-z')dzdz'dr\\
    &\xrightarrow[(s,y) \to (t,x)]{} 0,
\end{align*}
where the last step follows from a similar argument that we used to show \eqref{4.1.1}. 
Also, for the term $\Xi_{n+1}(t,x)$, we have
\begin{align*}
    &\lVert \Xi_{n+1}(t,x) - \Xi_{n+1}(s,y) \rVert_{L^p(\Omega;\mathcal{H}_T)}^2\\
    &\leq \int_0^T\int_{\R^{2d}} \lVert DV_{n+1}^{(t,x)}(r,z) - DV_{n+1}^{(s,y)}(r,z) \rVert_p\lVert DV_{n+1}^{(t,x)}(r,z') - DV_{n+1}^{(s,y)}(r,z') \rVert_p\gamma(z-z')dzdz'dr\\
    &\lesssim_{n,\sigma_{\mathrm{Lip}}} \int_0^T\int_{\R^{2d}}|G_{n+1}^{(t,x)}(r,z) - G_{n+1}^{(s,y)}(r,z)||G_{n+1}^{(t,x)}(r,z') - G_{n+1}^{(s,y)}(r,z')|\gamma(z-z')dzdz'dr \xrightarrow[(s,y) \to (t,x)]{} 0.
\end{align*}
Therefore we conclude that $(t,x) \mapsto Du_{n+1}(t,x)$ is $L^p(\Omega;\mathcal{H}_T)$-continuous.

Finally, statements (ii) and (iv) for $n+1$ imply that we can take a predictable modification of the process $\langle Du_{n+1}(t,x), e_k \rangle_{\mathcal{H}_T}$ for each $k \in \mathbb{N}$.
The claim $Du_{n+1}(t,x) \in \mathbb{D}^{1,p}$ and statements (iii) and (iv) for $p \in [1,2)$ automatically follow from the case $p \in [2,\infty)$, and the proof is completed.
\end{proof}

\subsection{Moment bounds for the derivative}
\label{subsection Version of the derivative}
In order to obtain the pointwise moment estimate \eqref{int powise est}, it is first necessary to show that for each $(t,x)$ and $n$, $Du_n(t,x)$ has a version $D_{s,y}u_n(t,x)$ as a random function on $[0,T] \times \R^d$ since $\mathcal{H}_T$ generally contains some tempered distributions.
To do this, let us define successively the stochastic processes $(\{M_n(t,x,s,y) \mid (t,x,s,y) \in ([0,T]\times \R^d)^2\})_{n \geq 1}$ as follows: 
\begin{align}
    M_1(t,x,s,y) &= G_1^{(t,x)}(s,y)\sigma(1), \nonumber \\
    M_{n+1}(t,x,s,y) &= G_{n+1}^{(t,x)}(s,y)\sigma(\widetilde{u_n}(s,y)) 
     + \int_0^t\int_{\R^d}G_{n+1}^{(t,x)}(r,z)\sigma'(\widetilde{u_n}(r,z))\widetilde{M_n}(r,z,s,y)W(dr,dz), \label{M_(n+1)}
\end{align}
where $\widetilde{u_n}$ is a predictable modification of $u_n$ and $\widetilde{M_n}$ is a suitable modification of $M_n$ (see the proof of the next proposition).
Note that the stochastic integral in \eqref{M_(n+1)} is defined for each $(t,x,s,y)$ as the Walsh integral of $G_{n+1}^{(t,x)}(\cdot,\star)\sigma'(\widetilde{u_n}(\cdot,\star))\widetilde{M_n}(\cdot,\star,s,y) \in \mathcal{P}_+$.
We first show the properties of $M_n$ and then check that $M_n(t,x,\cdot,\star)$ is a version of $Du_n(t,x)$.

\begin{Prop}
\label{Prop M_(n+1) property}
Let $p \in [1,\infty)$.
For every integer $n \geq 1$, the following statements hold.
\begin{enumerate}
    \item [\normalfont(i)] The process $\{M_n(t,x,s,y) \mid (t,x,s,y) \in ([0,T]\times \R^d)^2\}$ has a measurable modification and $M_n(t,x,s,y)$ is  $\mathscr{F}_t$-measurable for every $(t,x,s,y) \in ([0,T]\times \R^d)^2$.
    \item [\normalfont(ii)] $M_n(t,x,s,y) = 0$ for any $s \geq t$.
    \item [\normalfont(iii)] $\sup_{(t,x,s,y) \in ([0,T]\times \R^d)^2}\lVert M_n(t,x,s,y)\rVert_p < \infty$.
    \item [\normalfont(iv)] $(t,x,s,y) \mapsto M_n(t,x,s,y)$ is $L^p(\Omega)$-continuous.
\end{enumerate}
\end{Prop}

\begin{proof}
It is easily seen that $M_1(t,x,s,y)$ satisfies (i)-(iv).
We proceed by induction and assume that all statements hold for $M_n(t,x,s,y)$.
Then, for each $(s,y)$, we can take a predictable modification of the process $\{M_n(t, x, s,y) \mid (t,x) \in [0,T] \times \R^d\}$, which we will write as $\widetilde{M_n}^{(s,y)}(t,x)$.
The induction assumption together with Lemma \ref{Lem property of G_k} and \eqref{G_nG_ngamma uniform estimate} implies that $G_{n+1}^{(t,x)}(\cdot,\star)\sigma'(\widetilde{u_n}(\cdot,\star))\widetilde{M_n}^{(s,y)}(\cdot,\star) \in \mathcal{P}_+$ for each $(t,x,s,y)$ and that
\begin{align}
    &\int_0^t\int_{\R^{2d}} |G_{n+1}^{(t,x)}(r,z)G_{n+1}^{(t,x)}(r,z')|\gamma(z-z')\E[|\sigma'(\widetilde{u_n}(r,z))\sigma'(\widetilde{u_n}(r,z'))\widetilde{M_n}^{(s,y)}(r,z)\widetilde{M_n}^{(s,y)}(r,z')|]dzdz'dr \nonumber \\
    &\leq \sup_{(\tau,\zeta,\theta,\eta) \in ([0,T]\times \R^d)^2}\lVert M_n(\tau,\zeta,\theta,\eta)\rVert_2^2\sigma_{\mathrm{Lip}}^2\Theta(n+1, T)^2 T\int_{\mathbb{B}_{T+1}^2}\gamma(z-z')dzdz' < \infty. \label{5a}
\end{align}
It follows that $M_{n+1}(t,x,s,y)$ is well-defined and $\mathscr{F}_t$-measurable for every $(t,x,s,y)$. 
Moreover, it is clear from \eqref{M_(n+1)} that $M_{n+1}$ satisfies (ii). 
Using the Burkholder-Davis-Gundy inequality,  Proposition \ref{Prop u_n property}, and Lemma \ref{Lem property of G_k},  we also see that
\begin{align*}
    \lVert M_{n+1}(t,x,s,y) \rVert_p^2 
    &\leq 2 |G_{n+1}^{(t,x)}(s,y)|^2\lVert \sigma(\widetilde{u_n}(s,y)) \rVert_p^2\\
    &\quad + 2\left \lVert \int_0^t\int_{\R^d}G_{n+1}^{(t,x)}(r,z)\sigma'(\widetilde{u_n}(r,z))\widetilde{M_n}^{(s,y)}(r,z)W(dr,dz) \right\rVert_p^2\\
    &\leq 2\Theta(n+1, T)^2\sup_{(\tau, \zeta)\in [0,T]\times \R^d}\lVert \sigma(u_n(\tau,\zeta)) \rVert_p^2\\
    &\quad + 2 C_p\sup_{(\tau,\zeta,\theta,\eta) \in ([0,T]\times \R^d)^2}\lVert M_n(\tau,\zeta,\theta,\eta)\rVert_p^2\sigma_{\mathrm{Lip}}^2\Theta(n+1,T)^2 T\int_{\mathbb{B}_{T+1}^2}\gamma(z-z')dzdz',
\end{align*}
where $C_p$ is a constant depending on $p$.
Hence $M_{n+1}$ satisfies (iii).

We are now reduced to proving (iv) for $M_{n+1}$ since the $L^p(\Omega)$-continuity implies that $M_{n+1}$ has a measurable modification (\textit{cf.} \cite[Proposition B.1]{MR4017124}).
By Proposition \ref{Prop u_n property} and Lemma \ref{Lem property of G_k}, it is obvious that $(t,x,s,y) \mapsto G_{n+1}(t-s, x-y)\sigma(\widetilde{u_n}(s,y))$ is $L^p(\Omega)$-continuous.
If we let 
\begin{equation}
\label{N_(n+1)}
    N_{n+1}(t,x,s,y) = \int_0^t\int_{\R^d}G_{n+1}(t-r,x-z)\sigma'(\widetilde{u_n}(r,z))\widetilde{M_n}^{(s,y)}(r,z)W(dr,dz),
\end{equation}
then we have for any $(t',x',s',y'), (t,x,s,y) \in ([0,T]\times \R^d)^2$, 
\begin{align*}
    &\lVert N_{n+1}(t',x',s',y') - N_{n+1}(t,x,s,y) \rVert_p^2\\
    &\leq 2 \left\lVert \int_0^T\int_{\R^d} (G_{n+1}(t'-r,x'-z) - G_{n+1}(t-r,x-z))\sigma'(\widetilde{u_n}(r,z))\widetilde{M_n}^{(s',y')}(r,z)W(dr,dz) \right\rVert_p^2\\
    &\quad + 2  \left\lVert \int_0^T\int_{\R^d} G_{n+1}(t-r,x-z)\sigma'(\widetilde{u_n}(r,z))(\widetilde{M_n}^{(s',y')}(r,z) - \widetilde{M_n}^{(s,y)}(r,z)) W(dr,dz) \right\rVert_p^2\\
    &\eqqcolon \mathbf{J_1} + \mathbf{J_2}.
\end{align*}
The Burkholder-Davis-Gundy inequality yields that
\begin{align}
    \mathbf{J_1} 
    &\leq 2C_p\sigma_{\mathrm{Lip}}^2\sup_{(\tau,\zeta,\theta,\eta) \in ([0,T]\times \R^d)^2}\lVert M_n(\tau,\zeta,\theta,\eta)\rVert_p^2 \nonumber \\
    &\quad \times \int_0^T\int_{\R^{2d}}|G_{n+1}(t'-r,x'-z) - G_{n+1}(t-r,x-z)|\gamma(z-z') \nonumber \\
    &\qquad \qquad \times |G_{n+1}(t'-r,x'-z') - G_{n+1}(t-r,x-z')|dzdz'dr \nonumber \\
    &\xrightarrow[(t',x') \to (t,x)]{} 0, \label{5.2.1}
\end{align}
where the last step follows from the same argument we used in Proposition \ref{Prop u_n property} to show \eqref{4.1.1}.
Similarly, the term $\mathbf{J_2}$ can be estimated as 
\begin{align*}
    \mathbf{J_2} 
    &\leq 2C_p \sigma_{\mathrm{Lip}}^2\int_0^T\int_{\R^{2d}}|G_{n+1}(t-r,x-z)G_{n+1}(t-r,x-z')|\gamma(z-z')\\
    &\qquad \qquad \qquad \times \lVert \widetilde{M_n}^{(s',y')}(r,z) - \widetilde{M_n}^{(s,y)}(r,z) \rVert_p \lVert \widetilde{M_n}^{(s',y')}(r,z') - \widetilde{M_n}^{(s,y)}(r,z') \rVert_p dzdz'dr.
\end{align*}
We have
\begin{align*}
    &G_{n+1}(t-r,x-z)\lVert \widetilde{M_n}^{(s',y')}(r,z) - \widetilde{M_n}^{(s,y)}(r,z) \rVert_p\\
    &\leq 2 \sup_{(\tau,\zeta,\theta,\eta) \in ([0,T]\times \R^d)^2}\lVert M_n(\tau,\zeta,\theta,\eta)\rVert_p\Theta(n+1,T)\ind_{\mathbb{B}_{T+1}}(x-z),
\end{align*}
and the induction assumption for $M_n$ implies that for each $(r,z) \in [0,T] \times \R^d$,
\begin{equation*}
    \lim_{(s',y') \to (s,y)}\lVert \widetilde{M_n}^{(s',y')}(r,z) - \widetilde{M_n}^{(s,y)}(r,z) \rVert_p = 0.
\end{equation*}
Thus we can apply the dominated convergence theorem to obtain
\begin{equation}
\label{5.2.2}
    \lim_{(s',y') \to (s,y)} \mathbf{J_2} = 0. 
\end{equation}
Combining \eqref{5.2.1} and \eqref{5.2.2}, we obtain
\begin{equation*}
    \lim_{(t',x',s',y') \to (t,x,s,y)} \lVert N_{n+1}(t',x',s',y') - N_{n+1}(t,x,s,y) \rVert_p =0,
\end{equation*}
which implies the $L^p(\Omega)$-continuity of $(t,x,s,y) \mapsto M_{n+1}(t,x,s,y)$.
\end{proof}

To carry out our analysis, we need more precise bounds for the moments of $M_n$, which we show next.
For every integer $n \geq 1$, let $L_{1,n}$ be the process defined by
\begin{equation*}
    L_{1,n}(t,x,s,y) = G_{n}^{(t,x)}(s,y)\sigma(\widetilde{u_{n-1}}(s,y)).
\end{equation*}
We successively define $L_{k,n}$ for $2 \leq k \leq n$ by
\begin{equation*}
    L_{k,n}(t,x,s,y) = \int_s^t\int_{\R^d}G_n^{(t,x)}(r,z)\sigma'(\widetilde{u_{n-1}}(r,z))\widetilde{L_{k-1, n-1}}(r,z,s,y)W(dr,dz)
\end{equation*}
using an appropriate modification of $L_{k-1,n-1}$ for each $(s,y)$.
As with $M_n$ before, we can check that $L_{k,n}$ is well-defined for each $(t,x,s,y) \in ([0,T] \times \R^d)^2$ and integers $1 \leq k \leq n$.
It follows from \eqref{M_(n+1)} that
\begin{equation*}
    M_{n}(t,x,s,y) = \sum_{k=1}^{n} L_{k,n}(t,x,s,y).
\end{equation*}
Using this expression and the fact that $G_n$ has compact support on $[0,T] \times \R^d$, we can obtain estimates for the moments of $M_n$.
For simplicity, we will use the following notation:
\begin{align*}
    \widetilde{\Theta}(n,T) &= \max_{1 \leq l \leq n}\Theta(l,T), & J_{T} &= \int_{\mathbb{B}_{T+1}^2}\gamma(z-z')dzdz',\\
    \Sigma(n,p,T) &= \sup_{(t,x) \in [0,T] \times \R^d}\lVert \sigma(u_n(t,x)) \rVert_p, & \widetilde{\Sigma}(n,p,T) &= \max_{0 \leq l \leq n-1} \Sigma(l,p,T).
\end{align*}

\begin{Prop}
\label{Prop M_n moment bound}
Let $p \geq 2$.
For every $n \geq 1$ and $(t,x,s,y) \in ([0,T] \times \R^d)^2$, we have
\begin{align*}
    \lVert M_n(t,x,s,y) \rVert_{p} \leq \left( \sum_{k=1}^{\infty} \sqrt{\frac{(C_p\sigma_{\mathrm{Lip}}^2TJ_T)^{k-1}}{(k-1)!}} \right) \widetilde{\Sigma}(n,p,T)(\widetilde{\Theta}(n,T) \lor 1)^n \ind_{\mathbb{B}_{t-s+1}}(x-y),
\end{align*}
where $C_p$ is a $p$-dependent constant appearing in the Burkholder-Davis-Gundy inequality.
\end{Prop}

\begin{proof}
We show by induction on $k$ that for every $n \geq k$ and $(t,x,s,y) \in ([0,T] \times \R^d)^2$, 
\begin{equation}
\label{M_n MB induc ineq}
    \lVert L_{k,n}(t,x,s,y) \rVert_p \leq \sqrt{\frac{(C_p\sigma_{\mathrm{Lip}}^2 (t-s) J_T)^{k-1}}{(k-1)!}}\widetilde{\Sigma}(n,p,T)\widetilde{\Theta}(n,T)^k \ind_{\mathbb{B}_{t-s + \sum_{i=n+1-k}^{n} \frac{1}{a_i}}}(x-y).
\end{equation}
When $k = 1$, it is easily seen that for any $n \geq k = 1$,
\begin{align*}
    \lVert L_{1,n}(t,x,s,y) \rVert_p 
    &= |G_n^{(t,x)}(s,y)|\lVert \sigma(\widetilde{u_{n-1}}(s,y)) \rVert_p\\
    &\leq \widetilde{\Sigma}(n,p,T)\widetilde{\Theta}(n,T)\ind_{\mathbb{B}_{t-s + \frac{1}{a_n}}}(x-y).
\end{align*}
Assuming by induction that \eqref{M_n MB induc ineq} holds for $k$, we will prove it for $k+1$.
For any $n \geq k+1$, the Burkholder-Davis-Gundy inequality, Minkowski's inequality, and \eqref{M_n MB induc ineq} yield that
\begin{align*}
    \lVert L_{k+1,n}(t,x,s,y) \rVert_p^2 
    &\leq C_p \sigma_{\mathrm{Lip}}^2 \int_s^t \int_{\R^{2d}}|G_n^{(t,x)}(r,z)G_n^{(t,x)}(r,z')|\gamma(z-z') \\
    &\qquad \qquad \qquad \times \lVert \widetilde{L_{k,n-1}}(r,z,s,y) \rVert_p \lVert \widetilde{L_{k,n-1}}(r,z',s,y) \rVert_p dzdz'dr\\
    &\leq C_p \sigma_{\mathrm{Lip}}^2 \frac{(C_p\sigma_{\mathrm{Lip}}^2 J_T)^{k-1}}{(k-1)!} \widetilde{\Sigma}(n-1,p,T)^2 \widetilde{\Theta}(n-1,T)^{2k}\\
    &\quad \times \int_s^t (r-s)^{k-1} \int_{\R^{2d}}|G_n^{(t,x)}(r,z)G_n^{(t,x)}(r,z')|\gamma(z-z')\\
    &\qquad \quad \times \ind_{\mathbb{B}_{r-s + \sum_{i=n-k}^{n-1} \frac{1}{a_i}}}(z-y)\ind_{\mathbb{B}_{r-s + \sum_{i=n-k}^{n-1} \frac{1}{a_i}}}(z'-y)dzdz'dr.
\end{align*}
Since 
\begin{align*}
    |G_n(t-r,x-z)|\ind_{\mathbb{B}_{r-s + \sum_{i=n-k}^{n-1} \frac{1}{a_i}}}(z-y) \leq \Theta(n,T) \ind_{\mathbb{B}_{t-s + \sum_{i=n-k}^{n} \frac{1}{a_i}}}(x-y)
\end{align*}
by Lemma \ref{Lem property of G_k}, we have 
\begin{align*}
    &\lVert L_{k+1,n}(t,x,s,y) \rVert_p^2 \\
    &\leq \frac{(C_p\sigma_{\mathrm{Lip}}^2 J_T)^{k}}{(k-1)!} \widetilde{\Sigma}(n-1,p,T)^2 \widetilde{\Theta}(n-1,T)^{2k} \Theta(n,T)^2 \int_s^t (r-s)^{k-1} dr \ind_{\mathbb{B}_{t-s + \sum_{i=n-k}^{n} \frac{1}{a_i}}}(x-y) \\
    &\leq \frac{(C_p\sigma_{\mathrm{Lip}}^2 (t-s) J_T)^{k}}{k!}\widetilde{\Sigma}(n,p,T)^2 \widetilde{\Theta}(n,T)^{2(k+1)}\ind_{\mathbb{B}_{t-s + \sum_{i=n-k}^{n} \frac{1}{a_i}}}(x-y),
\end{align*}
where in the last step we use the trivial estimate $\widetilde{\Sigma}(n-1,p,T) \leq \widetilde{\Sigma}(n,p,T)$.
This shows that \eqref{M_n MB induc ineq} holds for $k+1$, and consequently, \eqref{M_n MB induc ineq} holds for any $k$.
It follows that
\begin{align*}
    \lVert M_n(t,x,s,y) \rVert_p
    &\leq \sum_{k=1}^{n} \lVert L_{k,n}(t,x,s,y) \rVert_p \\
    &\leq \left(\sum_{k=1}^{\infty}\sqrt{\frac{(C_p\sigma_{\mathrm{Lip}}^2 T J_T)^{k-1}}{(k-1)!}}\right) (\widetilde{\Theta}(n,T) \lor 1)^n\widetilde{\Sigma}(n,p,T)\ind_{\mathbb{B}_{t-s + 1}}(x-y),
\end{align*}
and the proof is completed. 
\end{proof}

Let $\widetilde{M_n}(t,x,s,y)$ be a measurable modification of the process $M_n(t,x,s,y)$ that exists by Proposition \ref{Prop M_(n+1) property}. 
It holds that $\widetilde{M_n}(t,x,\cdot,\star) \in L^p(\Omega;\mathcal{H}_T)$ by Propositions \ref{Prop L^2(Omega H_T)} and \ref{Prop M_n moment bound}.
Now we will show that $\widetilde{M_n}(t,x,\cdot,\star)$ is a version of $Du_n(t,x)$.
Combining this and Proposition \ref{Prop M_n moment bound} gives the pointwise moment bounds for $Du_n(t,x)$.
We first provide the following lemma.

\begin{Lem}
\label{Lem Mn varphi modi}
Let $n \geq 1$, $p \in [1,\infty)$, and $\varphi \in C_{\mathrm{c}}^{\infty}([0,T] \times \R^d)$.
The process $\{\langle \widetilde{M_n}(t,x,\cdot, \star), \varphi \rangle_{\mathcal{H}_T} \mid (t,x) \in [0,T] \times \R^d\}$ is $(\mathscr{F}_t)$-adapted and is $L^p(\Omega)$-continuous.
In particular, it has a predictable modification.
\end{Lem}

\begin{proof}
Using Propositions \ref{Prop L^2(Omega H_T)} and \ref{Prop M_n moment bound} and (ii) of Proposition \ref{Prop M_(n+1) property}, we obtain for each $(t,x) \in [0,T] \times \R^d$, 
\begin{equation}
\label{M_n varphi}
    \langle \widetilde{M_n}(t,x,\cdot, \star), \varphi \rangle_{\mathcal{H}_T}
    = \int_0^t\int_{\R^{2d}}\widetilde{M_n}(t,x,s,y)\ind_{\mathbb{B}_{T+1}(x)}(y)\varphi(s,z)\gamma(y-z)dydzds \quad \text{a.s.}.
\end{equation}
Then a simple computation using the dominated convergence theorem yields the $L^p(\Omega)$-continuity.
Moreover, since $(s,y) \mapsto M_n(t,x,s,y)$ is uniformly continuous on the compact set $[0,t] \times \mathbb{B}_{T+1}(x)$, we can take the partitions $\{0=s_0^m < s_1^m< \cdots <s_{k_m}^m = t\}$ of $[0,t]$ and $(U_j^m)_{0 \leq j \leq l_m}$ of $\mathbb{B}_{T+1}(x)$ as before and can approximate \eqref{M_n varphi} in $L^p(\Omega)$ by the finite sum of $\mathscr{F}_t$-measurable random variables.
Consequently, \eqref{M_n varphi} is $\mathscr{F}_t$-measurable, and the proposition follows.
\end{proof}

\begin{Prop}
\label{Prop Du_n M_n version}
Let $n\geq 1$ and $p \in [1,\infty)$. 
For any $(t,x) \in [0,T] \times \R^d$, we have
\begin{equation}
\label{M_n Du_n}
    Du_n(t,x) = \widetilde{M_n}(t,x,\cdot,\star) \quad \text{in $L^p(\Omega;\mathcal{H}_T)$}.
\end{equation}
\end{Prop}

\begin{proof}
It suffices to show that \eqref{M_n Du_n} holds for $p=2$ because we know from Proposition \ref{Prop u_n diff'bility} that $Du_n(t,x) \in L^p(\Omega;\mathcal{H}_T)$ for any $p \in [1, \infty)$.
When $n =1$, it is clear from \eqref{Du_(n)(t,x)} and \eqref{M_(n+1)} that
\begin{equation*}
    Du_1(t,x) = G_1^{(t,x)}(\cdot,\star)\sigma(1) = M_1(t,x,\cdot,\star). 
\end{equation*}
Assuming by induction that \eqref{M_n Du_n} holds for $n$, we will now check that it holds for $n+1$.
To do this, it suffices to check that for any $(t,x) \in [0,T] \times \R^d$,
\begin{equation*}
    \langle Du_{n+1}(t,x), e_k \rangle_{\mathcal{H}_T} = \langle \widetilde{M_{n+1}}(t,x,\cdot,\star), e_k \rangle_{\mathcal{H}_T}, \quad \text{a.s.}
\end{equation*}
for every $k \geq 1$.
Here $(e_k)_{k=1}^{\infty}$ is a complete orthonormal system of $\mathcal{H}_T$, and we can assume by Lemma \ref{Lem dense subspace in H_T} that $(e_k)_{k=1}^{\infty} \subset C_{\mathrm{c}}^{\infty}([0,T] \times \R^d)$.
Below we use the notations $V_{n+1}^{(t,x)}$ and $V_{n+1,m}^{(t,x)}$ used in the proof of Proposition \ref{Prop u_n diff'bility}.

Because of the uniform $L^2(\Omega)$-continuity of $M_{n+1}(t,x,\cdot,\star)$ and $V_{n+1}^{(t,x)}$ on $[0,T] \times \mathbb{B}_{T+1}(x)$ and $M_n$ on $([0,T] \times \mathbb{B}_{T+1}(x))^2$, there exists $\delta_m > 0$ such that
\begin{align}
    \lVert M_{n+1}(t,x,s,y) - M_{n+1}(t,x,s',y') \rVert_2 &< 2^{-m}, \label{5.6.1}\\
    \lVert V_{n+1}^{(t,x)}(s,y) - V_{n+1}^{(t,x)}(s',y') \rVert_2 &< 2^{-m}, \quad \text{and} \label{5.6.2}\\
    \sup_{(r,z) \in [0,T] \times \mathbb{B}_{T+1}(x)}\lVert M_{n}(r,z,s,y) - M_{n}(r,z,s',y') \rVert_2 &< 2^{-m}, \label{5.6.3}
\end{align}
for any $(s,y),(s',y') \in [0,T] \times \mathbb{B}_{T+1}(x)$ with $|(s,y) - (s',y')| < 2^{-m}$.
We take the partitions $\{0=s_0^m < s_1^m< \cdots <s_{k_m}^m = t\}$ of $[0,T]$ and $(U_j^m)_{0 \leq j \leq l_m}$ of $\mathbb{B}_{T+1}(x)$ as in the proof of Proposition \ref{Prop u_n diff'bility} and define
\begin{equation*}
    \widetilde{M_{n+1,m}}(r,z,s,y) = \sum_{i=0}^{k_m-1}\sum_{j=0}^{l_m}\widetilde{M_{n+1}}(r,z,s_i^m,y_j^m)\ind_{(s_i^m,s_{i+1}^m]}(s)\ind_{U_j^m}(y),
\end{equation*}
where $y_j^m$ is an arbitrary point in $U_j^m$.
From \eqref{5.6.1}, it is straightforward to check that
\begin{equation}
\label{Mnm ek}
    \int_0^T\int_{\R^{2d}}\widetilde{M_{n+1,m}}(t,x,s,y)e_k(s,z)\gamma(y-z)dydzds \xrightarrow[m \to \infty]{} \langle \widetilde{M_{n+1}}(t,x,\cdot,\star), e_k \rangle_{\mathcal{H}_T} \quad \text{in $L^2(\Omega)$}.
\end{equation}
We next show that the left-hand side of \eqref{Mnm ek} converges in $L^2(\Omega)$ to 
\begin{equation}
    \langle V_{n+1}^{(t,x)}(\cdot,\star), e_k \rangle_{\mathcal{H}_T} + \int_0^t\int_{\R^d}G_{n+1}^{(t,x)}(r,z)\sigma'(\widetilde{u_n}(r,z))\widetilde{M^{e_k}_n}(r,z)W(dr,dz), \label{5.6.0}
\end{equation}
where $\widetilde{M^{e_k}_n}(r,z)$ is a predictable modification of $\langle \widetilde{M_n}(r,z,\cdot,\star), e_k \rangle_{\mathcal{H}_T}$ that exists by Lemma \ref{Lem Mn varphi modi}.
Since $\widetilde{M_{n+1}}$ is a measurable modification of $M_{n+1}$, the left-hand side of \eqref{Mnm ek} is equal to 
\begin{align}
    &\int_0^T\int_{\R^{2d}}V_{n+1,m}^{(t,x)}(s,y)e_k(s,z)\gamma(y-z)dydzds \nonumber\\
    &+ \int_0^t\int_{\R^d}G_{n+1}^{(t,x)}(r,z)\sigma'(\widetilde{u_n}(r,z))\widetilde{M^{e_k}_{n,m}}(r,z)W(dr,dz), \label{5.6.4}
\end{align}
where 
\begin{equation*}
    \widetilde{M^{e_k}_{n,m}}(r,z) \coloneqq \sum_{i=0}^{k_m-1}\sum_{j=0}^{l_m}\widetilde{M_n}^{(s_i^m,y_j^m)}(r,z)\int_0^T\int_{\R^{2d}}\ind_{(s_i^m,s_{i+1}^m]}(s)\ind_{U_j^m}(y)e_k(s,w)\gamma(y-w)dydwds.
\end{equation*}
From \eqref{5.6.2}, it is again straightforward to show that the first term in \eqref{5.6.4} converges in $L^2(\Omega)$ to $\langle V_{n+1}^{(t,x)}(\cdot,\star), e_k \rangle_{\mathcal{H}_T}$.
For the second term, we have
\begin{align*}
    &\left\lVert \int_0^t\int_{\R^d}G_{n+1}^{(t,x)}(r,z)\sigma'(\widetilde{u_n}(r,z))(\widetilde{M^{e_k}_{n,m}}(r,z) - \widetilde{M^{e_k}_{n}}(r,z))W(dr,dz) \right\rVert_2^2\\
    &\lesssim \int_0^t\int_{\mathbb{B}^2_{T+1}(x)}G_{n+1}^{(t,x)}(r,z)G_{n+1}^{(t,x)}(r,z')\gamma(z-z')\\
    &\qquad \qquad \qquad \times \lVert \widetilde{M^{e_k}_{n,m}}(r,z) - \widetilde{M^{e_k}_{n}}(r,z) \rVert_2\lVert \widetilde{M^{e_k}_{n,m}}(r, z') - \widetilde{M^{e_k}_{n}}(r, z') \rVert_2dzdz'dr\\
    &\lesssim \sup_{(r,z) \in [0,T] \times \mathbb{B}_{T+1}(x)}\lVert \widetilde{M^{e_k}_{n,m}}(r,z) - \widetilde{M^{e_k}_{n}}(r,z) \rVert_2^2.
\end{align*}
We then apply \eqref{5.6.3} to obtain that
\begin{align*}
    &\sup_{(r,z) \in [0,T] \times \mathbb{B}_{T+1}(x)}\lVert \widetilde{M^{e_k}_{n,m}}(r,z) - \widetilde{M^{e_k}_{n}}(r,z) \rVert_2\\
    &\leq \sup_{(r,z) \in [0,T] \times \mathbb{B}_{T+1}(x)}\sum_{i=0}^{k_m-1}\sum_{j=0}^{l_m}\int_{s_i^m}^{s_{i+1}^m}\int_{\R^{d}}\int_{U_j^m} \lVert \widetilde{M_n}^{(s_i^m,y_j^m)}(r,z) - \widetilde{M_n}(r,z,s,y)
    \lVert_2 |e_k(s,w)|\gamma(y-w)dydwds\\
    &\leq 2^{-m}\int_0^T\int_{\R^{2d}}\ind_{\mathbb{B}_{T+1}(x)}(y)|e_{k}(s,w)|\gamma(y-w)dydwds \xrightarrow[m\to \infty]{} 0,
\end{align*}
which shows that the second term in \eqref{5.6.4} converges in $L^2(\Omega)$ to the second term in \eqref{5.6.0}.

On the other hand, Proposition \ref{Prop u_n diff'bility} shows that
\begin{equation*}
    \langle Du_{n+1}(t,x), e_k \rangle_{\mathcal{H}_T} = \langle V_{n+1}^{(t,x)}(\cdot,\star), e_k \rangle_{\mathcal{H}_T} + \int_0^t\int_{\R^d}G_{n+1}^{(t,x)}(r,z)\sigma'(\widetilde{u_n}(r,z))\widetilde{D^{e_k}u_n}(r,z)W(dr,dz).
\end{equation*}
Since 
\begin{equation*}
    \widetilde{D^{e_k}u_n}(r,z) = \langle Du_{n}(r,z), e_k \rangle_{\mathcal{H}_T} = \langle \widetilde{M_n}(r,z,\cdot,\star), e_k \rangle_{\mathcal{H}_T} = \widetilde{M_n^{e_k}}(r,z) \quad \text{a.s.}
\end{equation*}
by induction assumption, \eqref{M_n Du_n} holds for $n+1$ and the proposition follows.
\end{proof}

\subsection{Covariance estimates}
\label{subsection Covariance estimates}
Finally, we estimate the covariance functions of $\sigma(u_n(t,x))$. 
Since $\sigma(u_n(t,x)) \in \mathbb{D}^{1,4}$ by Proposition \ref{Prop u_n diff'bility} and the chain rule, the Clark-Ocone formula (Proposition \ref{Prop ClarkOcone}) shows that
\begin{equation}
    \sigma(u_n(t,x)) = \E[\sigma(u_n(t,x))] + (\pi_{\mathcal{P}_0}[D\sigma(u_n(t,x))])\cdot W, \quad \text{a.s.} \label{su1}
\end{equation}
and 
\begin{align}
    &\sigma(u_n(t_1,x_1))\sigma(u_n(t_2,x_2)) \nonumber \\
    &= \E[\sigma(u_n(t_1,x_1))\sigma(u_n(t_2,x_2))]+ (\pi_{\mathcal{P}_0}[D(\sigma(u_n(t_1,x_1))\sigma(u_n(t_2,x_2)))])\cdot W, \quad \text{a.s.}. \label{su2}
\end{align}
Let us define, for any $(s,y) \in [0,T] \times \R^d$,  
\begin{align*}
    R^{(t,x)}_n(s,y) &= \E[\sigma'(u_n(t,x))M_n(t,x,s,y)|\mathscr{F}_s], \\
    R^{(t_1,t_2,x_1,x_2)}_n(s,y) &= \E[\sigma'(u_n(t_1,x_1))M_n(t_1,x_1,s,y)\sigma(u_n(t_2,x_2))|\mathscr{F}_s] \\
    &\quad + \E[\sigma(u_n(t_1,x_1))\sigma'(u_n(t_2,x_2))M_n(t_2,x_2,s,y)|\mathscr{F}_s].
\end{align*}
It is clear that the processes $\{R^{(t,x)}_n(s,y) \mid (s,y) \in [0,T] \times \R^d \}$ and $\{R^{(t_1,t_2,x_1,x_2)}_n(s,y) \mid (s,y) \in [0,T] \times \R^d \}$ are $(\mathscr{F}_s)$-adapted and $L^2(\Omega)$-continuous, and hence have predictable modifications $\widetilde{R_n}^{(t,x)}$ and $\widetilde{R_n}^{(t_1,t_2,x_1,x_2)}$. 
Moreover, by Proposition \ref{Prop M_n moment bound}, we have $\widetilde{R_n}^{(t,x)}, \widetilde{R_n}^{(t_1,t_2,x_1,x_2)} \in \mathcal{P}_+ \subset \mathcal{P}_0$, and so $\widetilde{R_n}^{(t,x)}, \widetilde{R_n}^{(t_1,t_2,x_1,x_2)} \in L^2(\Omega;\mathcal{H}_T)$.

\begin{Lem}
\label{Lem proj version}
For $i = 1,2$, let $t_i \in [0,T]$ and $x_i \in \R^d$. 
It holds that
\begin{align}
    \pi_{\mathcal{P}_0}[D\sigma(u_n(t_1,x_1))] &= \widetilde{R_n}^{(t_1,x_1)}, \label{5.7.a} \\
    \pi_{\mathcal{P}_0}[D(\sigma(u_n(t_1,x_1))\sigma(u_n(t_2,x_2)))] &= \widetilde{R_n}^{(t_1,t_2,x_1,x_2)}. \label{5.7.b}
\end{align}
in $L^2(\Omega;\mathcal{H}_T)$.
\end{Lem}

\begin{proof}
Because 
\begin{align*}
    &\langle \pi_{\mathcal{P}_0}(D\sigma(u_n(t_1,x_1))), X \rangle_{L^2(\Omega;\mathcal{H}_T)} = \langle D\sigma(u_n(t_1,x_1)), \pi_{\mathcal{P}_0}X \rangle_{L^2(\Omega;\mathcal{H}_T)} \quad \text{and}\\
    &\langle \widetilde{R_n}^{(t_1,x_1)}, X \rangle_{L^2(\Omega;\mathcal{H}_T)}=\langle \pi_{\mathcal{P}_0}\widetilde{R_n}^{(t_1,x_1)}, X \rangle_{L^2(\Omega;\mathcal{H}_T)} = \langle \widetilde{R_n}^{(t_1,x_1)}, \pi_{\mathcal{P}_0}X \rangle_{L^2(\Omega;\mathcal{H}_T)}
\end{align*}
for any $X \in L^2(\Omega;\mathcal{H}_T)$, it suffices to show that
\begin{equation}
\label{5.7.1}
    \langle \pi_{\mathcal{P}_0}(D\sigma(u_n(t_1,x_1))), X \rangle_{L^2(\Omega;\mathcal{H}_T)} = \langle \widetilde{R_n}^{(t_1,x_1)}, X \rangle_{L^2(\Omega;\mathcal{H}_T)}
\end{equation}
for any $X \in \mathcal{P}_0$ for \eqref{5.7.a}. 
Moreover, since $\mathcal{P}_0$ is the closure of $\mathfrak{S}$ in $L^2(\Omega;\mathcal{H}_T)$, we only need to check \eqref{5.7.1} for any $X \in \mathfrak{S}$.
(Recall that $\mathfrak{S}$ is defined in Section \ref{subsection Stochastic integrals} as the set of simple processes.)
If $X \in \mathfrak{S}$, then we apply Proposition \ref{Prop Du_n M_n version} to obtain that
\begin{align*}
    &\langle \pi_{\mathcal{P}_0}(D\sigma(u_n(t_1,x_1))), X \rangle_{L^2(\Omega;\mathcal{H}_T)} 
    = \langle D\sigma(u_n(t_1,x_1)), X \rangle_{L^2(\Omega;\mathcal{H}_T)}\\
    &= \langle \sigma'(u_n(t_1,x_1))\widetilde{M_n}(t_1,x_1,\cdot,\star), X \rangle_{L^2(\Omega;\mathcal{H}_T)}\\
    &= \E\left[ \sigma'(u_n(t_1,x_1))\int_0^T\int_{\R^{2d}} \widetilde{M_n}(t_1,x_1,s,y) X(s,z)\gamma(y-z)dydzds\right]\\
    &= \int_0^T\int_{\R^{2d}} \E[\widetilde{R_n}^{(t_1,x_1)}(s,y)X(s,z)]\gamma(y-z)dydzds\\
    &= \langle \widetilde{R_n}^{(t_1,x_1)}, X \rangle_{L^2(\Omega;\mathcal{H}_T)},
\end{align*}
and \eqref{5.7.a} is proved.
Since 
\begin{align*}
    &D(\sigma(u_n(t_1,x_1))\sigma(u_n(t_2,x_2)))\\ 
    &= \sigma'(u_n(t_1,x_1))Du_n(t_1,x_1)\sigma(u_n(t_2,x_2)) + \sigma(u_n(t_1,x_1))\sigma'(u_n(t_2,x_2))Du_n(t_2,x_2)\\
    &= \sigma'(u_n(t_1,x_1))\widetilde{M_n}(t_1,x_1,\cdot,\star)\sigma(u_n(t_2,x_2)) + \sigma(u_n(t_1,x_1))\sigma'(u_n(t_2,x_2))\widetilde{M_n}(t_2,x_2,\cdot,\star)
\end{align*}
by the chain rule and Proposition \ref{Prop Du_n M_n version}, \eqref{5.7.b} follows from the same argument as above.
\end{proof}

Combining Proposition \ref{Prop M_n moment bound} and Lemma \ref{Lem proj version} yields the following estimates.

\begin{Prop}
\label{Prop cov estimates}
For integers $1 \leq i \leq 4$, let $t_i \in [0,T]$ and $x_i \in \R^d$.
It holds that
\begin{align}
\label{cov es 1}
    &|\mathrm{Cov}(\sigma(u_n(t_1,x_1)), \sigma(u_n(t_2,x_2)))| \nonumber\\
    &\lesssim_{n,T,\sigma_{\mathrm{Lip}}}\int^{t_1 \land t_2}_0\int_{\R^{2d}}\ind_{\mathbb{B}_{t_1 - s + 1}}(x_1 -y)\ind_{\mathbb{B}_{t_2 - s + 1}}(x_2 -z)\gamma(y-z)dydzds
\end{align}
and 
\begin{align}
\label{cov es 2}
    &|\mathrm{Cov}(\sigma(u_n(t_1,x_1))\sigma(u_n(t_2,x_2)), \sigma(u_n(t_3,x_3)))\sigma(u_n(t_4,x_4)))| \nonumber\\
    &\lesssim_{n,T,\sigma_{\mathrm{Lip}}}\sum_{k=1}^2 \sum_{l=3}^4 \int_0^{t_k \land t_l}\int_{\R^{2d}}\ind_{\mathbb{B}_{t_k - s + 1}}(x_k -y)\ind_{\mathbb{B}_{t_l - s + 1}}(x_l -z)\gamma(y-z)dydzds.
\end{align}
\end{Prop}

\begin{proof}
From \eqref{su1} and \eqref{5.7.a}, we have
\begin{align*}
    &|\mathrm{Cov}(\sigma(u_n(t_1,x_1)), \sigma(u_n(t_2,x_2)))|\\
    &= \left|\int_0^T\int_{\R^{2d}}\E[\widetilde{R_n}^{(t_1,x_1)}(s,y)\widetilde{R_n}^{(t_2,x_2)}(s,z)]\gamma(y-z)dydzds\right|\\
    &\leq \sigma_{\mathrm{Lip}}^2\int_0^{t_1\land t_2}\int_{\R^{2d}}\lVert M_n(t_1,x_1,s,y) \rVert_2\lVert M_n(t_2,x_2,s,z) \rVert_2\gamma(y-z)dydzds,
\end{align*}
wherein the last inequality, we used Jensen's inequality for conditional expectations.
Then \eqref{cov es 1} follows by applying Proposition \ref{Prop M_n moment bound}.
Similarly, we have
\begin{align*}
    &|\mathrm{Cov}(\sigma(u_n(t_1,x_1))\sigma(u_n(t_2,x_2)), \sigma(u_n(t_3,x_3)))\sigma(u_n(t_4,x_4)))|\\
    &\leq \sigma_{\mathrm{Lip}}^2 \sup_{(\tau,\zeta) \in [0,T] \times \R^d}\lVert \sigma(u_n(\tau,\zeta)) \rVert_4^2 \\
    &\quad \times \sum_{k=1}^2 \sum_{l=3}^4 \int_0^{t_k \land t_l}\int_{\R^{2d}}\lVert M_n(t_k,x_k,s,y) \rVert_4\lVert M_n(t_l,x_l,s,z) \rVert_4\gamma(y-z)dydzds,
\end{align*}
and hence \eqref{cov es 2} follows from Propositions \ref{Prop u_n property} and \ref{Prop M_n moment bound}.
\end{proof}

\section{Limit of the covariance functions}
\label{section Limit of the covariance functions}
In the remainder of the paper, we assume that $\sigma(1) \neq 0$, and focus on the particular case $\gamma(x) = |x|^{-\beta}, \ \  0 < \beta <2$.
In this section, we identify the limit of the covariance functions of slightly modified $F_R(t)$ and $F_{n,R}(t)$ to prove the central limit theorems.

Recall that in Section \ref{section Introduction}, $F_R(t)$ and $F_{n,R}(t)$ are defined by
\begin{equation*}
    F_R(t) = \int_{\mathbb{B}_R} (U(t,x) - 1)dx \quad \text{and} \quad F_{n,R}(t) = \int_{\mathbb{B}_R}(u_n(t,x) - 1)dx.
\end{equation*}
To be precise, $F_{n,R}(t)$ has to be defined using $\widetilde{u_n}(t,x)$ instead of $u_n(t,x)$.
Let $\psi_R$ be a smooth radial function on $\R^d$ satisfying
\begin{equation*}
    \ind_{\mathbb{B}_R}(x) \leq \psi_R(x) \leq \ind_{\mathbb{B}_{R+1}}(x).
\end{equation*}
For the technical reason that $G$ is a distribution in our case, it is more convenient to consider
\begin{equation*}
    \mathfrak{F}_R(t) \coloneqq \int_{\R^d} (U(t,x) - 1)\psi_R(x)dx \quad \text{and} \quad \mathfrak{F}_{n,R}(t) \coloneqq \int_{\R^d}(\widetilde{u_n}(t,x) - 1)\psi_R(x)dx.
\end{equation*}
For $R >0$, define
\begin{equation*}
    \varphi_{t,R}(s,x) = (G(t-s) \ast \psi_R)(x) \quad \text{and} \quad \varphi_{n,t,R}(s,x) = (G_n(t-s) \ast \psi_R)(x).
\end{equation*}
Since $G(t)$ is a rapidly decreasing distribution with $\supp G(t) \subset \mathbb{B}_t$, we have $\varphi_{t,R}(s, \cdot), \varphi_{n,t,R}(s,\cdot) \in C_0^{\infty}(\mathbb{R}^d)$ and 
\begin{equation*}
    \supp \varphi_{t,R}(s, \cdot) \subset \mathbb{B}_{t-s + R+1}, \quad \supp \varphi_{n,t,R}(s, \cdot) \subset \mathbb{B}_{t - s + 1/a_{n} + R + 1}.
\end{equation*}
A simple computation using Lemma \ref{Lem property of G_k} shows that
\begin{equation}
\label{varphi bound}
    |\varphi_{n,t,R}(s, x)| 
    \lesssim_{n,T}\ind_{\mathbb{B}_{R+T+2}}(x)
\end{equation}
Note that the functions $\varphi_{t,R}(s, \cdot)$ and $\varphi_{n,t,R}(s, \cdot)$ are generally not nonnegative when $d \geq 4$.
We observe that $\mathfrak{F}_R(t)$ and $\mathfrak{F}_{n,R}(t)$ can be expressed in terms of Walsh integrals.

\begin{Prop}
\label{Prop F rep Walsh}
Let $p \in [1,\infty)$.
For any $t \in [0,T]$, we have 
\begin{align}
    \mathfrak{F}_{n,R}(t) &= \int_0^t \int_{\R^d}\varphi_{n,t,R}(s,y)\sigma(\widetilde{u_{n-1}}(s,y))W(ds,dy) \quad \text{in $L^p(\Omega)$}, \label{F_n,R}\\
    \mathfrak{F}_R(t) &= \int_0^t\int_{\R^d}\varphi_{t,R}(s,y)\sigma(U(s,y))W(ds,dy) \quad \text{in $L^2(\Omega)$}. \label{F_R}
\end{align}
\end{Prop}

\begin{proof}
The uniform $L^p(\Omega)$-continuity of $x \mapsto \widetilde{u_n}(t,x)$ on $\mathbb{B}_{R+1}$ and the uniform continuity of $G_n$ imply that there is a sequence of partitions $\{(U_j^m)_{1\leq j \leq k_m} \mid m \in \mathbb{N} \}$ of $\mathbb{B}_{R+1}$ such that 
\begin{align}
    \mathfrak{u}^{R}_{n,m}(t,x) \coloneqq \sum_{j=1}^{k_m}\widetilde{u_n}(t,x_j^m)\ind_{U_j^m}(x) \xrightarrow[m \to \infty]{} \widetilde{u_n}(t,x) \quad \text{in $L^p(\Omega)$} \label{u conv} \\
    \sum_{j=1}^{k_m}\int_{U_j^m}G_{n}(t-s,x_j^m-y)\psi_R(z)dz \xrightarrow[m \to \infty]{} \varphi_{n,t,R}(s,y) \nonumber
\end{align}
for any $(s,y)$ and  $x \in \mathbb{B}_{R+1}$, where $x_j^m$ is an arbitrary point in $U_j^m$. 
Since $\widetilde{u_n}$ is a modification of $u_n$, we can apply the Burkholder-Davis-Gundy inequality and the dominated convergence theorem to obtain that
\begin{align*}
    \int_{\R^d}(\mathfrak{u}^{R}_{n,m}(t,x) -1)\psi_R(x)dx 
    &= \int_0^t \int_{\R^d} \left( \sum_{j=1}^{k_m}\int_{U_j^m}G_{n}(t-s,x_j^m-y)\psi_R(x)dx \right) \sigma(\widetilde{u_{n-1}}(s,y))W(ds,dy) \\
    &\xrightarrow[m \to \infty]{} \int_0^t \int_{\R^d}\varphi_{n,t,R}(s,y)\sigma(\widetilde{u_{n-1}}(s,y))W(ds,dy) \quad \text{in $L^p(\Omega)$}.
\end{align*}
On the other hand,  Minkowski's inequality and \eqref{u conv} yields that 
\begin{equation*}
    \int_{\R^d}(\mathfrak{u}^{R}_{n,m}(t,x) -1)\psi_R(x)dx \xrightarrow[m \to \infty]{} \mathfrak{F}_{n,R}(t) \quad \text{in $L^p(\Omega)$},
\end{equation*}
which shows \eqref{F_n,R}.
Finally, we can check \eqref{F_R} from \eqref{F_n,R} by using Proposition \ref{Prop u_n convergence to U} to show that $\mathfrak{F}_{n,R}(t) \xrightarrow[n \to \infty]{} \mathfrak{F}_{R}(t)$ in $L^2(\Omega)$ and 
\begin{equation*}
    \int_0^t \int_{\R^d}\varphi_{n,t,R}(s,y)\sigma(\widetilde{u_{n-1}}(s,y))W(ds,dy) \xrightarrow[n\to \infty]{}\int_0^t\int_{\R^d}\varphi_{t,R}(s,y)\sigma(U(s,y))W(ds,dy) \quad \text{in $L^2(\Omega)$}.
\end{equation*}
\end{proof}

\begin{Rem}
\label{Rem F rep divergence}
As we observed in Lemma \ref{Lem Skorokhod Walsh}, the divergence operator $\delta$ coincides with the Walsh integral, and therefore $\mathfrak{F}_R(t)$ and $\mathfrak{F}_{n,R}(t)$ can also be expressed by $\delta$.
That is, if we let
\begin{equation*}
    Q_{n,t,R}(s,y) = \varphi_{n,t,R}(s,y) \sigma(\widetilde{u_{n-1}}(s,y)) \quad \text{and} \quad Q_{t,R}(s,y) = \varphi_{t,R}(s,y) \sigma(U(s,y)),
\end{equation*}
then $Q_{n,t,R}, Q_{t,R} \in \mathcal{P}_+ \subset \mathrm{Dom(\delta)}$ and we have
\begin{equation*}
    \mathfrak{F}_{n,R}(t) = \delta(Q_{n,t,R}) \quad \text{and} \quad \mathfrak{F}_R(t) = \delta(Q_{t,R}).
\end{equation*}
\end{Rem}

To determine the limit of the normalized covariance functions of $\mathfrak{F}_R(t)$ and $\mathfrak{F}_{n,R}(t)$, we first provide some technical computations. 
Recall that $\tau_{\beta} = \int_{\mathbb{B}_1^2}|x-y|^{-\beta}dxdy$.

\begin{Lem}
\label{Lem 6.2}
Let $t_1, t_2 \in [0,T]$ and $n \geq 1$. 
If $s \in [0, t_1 \land t_2]$, then we have
\begin{align}
    &\lim_{R\to \infty}\frac{1}{R^{2d-\beta}}\int_{\R^{2d}}\varphi_{t_1,R}(s,y)\varphi_{t_2,R}(s,z)|y-z|^{-\beta}dydz = \tau_{\beta}(t_1-s)(t_2-s), \label{c1} \\
    &\lim_{R\to \infty}\frac{1}{R^{2d-\beta}}\int_{\R^{2d}}\varphi_{n,t_1,R}(s,y)\varphi_{n,t_2,R}(s,z)|y-z|^{-\beta}dydz = \tau_{\beta}(t_1-s)(t_2-s). \label{c2}
\end{align}
\end{Lem}

\begin{proof}
Using the Fourier transform, we see that 
\begin{align*}
    &\frac{1}{R^{2d-\beta}}\int_{\R^{2d}}\varphi_{t_1,R}(s,y)\varphi_{t_2,R}(s,z)|y-z|^{-\beta}dydz\\
    &= \frac{c_\beta}{R^{2d-\beta}}\int_{\R^d}|\F\psi_R(\xi)|^2\F G(t_1-s)(\xi)\overline{\F G(t_2-s)(\xi)}|\xi|^{\beta-d}d\xi\\
    &= c_{\beta}\int_{\R^d}|R^{-d}\F\psi_R(R^{-1}\xi)|^2\F G(t_1-s)(R^{-1}\xi)\F G(t_2-s)(R^{-1}\xi)|\xi|^{\beta-d}d\xi,
\end{align*}
where the last equality follows from the change of variables $\xi \mapsto R^{-1}\xi$.
Set
\begin{align*}
    f_R(\xi) &= c_\beta |R^{-d}\F\psi_R(R^{-1}\xi)|^2\F G(t_1-s)(R^{-1}\xi)\F G(t_2-s)(R^{-1}\xi)|\xi|^{\beta-d}, \\
    g_R(\xi) &= c_\beta T^2|R^{-d}\F\psi_R(R^{-1}\xi)|^2|\xi|^{\beta-d}.
\end{align*}
Clearly, $|f_R(\xi)| \leq g_R(\xi)$. 
A simple computation shows that 
\begin{align*}
    R^{-d}\F\psi_R(R^{-1}\xi) 
    &= \F\ind_{B_1}(\xi) + \int_{\R^d\setminus B_1}e^{-2\pi\sqrt{-1}\xi \cdot x}\psi_R(Rx)dx
    \xrightarrow[R \to \infty]{} \F\ind_{B_1}(\xi),
\end{align*}
and this together with \eqref{Fourier transform of G} implies that
\begin{align*}
    \lim_{R\to \infty}f_R(\xi) = c_\beta |\F\ind_{B_1}(\xi)|^2(t_1-s)(t_2-s)|\xi|^{\beta-d}, \quad
    \lim_{R \to \infty}g_R(\xi) = c_\beta T^2|\F\ind_{B_1}(\xi)|^2|\xi|^{\beta-d}.
\end{align*}
Moreover, because $\psi_R(Rx) \leq \ind_{B_2}(x)$ and 
\begin{equation*}
    \psi_R(Rx) = \ind_{B_1}(x) + \psi_R(Rx)\ind_{B_{1+\frac{1}{R}}\setminus B_1}(x) \xrightarrow[]{R \to \infty} \ind_{B_1}(x)
\end{equation*}
for any $x \in \R^d$, the dominated convergence theorem yields that
\begin{align*}
    \lim_{R \to \infty} \int_{\R^d}g_R(\xi)d\xi 
    &= T^2\lim_{R \to \infty} \int_{\R^{2d}}\psi_R(Ry)\psi_R(Rz)|y-z|^{-\beta}dydz\\
    &= T^2\int_{B_1^2}|y-z|^{-\beta}dydz\\
    &= c_\beta T^2\int_{\R^{d}}|\F\ind_{B_1}(\xi)|^2|\xi|^{\beta-d}d\xi.
\end{align*}
Therefore, we can apply Proposition \ref{Prop general DCT} to obtain 
\begin{align*}
    \lim_{R \to \infty}\int_{\R^d}f_R(\xi)dx 
    &= (t_1-s)(t_2-s)c_\beta\int_{\R^d}|\F\ind_{B_1}(\xi)|^2|\xi|^{\beta-d}d\xi\\
    &= (t_1-s)(t_2-s)\int_{B_1^2}|y-z|^{-\beta}dydz,
\end{align*}
which proves \eqref{c1}. 
Because 
\begin{equation*}
    \F \varphi_{n,t,R}(s,\cdot)(\xi) = \F \psi_R(\xi) \F G_n(t-s)(\xi) \quad \text{and} \quad \lim_{R \to \infty}\mathcal{F}G_n(t)(R^{-1}\xi) = t, 
\end{equation*}
a similar argument above leads to \eqref{c2}.  
\end{proof}

Set $\gamma_0(x) = |x|^{-\beta} \ind_{\mathbb{B}_1}(x)$ for all $x \in \R^d$.
It is easily seen that $\gamma_0 \in L^1(\R^d)$ and 
\begin{equation}
\label{RK decomposition}
    |x|^{-\beta} \leq \gamma_0(x) + 2\langle x \rangle^{-\beta}.
\end{equation}
Recall that, for any $\alpha >0$, $B_{\alpha}$ denotes the Bessel kernel of order $\alpha >0$ and that we defined $b_{\alpha}(x) = (2\pi)^dB_{\alpha}(2\pi x)$ in Section \ref{subsection Some technical tools}.

\begin{Lem}
\label{Lem 6.3}
Let $t_1, t_2 \in [0,T]$, $R \geq 1$, and $n \geq 0$.
It holds that
\begin{align}
    &\int_0^{t_1 \land t_2} \int_{\R^{2d}}(\psi_R \ast b_1)(y)(\psi_R \ast b_1)(z)|y-z|^{-\beta}|\mathrm{Cov}(\sigma(u_n(s,y)), \sigma(u_n(s,z)))|dydzds \nonumber \\
    &\lesssim_{n, T,\beta, \sigma_{\mathrm{Lip}}} R^{2d-2\beta}.\label{6.3.1}
\end{align}
\end{Lem}

\begin{proof}
From \eqref{RK decomposition}, the left-hand side of \eqref{6.3.1} is dominated by 
\begin{align*}
    &\int_0^{t_1 \land t_2}\int_{\R^{2d}}(\psi_R \ast b_1)(y)(\psi_R \ast b_1)(z)\gamma_0(y-z)|\mathrm{Cov}(\sigma(u_n(s,y)), \sigma(u_n(s,z)))|dydzds\\ 
    &+2\int_0^{t_1 \land t_2}\int_{\R^{2d}}(\psi_R \ast b_1)(y)(\psi_R \ast b_1)(z)\langle y-z \rangle^{-\beta}|\mathrm{Cov}(\sigma(u_n(s,y)), \sigma(u_n(s,z)))|dydzds\\ 
    &\eqqcolon \mathbf{J_1} + 2\mathbf{J_2}.
\end{align*}
Because 
\begin{align*}
    |\mathrm{Cov}(\sigma(u_n(s,y)), \sigma(u_n(s,z)))| 
    &\leq 2 \sup_{n \geq 0}\sup_{(t,x) \in [0,T]\times \R^d}\lVert \sigma(u_n(t,x)) \rVert_2^2 < \infty,\\ 
    \lVert \psi_R \ast b_1 \rVert_{L^{\infty}(\R^d)} 
    &\leq \lVert \psi_R \rVert_{L^{\infty}(\R^d)}\lVert b_1 \rVert_{L^1(\R^d)} \leq \lVert b_1 \rVert_{L^1(\R^d)} < \infty,
\end{align*}
by Proposition \ref{Prop u_n uniform bound} and Young's convolution inequality, the term $\mathbf{J_1}$ is estimated by 
\begin{align}
    \mathbf{J_1}
    &\lesssim T \int_{\R^{2d}}(\psi_R \ast b_1)(z)\gamma_0(y-z)dydz\nonumber\\
    &\leq T \lVert b_1 \rVert_{L^1(\R^d)} \lVert \gamma_0 \rVert_{L^1(\R^d)}\lVert \psi_R \rVert_{L^1(\R^d)}\nonumber\\
    &\lesssim_T (R+1)^d. \label{6.3.2}
\end{align}
Next, we evaluate $\mathbf{J_2}$. 
Proposition \ref{Prop cov estimates} and \eqref{RK decomposition} lead to 
\begin{align}
    |\mathrm{Cov}(\sigma(u_n(s,y)), \sigma(u_n(s,z)))| 
    &\lesssim_{n,T, \sigma_{\mathrm{Lip}}} \int_{\R^{2d}}\ind_{\mathbb{B}_{T+1}}(y-w)\ind_{\mathbb{B}_{T+1}}(z-w')\gamma_0(w-w')dwdw' \nonumber \\
    &\quad + \int_{\R^{2d}}\ind_{\mathbb{B}_{T+1}}(y-w)\ind_{\mathbb{B}_{T+1}}(z-w')\langle w-w'\rangle^{-\beta} dwdw', \label{61}
\end{align}
and if we let
\begin{equation*}
    g(x) = \int_{\R^{2d}}\ind_{\mathbb{B}_{T+1}}(x-w)\ind_{\mathbb{B}_{T+1}}(w')\gamma_0(w-w')dwdw',
\end{equation*}
then $g \in L^1(\R^d)$ and $g(y-z)$ equals the first term on the right-hand side of \eqref{61}.
Moreover, Peetre's inequality (Lemma \ref{Lem Peetre inequality}) gives 
\begin{align*}
    &\int_{\R^{2d}}\ind_{\mathbb{B}_{T+1}}(y-w)\ind_{\mathbb{B}_{T+1}}(z-w')\langle w-w'\rangle^{-\beta} dwdw'\\
    &\lesssim_{\beta} \langle y-z \rangle^{-\beta}\int_{\R^{2d}}\ind_{\mathbb{B}_{T+1}}(y-w)\langle y-w\rangle^{-\beta} \ind_{\mathbb{B}_{T+1}}(z-w') \langle z-w'\rangle^{-\beta} dwdw'\\
    &\lesssim_{T,\beta} \langle y-z \rangle^{-\beta}.
\end{align*}
Thus we obtain
\begin{align*}
    \mathbf{J_2} 
    &\lesssim_{n, T,\beta, \sigma_{\mathrm{Lip}}} \int_{\R^{2d}}(\psi_R \ast b_1)(y)(\psi_R \ast b_1)(z)\langle y-z \rangle^{-\beta}g(y-z)dydz\\
    &\quad + \int_{\R^{2d}}(\psi_R \ast b_1)(y)(\psi_R  \ast b_1)(z)\langle y-z \rangle^{-2\beta}dydz\\
    &\eqqcolon \mathbf{J_{21}} + \mathbf{J_{22}}.
\end{align*}
By the same calculation as for $\mathbf{J_1}$, we get 
\begin{equation}
    \mathbf{J_{21}} \lesssim_T (R+1)^d. \label{6.3.3}
\end{equation}
Using the fact that $\psi_R \ast b_1 \in \mathcal{S}(\R^d)$ and $\langle x \rangle^{-2\beta} = \F b_{2\beta}(x)$, we have
\begin{align}
    \mathbf{J_{22}} 
    &= \int_{\R^{d}}|\F \psi_R(\xi)|^2\langle \xi \rangle^{-2} b_{2\beta}(\xi)d\xi \nonumber\\
    &\lesssim \int_{\mathbb{B}_{R+1}^2}\langle y-z \rangle^{-2\beta}dydz \nonumber\\
    &\lesssim (R + 1)^{2d-2\beta}. \label{6.3.4}
\end{align}
Here the last step follows from $2\beta < 4 \leq d$ and the simple estimate
\begin{align*}
    \int_{\mathbb{B}_R^2}\langle y-z \rangle^{\alpha}dydz 
    \lesssim 
    \begin{cases}
        R^{2d + \alpha} &(\alpha > -d)\\
        R^{d}\log R &(\alpha = -d)\\
        R^d &(\alpha < -d)
    \end{cases},
\end{align*}
which holds when $R\geq 2$.
Since we are assuming $R \geq 1$, the proposition now follows by combining \eqref{6.3.2}, \eqref{6.3.3}, and \eqref{6.3.4}.
\end{proof}

Now we can show the following result, which we will use in the proof of Theorems \ref{Thm m1} and \ref{Thm m2}.
Recall that we are assuming $\sigma(1) \neq 0$.

\begin{Prop}
 \label{Prop F_R covariance limit}
Let $t_1, t_2 \in [0,T]$ and $n \geq 1$. It holds that 
\begin{align}
    \lim_{R \to \infty} R^{\beta-2d}\E[\mathfrak{F}_{n,R}(t_1)\mathfrak{F}_{n,R}(t_2)] &= \tau_{\beta}\int_0^{t_1\land t_2}(t_1-s)(t_2-s)\E[\sigma(u_{n-1}(s,0))]^2ds, \label{6.4.1}\\
    \lim_{R \to \infty} R^{\beta-2d}\E[\mathfrak{F}_{R}(t_1)\mathfrak{F}_{R}(t_2)] &= \tau_{\beta}\int_0^{t_1\land t_2}(t_1-s)(t_2-s)\E[\sigma(U(s,0))]^2ds. \label{6.4.2}
\end{align}
In particular, when $t_1 \land t_2 > 0$,  the limit \eqref{6.4.2} is strictly positive, and the limit \eqref{6.4.1} is strictly positive for all sufficiently large $n$.
\end{Prop}

\begin{Rem}\label{Rem rep2}
In the proof of Proposition \ref{Prop F_R covariance limit}, we replace $G(t)$ with $b_1 \in L^1(\R^d)$ by applying the Fourier transform and \eqref{int FG bound}. 
Since $G(t)$ is not a nonnegative distribution in the case $d \geq 4$, it cannot be regarded as a measure on $\R^d$, and consequently, it seems difficult to estimate $L^1(\R^d)$ or $L^{\infty}(\R^d)$-norms of $\varphi_{n,t,R}$.
This replacement makes it possible to use Young's convolution inequality and allows further evaluation, as carried out in the proof of Lemma \ref{Lem 6.3}.
\end{Rem}

\noindent
\textit{Proof of Proposition \ref{Prop F_R covariance limit}.} 
We first show \eqref{6.4.1}.
By \eqref{F_n,R}, we have
\begin{align*}
    &R^{\beta-2d}\E[\mathfrak{F}_{n,R}(t_1)\mathfrak{F}_{n,R}(t_2)] \\
    &= R^{\beta-2d}\int_0^{t_1\land t_2}\int_{\R^{2d}}\varphi_{n,t_1,R}(s,y)\varphi_{n,t_2,R}(s,z)|y-z|^{-\beta}\E[\sigma(\widetilde{u_{n-1}}(s,y))\sigma(\widetilde{u_{n-1}}(s,z))]dydzds\\
    &= R^{\beta-2d}\int_0^{t_1\land t_2} \E[\sigma(u_{n-1}(s,0))]^2 \int_{\R^{2d}}\varphi_{n,t_1,R}(s,y)\varphi_{n,t_2,R}(s,z)|y-z|^{-\beta}dydzds\\
    &\quad + R^{\beta-2d}\int_0^{t_1\land t_2}\int_{\R^{2d}}\varphi_{n,t_1,R}(s,y)\varphi_{n,t_2,R}(s,z)|y-z|^{-\beta}\mathrm{Cov}(\sigma(\widetilde{u_{n-1}}(s,y)), \sigma(\widetilde{u_{n-1}}(s,z)))dydzds\\
    &\coloneqq \mathbf{K_{n,1}} + \mathbf{K_{n,2}},
\end{align*}
where the second equality follows since 
\begin{equation*}
    \E[\sigma(u_{n-1}(s,0))] = \E[\sigma(\widetilde{u_{n-1}}(s,y))]
\end{equation*}
for all $y \in \R^d$ by Lemma \ref{Lem u_n property S}.
Since 
\begin{equation*}
    \sup_{s \in [0,T]}\E[\sigma(u_{n-1}(s,0))]^2 < \infty
\end{equation*}
by Proposition \ref{Prop u_n property}, we deduce from Lemma \ref{Lem 6.2} and the dominated convergence theorem that 
\begin{equation*}
    \lim_{R \to \infty}\mathbf{K_{n,1}} = \tau_{\beta}\int_0^{t_1 \land t_2}(t_1 -s)(t_2 -s)\E[\sigma(u_{n-1}(s,0))]^2ds.
\end{equation*}
Now we evaluate the term $\mathbf{K_{n,2}}$.
Set $\eta_{n-1}(s) = \E[\sigma(u_{n-1}(s,0))]$.
By Proposition \ref{Prop u_n property} and Lemmas \ref{Lem fourier bochner schwartz} and \ref{Lem u_n property S}, there is a nonnegative tempered measure $\mu_s^{\sigma(\widetilde{u_{n-1}}) - \eta_{n-1}}$ such that
\begin{equation*}
    |\cdot|^{-\beta}\E[(\sigma(\widetilde{u_{n-1}}(s,\cdot)) - \eta_{n-1}(s))(\sigma(\widetilde{u_{n-1}}(s,0)) - \eta_{n-1}(s))] = \F \mu_s^{\sigma(\widetilde{u_{n-1}}) - \eta_{n-1}} \quad \text{in $\mathcal{S}_{\C}'(\R^d)$}.
\end{equation*}
Then, by Lemmas \ref{Lem property of FG} and \ref{Lem 6.3}, we use the Fourier transform to obtain
\begin{align*}
    |\mathbf{K_{n,2}}|
    &=R^{\beta-2d}\left|\int_0^{t_1\land t_2}\int_{\R^{d}}|\F \psi_R(\xi)|^2|\F\Lambda_n(\xi)|^2\F G(t_1-s)(\xi)\overline{\F G(t_2 -s)(\xi)}\mu_s^{\sigma(\widetilde{u_{n-1}}) - \eta_{n-1}}(d\xi)ds \right|\\
    &\lesssim_{T} R^{\beta-2d}\int_0^{t_1\land t_2}\int_{\R^{d}}|\F \psi_R(\xi)|^2\langle \xi \rangle^{-2}\mu_s^{\sigma(\widetilde{u_{n-1}}) - \eta_{n-1}}(d\xi)ds\\
    &= R^{\beta-2d}\int_0^{t_1\land t_2}\int_{\R^{2d}}(\psi_R \ast b_1)(y)(\psi_R \ast b_1)(z)|y-z|^{-\beta}\mathrm{Cov}(\sigma(u_{n-1}(s,y)), \sigma(u_{n-1}(s,z)))dydzds\\
    &\lesssim_{T,\beta}\Theta(n-1,T)^2R^{-\beta} \xrightarrow[R \to \infty]{} 0.
\end{align*}
Here the third line follows since $\F \psi_R(\cdot) \langle \cdot \rangle^{-1} \in \mathcal{S}_{\C}(\R^d)$ and so $\psi_R \ast b_1 \in \mathcal{S}(\mathbb{R}^d)$.

Next, we show \eqref{6.4.2}.
The same computation shows that
\begin{align*}
    &R^{\beta-2d}\E[\mathfrak{F}_{R}(t_1)\mathfrak{F}_{R}(t_2)]\\
    &= R^{\beta-2d} \int_0^{t_1\land t_2} \E[\sigma(U(s,0))]^2 \int_{\R^{2d}}\varphi_{t_1,R}(s,y)\varphi_{t_2,R}(s,z)|y-z|^{-\beta}dydzds\\
    &\quad + R^{\beta-2d}\int_0^{t_1\land t_2}\int_{\R^{2d}}\varphi_{t_1,R}(s,y)\varphi_{t_2,R}(s,z)|y-z|^{-\beta}\mathrm{Cov}(\sigma(U(s,y)), \sigma(U(s,z)))dydzds\\
    &\coloneqq \mathbf{K_1} + \mathbf{K_2}.
\end{align*}
Using Lemma \ref{Lem 6.2} and the dominated convergence theorem again, we have
\begin{equation*}
    \lim_{R \to \infty}\mathbf{K_1}
    = \tau_{\beta}\int_0^{t_1 \land t_2}(t_1 -s)(t_2 -s)\E[\sigma(U(s,0))]^2ds
\end{equation*}
and are reduced to proving $\lim_{R \to \infty}\mathbf{K}_2 = 0$.
Note that $U$ satisfies \eqref{U uniform bound} and the property (S).
The same argument as in the evaluation of $\mathbf{K_{n,2}}$ yields
\begin{equation*}
    |\mathbf{K_2}| 
    \lesssim_T R^{\beta-2d}\int_0^{t_1\land t_2}\int_{\R^{2d}}(\psi_R \ast b_1)(y)(\psi_R \ast b_1)(z)|y-z|^{-\beta}\mathrm{Cov}(\sigma(U(s,y)), \sigma(U(s,z)))dydzds.
\end{equation*}
By \eqref{U uniform bound} and Propositions \ref{Prop u_n uniform bound} and \ref{Prop u_n convergence to U}, it is easy to check that 
\begin{equation*}
    \lim_{m\to \infty}\sup_{t \in [0,T]} \sup_{(y,z) \in \mathbb{R}^{2d}} |\mathrm{Cov}(\sigma(U(t,y)),\sigma(U(t,z))) - \mathrm{Cov}(\sigma(u_m(t,y)),\sigma(u_m(t,z)))| = 0.
\end{equation*}
Thus, for any $\varepsilon > 0$, we can find $m({\varepsilon})$ large enough such that
\begin{align}
    |\mathbf{K_2}| 
    &\lesssim_T R^{\beta-2d}\int_0^{t_1\land t_2}\int_{\R^{2d}}(\psi_R \ast b_1)(y)(\psi_R \ast b_1)(z)|y-z|^{-\beta}\mathrm{Cov}(\sigma(u_{m(\varepsilon)}(s,y)), \sigma(u_{m(\varepsilon)}(s,z)))dydzds \nonumber\\
    &\quad +\varepsilon R^{\beta-2d}\int_0^{t_1\land t_2}\int_{\R^{2d}}(\psi_R \ast b_1)(y)(\psi_R \ast b_1)(z)|y-z|^{-\beta}dydz. \nonumber\\
    &\coloneqq \mathbf{K_{21}} + \varepsilon\mathbf{K_{22}}. \nonumber
\end{align}
It follows from Lemma \ref{Lem 6.3} that
\begin{equation}
\label{6.4.3}
    \mathbf{K_{21}} \lesssim_T \Theta(m(\varepsilon),T)^2R^{-\beta}\xrightarrow[R \to \infty]{} 0
\end{equation}
for each fixed $\varepsilon$.
For the term $\mathbf{K_{22}}$, we can apply the Fourier transform to obtain
\begin{align}
    \mathbf{K_{22}} 
    &\leq R^{\beta-2d}c_{\beta}\int_0^{t_1\land t_2}\int_{\R^{d}}|\F \psi_R(\xi)|^2|\xi|^{\beta - d}d\xi \nonumber\\
    &= \int_0^{t_1\land t_2}\int_{\R^{2d}}\psi_R(Ry)\psi_R(Rz)|y-z|^{-\beta}dydz \nonumber\\
    &\leq T\int_{\mathbb{B}_2}|y-z|^{-\beta}dydz. \label{6.4.4}
\end{align}
Combining \eqref{6.4.3} and \eqref{6.4.4}, we obtain $\lim_{R \to \infty}|\mathbf{K_2}| = 0$, and \eqref{6.4.2} is proved.

Finally, the $L^2(\Omega)$-continuity of $(t,x) \mapsto U(t,x)$ implies that there are $\delta > 0$ such that for any $s \in [0,\delta]$, 
\begin{equation*}
    \left|\E[\sigma(U(s, 0))]^2 - \E[\sigma(U(0,0))]^2\right| \leq \frac{|\sigma(1)|^2}{2}.
\end{equation*}
Since $\E[\sigma(U(0,0))]^2 = |\sigma(1)|^2$, it follows that for any $s \in [0,\delta]$, 
\begin{equation}
    \E[\sigma(U(s,0))]^2 \geq \frac{|\sigma(1)|^2}{2} > 0 \label{posi estima}
\end{equation}
Moreover, by \eqref{U uniform bound} and Propositions \ref{Prop u_n uniform bound} and \ref{Prop u_n convergence to U}, we have
\begin{align}
    \lim_{n \to \infty}\sup_{s \in [0,T]}|\E[\sigma(U(s,0))]^2 - \E[\sigma(u_n(s,0))]^2| = 0.
\end{align}
 Thus, using \eqref{posi estima}, we can derive 
\begin{equation}
    \E[\sigma(u_{n-1}(s,0))]^2 \geq \frac{|\sigma(1)|^2}{4} > 0, \qquad s \in [0, \delta],
\end{equation}
for all sufficiently large $n$.
This completes the proof.
\qed 

\vspace{2ex}
The same result also holds for $F_R(t)$. 
\begin{Cor}
\label{Cor nonfrak F_R limit}
Let $t_1, t_2 \in [0,T]$. It holds that
\begin{equation*}
    \lim_{R \to \infty} R^{\beta-2d}\E[F_{R}(t_1)F_{R}(t_2)] = \tau_{\beta}\int_0^{t_1\land t_2}(t_1-s)(t_2-s)\E[\sigma(U(s,0))]^2ds. 
\end{equation*}
In particular, when $t_1 \land t_2 > 0$, the right-hand side is strictly positive.
\end{Cor}

\begin{proof}
By Minkowski's inequality and \eqref{U uniform bound}, 
\begin{equation}
    R^{\frac{\beta}{2} - d}\lVert F_R(t) - \mathfrak{F}_R(t) \rVert_2 \leq \left(\sup_{(\tau,\eta) \in [0,T] \times \R^d}\lVert U(\tau, \eta) \rVert_2 + 1 \right)R^{\frac{\beta}{2} - d} |\mathbb{B}_{R+1}\setminus \mathbb{B}_R|. \label{l0}
\end{equation}
Because $0<\beta<2$ and $|\mathbb{B}_{R+1}\setminus \mathbb{B}_R|$ is a polynomial in $R$ of degree at most $d-1$, the right-hand side above converges to zero as $R \to \infty$.
This and Proposition \ref{Prop F_R covariance limit} show that there exists $R_t$ such that
\begin{equation}
    \sup_{R \geq R_t}R^{\frac{\beta}{2} - d}\lVert F_R(t) \rVert_2 \leq \sup_{R \geq R_t}R^{\frac{\beta}{2} - d}\left(\lVert F_R(t) - \mathfrak{F}_R(t) \rVert_2 + \lVert \mathfrak{F}_R(t) \rVert_2\right) < \infty. \label{l1}
\end{equation}
From \eqref{l0} and \eqref{l1}, it can be shown that
\begin{equation*}
    \lim_{R \to \infty} R^{\beta - 2d}|\E[F_R(t_1)F_R(t_2)] - \E[\mathfrak{F}_R(t_1)\mathfrak{F}_R(t_2)]| = 0,
\end{equation*}
and the proof is complete.
\end{proof}

\begin{Rem}
In this section, the condition $\sigma(1) \neq 0$ is used only to guarantee the positivity of the limits of the covariances.
The converse is also true: If $\sigma(1) = 0$, then the approximation scheme \eqref{u_(n+1)} yields that $u_n(t,x) \equiv 1$ and $U(t,x) \equiv 1$, and so $\mathfrak{F}_{n,R}(t) = \mathfrak{F}_R(t) = F_R(t) = 0$.
See \cite[Lemma 3.4]{MR4164864} for more details.
\end{Rem}

\section{Proof of Theorem \ref{Thm m1}}
\label{section Proof of Theorem m1}
The first part of the Theorem \ref{Thm m1} follows from Proposition \ref{Prop F_R covariance limit}.
Now, let us assume $t \in (0,T]$. 
Recall that $\sigma_{R}(t) = \sqrt{\mathrm{Var}(F_R(t))}$.
Let us write 
\begin{equation*}
     \mathfrak{v}_R(t) \coloneqq \sqrt{\mathrm{Var}(\mathfrak{F}_{R}(t))}\quad \text{and} \quad \mathfrak{v}_{n,R}(t) \coloneqq \sqrt{\mathrm{Var}(\mathfrak{F}_{n,R}(t))}.
\end{equation*}
Let $n$ be large enough so that the limit \eqref{6.4.1} is positive when $t_1 = t_2 = t$.
By Proposition \ref{Prop F_R covariance limit} and Corollary \ref{Cor nonfrak F_R limit}, we can find $\widetilde{R}_t \geq 1$ for every $t \in (0,T]$ such that $\mathfrak{v}_{n,R}(t) > 0$, $\mathfrak{v}_{R}(t) > 0$, and $\sigma_{R}(t) > 0$ for any $R \geq \widetilde{R}_t$ and that
\begin{equation}
\label{var finite order cond}
    \sup_{R \geq \widetilde{R}_t}\frac{R^{2d-\beta}}{\mathfrak{v}^2_{n,R}(t)} + \sup_{R \geq \widetilde{R}_t}\frac{R^{2d-\beta}}{\mathfrak{v}^2_{R}(t)} < \infty.
\end{equation}
Taking into account that the Wasserstein distance $d_{\mathrm{W}}$ can be dominated by the $L^2(\Omega)$ distance, we have 
\begin{equation}
\label{distance bound}
    d_{\mathrm{W}}\left( \frac{F_R(t)}{\sigma_{R}(t)},\mathcal{N}(0,1) \right)
    \leq \left \lVert \frac{F_{R}(t)}{\sigma_{R}(t)} - \frac{\mathfrak{F}_{R}(t)}{\mathfrak{v}_{R}(t)} \right\rVert_2 + \left \lVert  \frac{\mathfrak{F}_{R}(t)}{\mathfrak{v}_{R}(t)} - \frac{\mathfrak{F}_{n,R}(t)}{\mathfrak{v}_{n,R}(t)} \right\rVert_2 +
    d_{\mathrm{W}}\left( \frac{\mathfrak{F}_{n,R}(t)}{\mathfrak{v}_{n,R}(t)},\mathcal{N}(0,1) \right).
\end{equation}
To show the convergence \eqref{Thm m1 conv}, let us start with the following lemma.

\begin{Lem}
\label{Lem first term limit}
For any $t \in (0,T]$, we have
\begin{equation*}
    \lim_{R \to \infty}\left \lVert  \frac{F_{R}(t)}{\sigma_{R}(t)} - \frac{\mathfrak{F}_{R}(t)}{\mathfrak{v}_{R}(t)} \right\rVert_2 = 0.
\end{equation*}
\end{Lem}

\begin{proof}
Noting that $\sigma_{R}(t) = \lVert F_{R}(t) \rVert_2$ and $\mathfrak{v}_{R}(t) = \lVert \mathfrak{F}_R(t) \rVert_2$ hold, we have for any $R \geq \widetilde{R}_t$ that 
\begin{align*}
    \left \lVert \frac{F_{R}(t)}{\sigma_{R}(t)} - \frac{\mathfrak{F}_{R}(t)}{\mathfrak{v}_{R}(t)} \right\rVert_2
    &\leq \frac{1}{\mathfrak{v}_{R}(t)}\lVert F_{R}(t) - \mathfrak{F}_{R}(t) \rVert_2 + \frac{|\sigma_{R}(t) - \mathfrak{v}_{R}(t)|}{\sigma_{R}(t)\mathfrak{v}_{R}(t)}\lVert F_{R}(t) \rVert_2 \nonumber \\
    &\leq \frac{2}{\mathfrak{v}_{R}(t)}\lVert F_{R}(t) - \mathfrak{F}_{R}(t) \rVert_2 \xrightarrow[R \to \infty]{} 0,
\end{align*}
where the last limit follows from \eqref{l0} and \eqref{var finite order cond}.
\end{proof}

Next, we show that for a sufficiently large $n$, the $L^2(\Omega)$ norm of the error in replacing $\mathfrak{F}_R(t)$ with $\mathfrak{F}_{n,R}(t)$ is negligible uniformly for large $R$.

\begin{Lem}
\label{Lem error limit}
Let $t \in (0,T]$ and $\widetilde{R}_t > 0$ be as above.
It holds that
\begin{equation*}
    \lim_{n \to \infty} \sup_{R \geq \widetilde{R}_t} \left \lVert  \frac{\mathfrak{F}_{R}(t)}{\mathfrak{v}_{R}(t)} - \frac{\mathfrak{F}_{n,R}(t)}{\mathfrak{v}_{n,R}(t)} \right\rVert_2  = 0.
\end{equation*}
\end{Lem}

\begin{proof}
Let $R \geq \widetilde{R}_t$.
As in Lemma \ref{Lem first term limit}, we have
\begin{equation}
    \left \lVert \frac{\mathfrak{F}_{R}(t)}{\mathfrak{v}_{R}(t)} - \frac{\mathfrak{F}_{n,R}(t)}{\mathfrak{v}_{n,R}(t)} \right\rVert_2^2
    \leq \frac{4}{\mathfrak{v}^2_{R}(t)}\lVert \mathfrak{F}_{R}(t) - \mathfrak{F}_{n,R}(t) \rVert_2^2. \label{7.1.0}
\end{equation}
 From \eqref{F_n,R} and \eqref{F_R}, we have $\lVert \mathfrak{F}_{R}(t) - \mathfrak{F}_{n,R}(t) \rVert_2^2 \leq 2\mathbf{L_1} + 2 \mathbf{L_2}$, where
\begin{align*}
    \mathbf{L_1} 
    &= \int_0^t\int_{\R^{2d}}\varphi_{n,t,R}(s,y)\varphi_{n,t,R}(s,z)|y-z|^{-\beta}\\
    &\qquad \times \E[(\sigma(U(s,y)) - \sigma(\widetilde{u_{n-1}}(s,y)))(\sigma(U(s,z)) - \sigma(\widetilde{u_{n-1}}(s,z)))]dydzds,\\
    \mathbf{L_2}
    &= \int_0^t\int_{\R^{2d}}(\varphi_{t,R}(s,y) - \varphi_{n,t,R}(s,y))(\varphi_{t,R}(s,z) - \varphi_{n,t,R}(s,z))|y-z|^{-\beta}\\
    &\qquad \times \E[\sigma(U(s,y))\sigma(U(s,z))]dydzds.
\end{align*}
By Lemmas \ref{Lem fourier bochner schwartz} and \ref{Lem V_n property S}, we can apply the Fourier transform to obtain that
\begin{align*}
    \mathbf{L_1} 
    &= \int_0^t\int_{\R^d}|\F \psi_R(\xi)|^2|\F \Lambda_n(\xi)|^2|\F G(t-s)(\xi)|^2 \mu_s^{\sigma(U)-\sigma(\widetilde{u_{n-1}})}(d\xi)ds\\
    &\lesssim_T \int_0^t\int_{\R^d}|\F \psi_R(\xi)|^2 \mu_s^{\sigma(U)-\sigma(\widetilde{u_{n-1}})}(d\xi)ds\\
    &= \int_0^t\int_{\R^{2d}}\psi_R(y)\psi_R(z)\E[(\sigma(U(s,y)) - \sigma(\widetilde{u_{n-1}}(s,y)))(\sigma(U(s,z)) - \sigma(\widetilde{u_{n-1}}(s,z)))]|y-z|^{-\beta}dydzds\\
    &\leq \sigma_{\mathrm{Lip}}^2 \sup_{(\tau,\zeta)\in [0,T]\times \R^d} \lVert U(\tau,\zeta) - u_{n-1}(\tau,\zeta) \rVert_2^2 \int_{\mathbb{B}_{R+1}^2} |y-z|^{-\beta}dydz,
\end{align*}
and hence 
\begin{align}
    \sup_{R \geq \widetilde{R}_t}\left(\frac{1}{\mathfrak{v}^2_{R}(t)} \mathbf{L_1} \right)
    &\lesssim \sup_{(\tau,\zeta) \in [0,T]\times \R^d} \lVert U(\tau,\zeta) - u_{n-1}(\tau,\zeta) \rVert_2^2 \sup_{R \geq \widetilde{R}_t}\frac{R^{2d-\beta}}{\mathfrak{v}^2_{R}(t)} \nonumber \\
    &\xrightarrow[n \to \infty]{} 0 \label{7.1.1}
\end{align}
by Proposition \ref{Prop u_n convergence to U}. 
Similarly, $\mathbf{L_2}$ is evaluated as
\begin{align*}
    \mathbf{L_2} 
    &= \int_0^t\int_{\R^d}|\F\varphi_{t,R}(s)(\xi) - \F\varphi_{n,t,R}(s)(\xi)|^2\mu_s^{\sigma(U)}(d\xi)ds\\
    &\lesssim_T \int_0^t\int_{\R^d}|\F \psi_R(\xi)|^2|1-\F \Lambda_n(\xi)|^2\langle \xi \rangle^{-2} \mu_s^{\sigma(U)}(d\xi)ds.
\end{align*}
Let $0 < \varepsilon < 1$ and let $K_{\varepsilon} = \left\{ \xi \in \R^d \mid \langle \xi \rangle^{-2} \geq \varepsilon \right\}$. 
Then we see that
\begin{align*}
    \mathbf{L_2} 
    &\lesssim 4\varepsilon \int_0^t\int_{\R^d \setminus K_{\varepsilon}}|\F \psi_R(\xi)|^2\mu_s^{\sigma(U)}(d\xi)ds + \int_0^t\int_{K_{\varepsilon}} |\F \psi_R(\xi)|^2|1-\F \Lambda_n(\xi)|^2\mu_s^{\sigma(U)}(d\xi)ds\\
    &\leq \left(4\varepsilon + \sup_{x \in K_{\varepsilon}} |1-\F \Lambda_n(x)|^2 \right) \int_0^t \int_{\mathbb{B}_{R+1}}\E[\sigma(U(s,y))\sigma(U(s,z))]|y-z|^{-\beta}dydzds\\
    &\lesssim_T \left(4\varepsilon + \sup_{x \in K_{\varepsilon}} \left|1-\F\Lambda \left(\frac{x}{a_n}\right) \right|^2 \right) R^{2d-\beta},
\end{align*}
where we used \eqref{U uniform bound} in the last step.
Since $\lim_{y \to 0}\F\Lambda(y) = 1$ and $K_{\varepsilon}$ is compact, we have
\begin{equation*}
    \lim_{n \to \infty}\sup_{x \in K_{\varepsilon}} \left|1-\F\Lambda \left(\frac{x}{a_n}\right) \right|^2 = 0,
\end{equation*}
and therefore, 
\begin{align}
    \sup_{R \geq \widetilde{R}_t}\left(\frac{1}{\mathfrak{v}^2_{R}(t)} \mathbf{L_2} \right)
    &\lesssim \left(4\varepsilon + \sup_{x \in K_{\varepsilon}} \left|1-\F\Lambda \left(\frac{x}{a_n}\right) \right|^2 \right)\sup_{R \geq \widetilde{R}_t}\frac{R^{2d-\beta}}{\mathfrak{v}^2_{R}(t)} \nonumber \\
    &\xrightarrow[n \to \infty]{} 4\varepsilon \sup_{R \geq \widetilde{R}_t}\frac{R^{2d-\beta}}{\mathfrak{v}^2_{R}(t)}. \label{7.1.2}
\end{align}
The lemma now follows by combining \eqref{7.1.0}, \eqref{7.1.1}, and \eqref{7.1.2}. 
\end{proof}

Finally, we need to show that
\begin{equation*}
    \lim_{R \to \infty}d_{\mathrm{W}}\left( \frac{\mathfrak{F}_{n,R}(t)}{\mathfrak{v}_{n,R}(t)},\mathcal{N}(0,1) \right) = 0.
\end{equation*}
To do this, using Proposition \ref{Prop Wasserstein bound}, we need the following two lemmas.
Recall (\textit{cf.} Remark \ref{Rem F rep divergence}) that $Q_{n,t,R}(s,y) = \varphi_{n,t,R}(s,y)\sigma(\widetilde{u_{n-1}}(s,y))$. 
We define $\mathcal{M}_{1,t,R}(s,y) = Q_{1,t,R}(s,y)$ and
\begin{align}
    \mathcal{M}_{n+1,t,R}(s,y) = Q_{n+1,t,R}(s,y) + \int_0^t\int_{\R^d}\varphi_{n+1,t,R}(r,z)\sigma'(\widetilde{u_n}(r,z))\widetilde{M_n}^{(s,y)}(r,z)W(dr,dz).  \label{cal M}
\end{align}

\begin{Lem}
\label{Lem DF vers}
Let $n \geq 1$ and $p \in [1,\infty)$.
Let $(e_k)_{k=1}^{\infty}$ be a complete orthonormal system of $\mathcal{H}_T$.
For any $t \in [0,T]$, we have $\mathfrak{F}_{n,R}(t) \in \mathbb{D}^{1,p}$ and 
\begin{align}
    &D\mathfrak{F}_{n,R}(t) \nonumber\\
    &= Q_{n,t,R}(\cdot,\star) + \sum_{k=1}^{\infty} \left( \int_0^t \int_{\R^d}\varphi_{n,t,R}(s,y)\sigma'(\widetilde{u_{n-1}}(s,y)) \widetilde{D^{e_k}u_{n-1}}(s,y) W(ds,dy) \right) e_k, \quad \text{a.s.}. \label{DF_nR(t)} 
\end{align}
Moreover, the process $\mathcal{M}_{n,t,R}(s,y)$ has a measurable modification $\widetilde{\mathcal{M}_{n,t,R}}(s,y)$ and it holds that
\begin{equation}
\label{DF M}
    D\mathfrak{F}_{n,R}(t) = \widetilde{\mathcal{M}_{n,t,R}}(\cdot,\star) \quad \text{in $L^p(\Omega;\mathcal{H}_T)$}.
\end{equation}
\end{Lem}
\begin{proof}
The proof is similar to the argument for $u_n(t,x)$ in Section \ref{section Malliavin derivative of the approximation sequence}, so we give only a brief outline.
In the proof of Proposition \ref{Prop F rep Walsh}, we have shown that
\begin{equation*}
    \int_{\R^d}(\mathfrak{u}^{R}_{n,m}(t,x) -1)\psi_R(x)dx \xrightarrow[m \to \infty]{} \mathfrak{F}_{n,R}(t) \quad \text{in $L^p(\Omega)$},
\end{equation*}
where $\mathfrak{u}^{R}_{n,m}$ was defined in \eqref{u conv}.
By Proposition \ref{Prop u_n diff'bility}, it is clear that the left-hand side above belongs to $\mathbb{D}^{1,p}$ and that 
\begin{align*}
    &D\int_{\R^d}(\mathfrak{u}^{R}_{n,m}(t,x) -1)\psi_R(x)dx\\
    &= \sum_{j=1}^{k_m}G_{n}^{(t,x_j^m)}(\cdot,\star)\sigma(\widetilde{u_{n-1}}(\cdot,\star))\int_{U_j^m}\psi_{R}(x)dx\\
    &\quad + \sum_{k=1}^{\infty} \left( \int_0^t \int_{\R^d} \left(\sum_{j=1}^{k_m} G_{n}^{(t,x_j^m)}(s,y)\int_{U_j^m}\psi_{R}(x)dx\right)\sigma'(\widetilde{u_{n-1}}(s,y)) \widetilde{D^{e_k}u_{n-1}}(s,y) W(ds,dy) \right) e_k.
\end{align*}
It is then straightforward to check that the right-hand side above converges in $L^p(\Omega;\mathcal{H}_T)$ to \eqref{DF_nR(t)} using Proposition \ref{Prop L^2(Omega H_T)}.
Therefore, by \cite[Lemma 1.5.3]{MR2200233}, we have $\mathfrak{F}_{n,R}(t) \in \mathbb{D}^{1,p}$ and \eqref{DF_nR(t)} holds.

Let $\mathcal{N}_{n+1,t,R}(s,y)$ be the stochastic integral on the right-hand side of \eqref{cal M}.
Then, by induction on $n$ as in the proof of Proposition \ref{Prop M_(n+1) property}, we can show that
$\mathcal{N}_{n,t,R}(s,y)$ has a measurable modification $\widetilde{\mathcal{N}_{n,t,R}}(s,y)$.
This shows the existence of a measurable modification of $\mathcal{M}_{n,t,R}$ since the process $Q_{n,t,R}(s,y)$ is predictable and hence measurable.
Finally, using \eqref{M_n Du_n}, we can show \eqref{DF M} by the same argument as in the proof of Proposition \ref{Prop Du_n M_n version}.
\end{proof}

\begin{Lem}
\label{Lem deri var estimate}
Let $t_1,t_2 \in (0,T]$, $n \geq 1$, and $R > T+2$. It holds that 
\begin{equation*}
    \mathrm{Var}(\langle D\mathfrak{F}_{n,R}(t_1), Q_{n,t_2,R} \rangle_{\mathcal{H}_T} ) \lesssim_{n,T,\beta, \sigma_{\mathrm{Lip}}} R^{4d-3\beta}.
\end{equation*}
\end{Lem}
\begin{proof}
We follow the argument in \cite[Section 3.2]{MR4290504}.
By Lemma \ref{Lem DF vers}, we can take $\widetilde{\mathcal{M}_{n,t,R}}(s,y)$ to be $Q_{n,t,R}(s,y) + \widetilde{\mathcal{N}_{n,t,R}}(s,y)$.
Then, by Lemma \ref{Lem DF vers} and Proposition \ref{Prop L^2(Omega H_T)}, we have almost surely
\begin{align*}
   \langle DF_{n,R}(t_1), Q_{n,t_2,R} \rangle_{\mathcal{H}_T}
   &= \langle Q_{n,t_1,R}, Q_{n,t_2,R} \rangle_{\mathcal{H}_T} + \langle \widetilde{\mathcal{N}_{n,t_1,R}}, Q_{n,t_2,R} \rangle_{\mathcal{H}_T}\\
   &= \int_0^T\int_{\R^{2d}}Q_{n,t_1,R}(s,y)Q_{n,t_2,R}(s,y')|y-y'|^{-\beta}dydy'ds\\
   &\quad +\int_0^T\int_{\R^{2d}}\widetilde{\mathcal{N}_{n,t_1,R}}(s,y)Q_{n,t_2,R}(s,y')|y-y'|^{-\beta}dydy'ds\\
   &\eqqcolon \mathbf{V_1} + \mathbf{V_2},
\end{align*}
and hence $\mathrm{Var}(\langle D\mathfrak{F}_{n,R}(t_1), Q_{n,t_2,R} \rangle_{\mathcal{H}_T}) \leq 2\mathrm{Var}(\mathbf{V_1}) + 2\mathrm{Var}(\mathbf{V_2})$.

First, we evaluate $\mathrm{Var}(\mathbf{V_1})$.
By Minkowski's inequality, $\mathrm{Var}(\mathbf{V_1})$ is dominated by 
\begin{equation*}
    \left(\int_0^T\left \{ \mathrm{Var}\left(\int_{\R^{2d}}Q_{n,t_1,R}(s,y)Q_{n,t_2,R}(s,y')|y-y'|^{-\beta}dydy' \right) \right \}^{\frac{1}{2}} ds\right)^2,
\end{equation*}
and the variance appearing in the integrand can be expressed as 
\begin{align*}
    &\int_{\R^{4d}}\varphi_{n,t_1,R}(s,y)\varphi_{n,t_2,R}(s,y')|y-y'|^{-\beta}\varphi_{n,t_1,R}(s,z)\varphi_{n,t_2,R}(s,z')|z-z'|^{-\beta}\\
    &\quad \times \mathrm{Cov}(\sigma(u_{n-1}(s,y))\sigma(u_{n-1}(s,y')), \sigma(u_{n-1}(s,z))\sigma(u_{n-1}(s,z')))dydy'dzdz'.
\end{align*}
We then apply \eqref{varphi bound} and Proposition \ref{Prop cov estimates} to obtain that
\begin{align*}
    &\mathrm{Var}\left(\int_{\R^{2d}}Q_{n,t_1,R}(s,y)Q_{n,t_2,R}(s,y')|y-y'|^{-\beta}dydy' \right)\\
    &\lesssim_{n,T,\sigma_{\mathrm{Lip}}} \int_{\R^{6d}}\ind_{\mathbb{B}_{R+T+2}}(y)\ind_{\mathbb{B}_{R+T+2}}(y')|y-y'|^{-\beta}\ind_{\mathbb{B}_{R+T+2}}(z)\ind_{\mathbb{B}_{R+T+2}}(z')|z-z'|^{-\beta} |w-w'|^{-\beta}\\
    &\qquad \qquad \qquad \times (\ind_{\mathbb{B}_{T+1}}(y-w) + \ind_{\mathbb{B}_{T+1}}(y'-w)) (\ind_{\mathbb{B}_{T+1}}(z-w') + \ind_{\mathbb{B}_{T+1}}(z'-w'))dydy'dzdz'dwdw'.
\end{align*}
By the symmetry, we only need to evaluate
\begin{align}
    &\int_{\R^{6d}}\ind_{\mathbb{B}_{R+T+2}}(y)\ind_{\mathbb{B}_{R+T+2}}(y')|y-y'|^{-\beta}\ind_{\mathbb{B}_{R+T+2}}(z)\ind_{\mathbb{B}_{R+T+2}}(z')|z-z'|^{-\beta} \nonumber\\
    &\qquad \times \ind_{\mathbb{B}_{T+1}}(y-w) \ind_{\mathbb{B}_{T+1}}(z-w')|w-w'|^{-\beta}dydy'dzdz'dwdw'. \label{varv1 bound}
\end{align}
From \eqref{RK decomposition}, this integral is dominated by
\begin{align*}
    &\int_{\R^{6d}}\ind_{\mathbb{B}_{R+T+2}}(y)\ind_{\mathbb{B}_{R+T+2}}(y')|y-y'|^{-\beta}\ind_{\mathbb{B}_{R+T+2}}(z)\ind_{\mathbb{B}_{R+T+2}}(z')|z-z'|^{-\beta} \\
    &\qquad \times \ind_{\mathbb{B}_{T+1}}(y-w) \ind_{\mathbb{B}_{T+1}}(z-w')\gamma_0(w-w')dydy'dzdz'dwdw'\\
    & + 2\int_{\R^{6d}}\ind_{\mathbb{B}_{R+T+2}}(y)\ind_{\mathbb{B}_{R+T+2}}(y')|y-y'|^{-\beta}\ind_{\mathbb{B}_{R+T+2}}(z)\ind_{\mathbb{B}_{R+T+2}}(z')|z-z'|^{-\beta} \\
    &\qquad \times \ind_{\mathbb{B}_{T+1}}(y-w) \ind_{\mathbb{B}_{T+1}}(z-w')\langle w-w'\rangle^{-\beta}dydy'dzdz'dwdw'\\
    &\eqqcolon \mathbf{V_{11}} + 2\mathbf{V_{12}}.
\end{align*}
A simple calculation using $R >T+2$ yields
\begin{align}
    \mathbf{V_{11}} 
    &\leq \int_{\R^{2d}}(\ind_{\mathbb{B}_{R+T+2}} \ast \ind_{\mathbb{B}_{T+1}})(w)(\ind_{\mathbb{B}_{R+T+2}} \ast \ind_{\mathbb{B}_{T+1}})(w')\gamma_0(w-w')dwdw' \nonumber\\
    &\quad \times \left(\int_{\mathbb{B}_{2(R+T+2)}}|y'|^{-\beta}dy' \right)\left(\int_{\mathbb{B}_{2(R+T+2)}}|z'|^{-\beta}dz' \right) \nonumber\\
    &\lesssim (R+T+2)^{2d-2\beta}\lVert \ind_{\mathbb{B}_{R+T+2}} \ast \ind_{\mathbb{B}_{T+1}} \rVert_{L^{\infty}(\R^{d})}\lVert \ind_{\mathbb{B}_{R+T+2}} \ast \ind_{\mathbb{B}_{T+1}} \rVert_{L^1(\R^d)}\lVert \gamma_0 \rVert_{L^1(\R^d)} \nonumber\\
    &\lesssim_{T,\beta}R^{3d-2\beta} \label{V11 est}
\end{align}
For $\mathbf{V_{12}}$, a change of variable and the Peetre's inequality (Lemma \ref{Lem Peetre inequality}) yield
\begin{equation*}
    \mathbf{V_{12}} 
    \lesssim_{\beta} \int_{\mathbb{B}_{R+T+2}^4}|y-y'|^{-\beta}|z-z'|^{-\beta}\langle y-z\rangle^{-\beta}dydy'dzdz' \int_{\mathbb{B}_{T+1}^2}\langle w-w'\rangle^{-\beta}dwdw'.
\end{equation*}
Moreover, making the change of variables $(y,y',z,z') \mapsto (Ry,Ry',Rz,Rz')$, we see that
\begin{align}
    \mathbf{V_{12}} 
    &\lesssim_{T,\beta} R^{4d-3\beta} \int_{\mathbb{B}_{1+\frac{T+2}{R}}^4}|y-y'|^{-\beta}|z-z'|^{-\beta}(R^{-2} + |y-z|^2)^{-\frac{\beta}{2}}dydy'dzdz' \nonumber\\
    &\lesssim R^{4d-3\beta} \int_{\mathbb{B}_{2}^4}|y-y'|^{-\beta}|z-z'|^{-\beta}|y-z|^{-\beta}dydy'dzdz'. \label{V12 est}
\end{align}
The second inequality follows since $R>T+2$, and it is easily seen that the last integral is finite.
Therefore combining \eqref{V11 est} and \eqref{V12 est} gives 
\begin{equation*}
    \mathrm{Var}(\mathbf{V_1}) \lesssim_{T,\beta} R^{4d-3\beta}.
\end{equation*}

Next, we evaluate $\mathrm{Var}(\mathbf{V_2})$.
By Minkowski's inequality again, we have
\begin{equation*}
    \mathrm{Var}(\mathbf{V_2})
    \leq \left(\int_0^T\left \{ \mathrm{Var}\left(\int_{\R^{2d}}\widetilde{\mathcal{N}_{n,t_1,R}}(s,y)Q_{n,t_2,R}(s,y')|y-y'|^{-\beta}dydy' \right) \right \}^{\frac{1}{2}} ds\right)^2,
\end{equation*}
where the variance appearing in the integrand is equal to 
\begin{align*}
    &\int_{\R^{4d}}\varphi_{n,t_2,R}(s,y')|y-y'|^{-\beta}\varphi_{n,t_2,R}(s,z')|z-z'|^{-\beta}\\
    &\quad \times \mathrm{Cov}(\mathcal{N}_{n,t_1,R}(s,y)\sigma(u_{n-1}(s,y')), \mathcal{N}_{n,t_1,R}(s,z)\sigma(u_{n-1}(s,z')))dydy'dzdz' \eqqcolon \mathbf{V_{21}}.
\end{align*}
Let 
\begin{equation*}
    \mathfrak{M}_{n,t,R}^{(s,y)}(\tau) = \int_s^\tau \int_{\R^d}\varphi_{n,t,R}(r,z)\sigma'(\widetilde{u_{n-1}}(r,z))\widetilde{M_{n-1}}^{(s,y)}(r,z)W(dr,dz).
\end{equation*}
Then $\{\mathfrak{M}_{n,t,R}^{(s,y)}(\tau) \mid \tau \in [s,t] \}$ is a mean zero continuous square-integrable $(\mathscr{F}_{\tau})$-martingale, and we have $\mathfrak{M}_{n,t_1,R}^{(s,y)}(t_1) = \mathcal{N}_{n,t_1,R}(s,y)$ a.s. for each $(s,y)$ since $\widetilde{M_{n-1}}^{(s,y)}(r,z) = 0$ a.s. for any $r \leq s$ by Proposition \ref{Prop M_(n+1) property}.
It follows that
\begin{equation*}
    \E[\mathcal{N}_{n,t_1,R}(s,y)\sigma(u_{n-1}(s,y'))] = \E[\E[\mathfrak{M}_{n,t_1,R}^{(s,y)}(t_1)|\mathscr{F}_s]\sigma(u_{n-1}(s,y'))] = 0
\end{equation*}
and 
\begin{align*}
    &\E[\mathcal{N}_{n,t_1,R}(s,y)\sigma(u_{n-1}(s,y'))\mathcal{N}_{n,t_1,R}(s,z)\sigma(u_{n-1}(s,z'))]\\
    &= \E[\E[\mathfrak{M}_{n,t_1,R}^{(s,y)}(t_1)\mathfrak{M}_{n,t_1,R}^{(s,z)}(t_1)|\mathscr{F}_s]\sigma(u_{n-1}(s,y'))\sigma(u_{n-1}(s,z'))]\\
    &\leq \sup_{(\tau,\zeta) \in [0,T] \times \R^d} \lVert \sigma(u_{n-1}(\tau,\zeta)) \rVert_4^2\lVert \llbracket \mathfrak{M}_{n,t_1,R}^{(s,y)},\mathfrak{M}_{n,t_1,R}^{(s,z)} \rrbracket_{t_1} \rVert_2.
\end{align*}
The last inequality follows from It\^{o}'s formula, and
$\llbracket \mathfrak{M}_{n,t_1,R}^{(s,y)},\mathfrak{M}_{n,t_1,R}^{(s,z)} \rrbracket_{t_1}$ denotes the cross variation that equals 
\begin{align*}
    &\int_s^{t_1}\int_{\R^{2d}}\varphi_{n,t_1,R}(r,w)\varphi_{n,t_1,R}(r,w')|w-w'|^{-\beta}\\
    &\quad \times\sigma'(\widetilde{u_{n-1}}(r,w))\widetilde{M_{n-1}}^{(s,y)}(r,w)\sigma'(\widetilde{u_{n-1}}(r,w'))\widetilde{M_{n-1}}^{(s,z)}(r,w')dwdw'dr.
\end{align*}
Thus, using again \eqref{varphi bound} and Proposition \ref{Prop M_n moment bound}, we obtain
\begin{align*}
    \mathbf{V_{21}} 
    &\lesssim_{n,T,\sigma_{\mathrm{Lip}}} \int_{\R^{6d}}\ind_{\mathbb{B}_{R+T+2}}(y')|y-y'|^{-\beta}\ind_{\mathbb{B}_{R+T+2}}(z')|z-z'|^{-\beta}
    \ind_{\mathbb{B}_{R+T+2}}(w)\ind_{\mathbb{B}_{R+T+2}}(w')\\
    &\qquad \qquad \qquad \times \ind_{\mathbb{B}_{T+1}}(y-w) \ind_{\mathbb{B}_{T+1}}(z-w')|w-w'|^{-\beta}dydy'dzdz'dwdw'.
\end{align*}
Although this right-hand side looks slightly different from the form of \eqref{varv1 bound}, a similar evaluation as before yields that $\mathbf{V_{21}} \lesssim_{n,T,\beta,\sigma_{\mathrm{Lip}}} R^{4d-3\beta}$, and therefore 
\begin{equation*}
    \mathrm{Var}(\mathbf{V_{2}}) \lesssim_{n,T,\beta,\sigma_{\mathrm{Lip}}} R^{4d-3\beta}.
\end{equation*}
This completes the proof.
\end{proof}

We now complete the proof of Theorem \ref{Thm m1}.

\begin{proof}[Proof of Theorem \ref{Thm m1}]
What is left is to show the convergence \eqref{Thm m1 conv}.
Thanks to Lemma \ref{Lem error limit}, we can take $n$ large enough so that
\begin{equation*}
    \sup_{R \geq \widetilde{R}_t} \left \lVert  \frac{\mathfrak{F}_{R}(t)}{\mathfrak{v}_{R}(t)} - \frac{\mathfrak{F}_{n,R}(t)}{\mathfrak{v}_{n,R}(t)} \right\rVert_2 < \varepsilon
\end{equation*}
for any given $\varepsilon \in (0,1)$.
Let us fix such $n$. 
From \eqref{distance bound} and Lemma \ref{Lem first term limit}, it remains to prove 
\begin{equation}
\label{m1left}
    \lim_{R \to \infty}d_{\mathrm{W}}\left( \frac{\mathfrak{F}_{n,R}(t)}{\mathfrak{v}_{n,R}(t)},\mathcal{N}(0,1) \right) = 0.
\end{equation}
In view of Remark \ref{Rem F rep divergence}, we can apply Proposition \ref{Prop Wasserstein bound}.
This, together with Proposition \ref{Prop F_R covariance limit} and Lemma \ref{Lem deri var estimate}, implies
\begin{align*}
    d_{\mathrm{W}}\left( \frac{\mathfrak{F}_{n,R}(t)}{\mathfrak{v}_{n,R}(t)},\mathcal{N}(0,1) \right) 
    &\leq \sqrt{\frac{2}{\pi}}\frac{\sqrt{\mathrm{Var}(\langle D\mathfrak{F}_{n,R}(t), Q_{n,t,R}\rangle_{\mathcal{H}_T})}}{\mathfrak{v}^2_{n,R}(t)}\\
    &\lesssim_{n,T,\beta,\sigma_{\mathrm{Lip}}} R^{-\frac{\beta}{2}} \xrightarrow[R \to \infty]{} 0,
\end{align*}
and the proof of Theorem \ref{Thm m1} is complete.
\end{proof}

\section{Proof of Theorem \ref{Thm m2}}
\label{section Proof of Theorem m2}
The proof of Theorem \ref{Thm m2} can be divided into two parts.
The first is to show the convergence of the finite-dimensional distributions, which is presented in Section \ref{subsection Convergence of finite-dimensional distributions}. The second is the tightness; we will prove it in Section \ref{Subsection Tightness}.

\subsection{Convergence of finite-dimensional distributions}
\label{subsection Convergence of finite-dimensional distributions}
Let $s,t \in [0,T]$ and define
\begin{equation*}
    \Phi(s,t) =\tau_\beta \int_0^{s\land t}(s-r)(t-r) \E[U(r,0)]^2dr, \quad \Phi_n(s,t) =\tau_\beta \int_0^{s\land t}(s-r)(t-r) \E[u_{n-1}(r,0)]^2dr.
\end{equation*}
By Proposition \ref{Prop u_n convergence to U} and the dominated convergence theorem, it is clear that
\begin{equation}
\label{Phi conv}
    \lim_{n \to \infty} \Phi_n(s,t) = \Phi(s,t).
\end{equation}

\begin{Prop}
\label{Prop conv of fdd}
Fix $m \geq 1$, and set $\mathbf{F}_R = (F_R(t_1), \ldots , F_R(t_m))$ for any $t_1, \ldots , t_m \in [0,T]$.
Let $\mathbf{G} = (G(t_1), \ldots , G(t_m))$ be a centered Gaussian random vector with covariance matrix $(\Phi(t_i,t_j))_{1 \leq i,j \leq m}$. 
Then, as $R \to \infty$, we have
\begin{equation*}
    R^{\frac{\beta}{2}-d}\mathbf{F}_R \xrightarrow[]{d} \mathbf{G},
\end{equation*}
where $\xrightarrow[]{d}$ denotes the convergence in distribution.
\end{Prop}

\begin{proof}
To prove the desired convergence, it suffices to show that
\begin{equation}
\label{8.1.0}
    |\E[h(R^{\frac{\beta}{2}-d}\mathbf{F}_R) - h(\mathbf{G})]| \xrightarrow[R \to \infty]{} 0
\end{equation}
for any $h \in C_{\mathrm{c}}^{\infty}(\mathbb{R}^m)$.
Let $\widetilde{\mathbf{F}}_{n,R} = (\mathfrak{F}_{n,R}(t_1), \ldots , \mathfrak{F}_{n,R}(t_m))$ and $\widetilde{\mathbf{F}}_{R} = (\mathfrak{F}_{R}(t_1), \ldots , \mathfrak{F}_{R}(t_m))$.
By simple estimates, we have
\begin{align}
    |\E[h(R^{\frac{\beta}{2}-d}\mathbf{F}_R)) - h(\mathbf{G})]|
    &\leq |\E[h(R^{\frac{\beta}{2}-d}\mathbf{F}_R)) - h(R^{\frac{\beta}{2}-d}\widetilde{\mathbf{F}}_{R})]| + |\E[h(R^{\frac{\beta}{2}-d}\widetilde{\mathbf{F}}_R)) - h(R^{\frac{\beta}{2}-d}\widetilde{\mathbf{F}}_{n,R})]| \nonumber \\
    &\quad + |\E[h(R^{\frac{\beta}{2}-d}\widetilde{\mathbf{F}}_{n,R})) - h(\mathbf{G})]| \nonumber \\
    &\lesssim_h R^{\frac{\beta}{2}-d} \E[|\mathbf{F}_R - \widetilde{\mathbf{F}}_{R}|] +\sup_{R > 0} \left( R^{\frac{\beta}{2}-d} \E[|\widetilde{\mathbf{F}}_R - \widetilde{\mathbf{F}}_{n,R}|]\right) \nonumber \\
    &\quad+ |\E[h(R^{\frac{\beta}{2}-d}\widetilde{\mathbf{F}}_{n,R}) - h(\mathbf{G})]|. \label{8.1.1}
\end{align}
From \eqref{l0}, we have
\begin{equation}
\label{fl0}
    R^{\frac{\beta}{2}-d}\E[|\mathbf{F}_R - \widetilde{\mathbf{F}}_{R}|] 
    \leq R^{\frac{\beta}{2}-d}\sum_{i=1}^{m}\lVert F_R(t_i) - \mathfrak{F}_R(t_i) \rVert_2 \xrightarrow[R \to \infty]{} 0.
\end{equation}

Fix $\varepsilon \in (0,1)$.
By the same argument as in the proof of Lemma \ref{Lem error limit}, the second term on the right-hand side of \eqref{8.1.1} can be evaluated as
\begin{align*}
    \sup_{R > 0} \left( R^{\frac{\beta}{2}-d} \E[|\widetilde{\mathbf{F}}_R - \widetilde{\mathbf{F}}_{n,R}|]\right)
    &\leq \sup_{R > 0} \left( R^{\frac{\beta}{2}-d} \sum_{i=1}^m\lVert \mathfrak{F}_R(t_i) - \mathfrak{F}_{n,R}(t_i) \rVert_2 \right)\\
    &\lesssim_m \left( \sup_{(\tau, \zeta) \in [0,T] \times \R^d} \lVert U(\tau, \zeta) - u_n(\tau,\zeta) \rVert_2 + 2\sqrt{\varepsilon} + \sup_{x \in K_{\varepsilon}}\left| 1 - \F\Lambda\left(\frac{x}{a_n} \right) \right|  \right)\\
    &\xrightarrow[n \to \infty]{} 2\sqrt{\varepsilon}.
\end{align*}
This together with \eqref{Phi conv} implies that we can take a sufficiently large $n$ such that
\begin{equation}
\label{8.1.2}
    \sup_{R > 0} \left( R^{\frac{\beta}{2}-d} \E[|\widetilde{\mathbf{F}}_R - \widetilde{\mathbf{F}}_{n,R}|]\right) \lesssim \sqrt{\varepsilon} \quad \text{and} \quad \max_{1 \leq i,j \leq m} |\Phi_n(t_i,t_j) - \Phi(t_i,t_j)| < \sqrt{\varepsilon}
\end{equation}
Hereafter, we fix such $n$ and evaluate the third term in \eqref{8.1.1}.
By Remark \ref{Rem F rep divergence}, we can apply Proposition \ref{Prop multivariate stein bound} to obtain
\begin{align*}
    &|\E[h(R^{\frac{\beta}{2}-d}\widetilde{\mathbf{F}}_{n,R}) - h(\mathbf{G})]| \\
    &\leq \frac{m}{2} \max_{1 \leq k,l \leq m} \sup_{x \in \mathbb{R}^m} \left| \frac{\partial^2h}{\partial x_k \partial x_l}(x) \right| \sqrt{\sum_{i,j =1}^m \E[|\Phi(t_i,t_j) - R^{\beta- 2d}\langle D\mathfrak{F}_{n,R}(t_i), Q_{n,t_j,R} \rangle_{\mathcal{H}_T}|^2]}.
\end{align*}
To check that the left-hand side above converges to zero as $R \to \infty$, we only need to show that 
\begin{equation*}
    \lim_{R \to \infty} \E[|\Phi(t_i,t_j) - R^{\beta- 2d}\langle D\mathfrak{F}_{n,R}(t_i), Q_{n,t_j,R} \rangle_{\mathcal{H}_T}|^2] = 0.
\end{equation*}
Because the duality relation \eqref{duality relation} and Proposition \ref{Prop F_R covariance limit} imply that 
\begin{equation*}
    R^{\beta -2d}\E[\langle D\mathfrak{F}_{n,R}(t_i), Q_{n,t_j,R} \rangle_{\mathcal{H}_T}] 
    = R^{\beta -2d} \E[\mathfrak{F}_{n,R}(t_i)\mathfrak{F}_{n,R}(t_j)] \xrightarrow[R \to \infty]{} \Phi_n(t_i,t_j),
\end{equation*}
we derive from Lemma \ref{Lem deri var estimate} that
\begin{align*}
    &\E[|\Phi(t_i,t_j) - R^{\beta- 2d}\langle D\mathfrak{F}_{n,R}(t_i), Q_{n,t_j,R} \rangle_{\mathcal{H}_T}|^2] \\
    &= \Phi(t_i,t_j)^2 - 2\Phi(t_i,t_j)(R^{\beta -2d}\E[\langle D\mathfrak{F}_{n,R}(t_i), Q_{n,t_j,R} \rangle_{\mathcal{H}_T}])\\
    &\quad + R^{2\beta-4d}(\mathrm{Var}(\langle D\mathfrak{F}_{n,R}(t_i), Q_{n,t_j,R} \rangle_{\mathcal{H}_T}) + \E[\langle D\mathfrak{F}_{n,R}(t_i), Q_{n,t_j,R} \rangle_{\mathcal{H}_T}]^2)\\
    &\xrightarrow[R \to \infty]{} |\Phi(t_i,t_j) - \Phi_n(t_i,t_j)|^2,
\end{align*}
and consequently we obtain 
\begin{equation*}
    \limsup_{R \to \infty}|\E[h(R^{\frac{\beta}{2}-d}\widetilde{\mathbf{F}}_{n,R}) - h(\mathbf{G})]| \lesssim_{h,m} \sqrt{\varepsilon}.
\end{equation*}
This, together with \eqref{8.1.1}, \eqref{fl0}, and \eqref{8.1.2}, shows \eqref{8.1.0}, and the proof is completed.
\end{proof}

\subsection{Tightness}
\label{Subsection Tightness}
Tightness in $C([0,T])$ of the process $\{ R^{-(d-\beta/2)}F_R(t) \mid t \in [0,T] \}$ can be obtained by applying the criterion of Kolmogorov-Chentsov (see \textit{e.g.}  \cite[Problem 4.11 of Chapter 2]{KaratzasShreve}).
Since $F_R(0) = 0$ for any $R>0$, it suffices to show the following regularity estimate.

\begin{Prop}
\label{Prop tightness}
For any $0 \leq s < t \leq T$ and $R \geq 1$, we have 
\begin{equation*}
    \E\left[\left|\frac{1}{R^{d-\beta/2}}F_R(t) - \frac{1}{R^{d-\beta/2}}F_R(s)\right|^2\right] \leq C_{d,T,\beta}(t-s)^{2},
\end{equation*}
where $C_{d,T,\beta}$ is a constant depending on $d,T,\beta$. 
\end{Prop}
\begin{proof}
Recall that 
\begin{equation*}
    F_{n,R}(t) = \int_{\mathbb{B}_R}(\widetilde{u_n}(t,x) - 1)dx.
\end{equation*}
We can show that for any $p \in [1,\infty)$,
\begin{equation}
\label{FnR rep walsh}
    {F}_{n,R}(t) = \int_0^t \int_{\R^d}(G_n(t-s) \ast \ind_{\mathbb{B}_R})(y)\sigma(\widetilde{u_{n-1}}(s,y))W(ds,dy) \quad \text{in $L^p(\Omega)$}
\end{equation}
by replacing $\psi_R$ with $\ind_{\mathbb{B}_R}$ in the proof of Proposition \ref{Prop F rep Walsh} and making the same argument.
Moreover, Proposition \ref{Prop u_n convergence to U} implies that $F_{n,R}(t) \xrightarrow[n \to \infty]{} F_{R}(t)$ in $L^2(\Omega)$ for each $R$ and $t$, and hence 
\begin{equation*}
    \mathbb{E}[|F_{R}(t)-F_{R}(s)|^{2}] = \lim_{n \to \infty}\E[|F_{n,R}(t)-F_{n,R}(s)|^{2}].
\end{equation*}
From \eqref{FnR rep walsh} and the isometry property of stochastic integral, we have
\begin{align*}
    &\mathbb{E}[|F_{n,R}(t)-F_{n,R}(s)|^{2}]\\
    &= \mathbb{E}\left[\left(\int_{0}^{t} \int_{\mathbb{R}^{d}}\left((G_n(t-r) \ast \ind_{\mathbb{B}_R})(z)-(G_n(s-r) \ast \ind_{\mathbb{B}_R})(z)\right) \sigma(\widetilde{u_{n-1}}(r,z)) W(dr,dz)\right)^{2}\right]\\
    &= \int_{0}^{t} \int_{\mathbb{R}^{d}}|\mathcal{F}G(t-r)(\xi) - \mathcal{F}G(s-r)(\xi)|^2|\F \Lambda_n(\xi)|^2 |\F \ind_{\mathbb{B}_R}(\xi)|^2 \mu_r^{\sigma(\widetilde{u_{n-1}})}(d\xi) dr\\
    &\leq \int_s^t\int_{\mathbb{R}^d}|\mathcal{F}G(t-r)(\xi)|^2|\F \ind_{\mathbb{B}_R}(\xi)|^2\mu_{r}^{\sigma(\widetilde{u_{n-1}})}(d\xi)dr \\
    &\quad + \int_0^s\int_{\mathbb{R}^d}|\mathcal{F}G(t-r)(\xi) - \mathcal{F}G(s-r)(\xi)|^2 |\F \ind_{\mathbb{B}_R}(\xi)|^2\mu_{r}^{\sigma(\widetilde{u_{n-1}})}(d\xi)dr\\
    &\leq \int_s^t(t-r)^2\int_{\R^d}|\F \ind_{\mathbb{B}_R}(\xi)|^2\mu_r^{\sigma(\widetilde{u_{n-1}})}(d\xi)dr + (t-s)^2\int_0^s\int_{\mathbb{R}^d}|\F \ind_{\mathbb{B}_R}(\xi)|^2\mu_{r}^{\sigma(\widetilde{u_{n-1}})}(d\xi)dr.
\end{align*}
Here the second equality follows since $(G_n(t) \ast \ind_{\mathbb{B}_R}) \in C_{\mathrm{c}}^{\infty}(\mathbb{R}^d)$, and the last inequality follows from Lemma \ref{Lem property of FG}.
A standard computation using Proposition \ref{Prop u_n uniform bound} shows that
\begin{align*}
    \int_{\R^d}|\F\ind_{\mathbb{B}_R}(\xi)|^2\mu_r^{\sigma(\widetilde{u_{n-1}})}(d\xi)
    &= \int_{\R^{2d}}\ind_{\mathbb{B}_R}(z)\ind_{\mathbb{B}_R}(z')|z-z'|^{-\beta}\E[\sigma(\widetilde{u_{n-1}}(r,z))\sigma(\widetilde{u_{n-1}}(r,z'))]dzdz'\\
    &\leq \sup_{n \geq 1}\sup_{(\tau,\eta)\in[0,T]\times \R^d}\norm{\sigma(u_{n-1}(\tau,\eta))}_2^2\int_{\mathbb{B}_{R+1}^2}|z-z'|^{-\beta}dzdz'\\
    &\lesssim R^{2d -\beta}.
\end{align*}
Therefore, we obtain
\begin{equation*}
    \E[|F_R(t) - F_R(s)|^2]  
    \lesssim \left(\int_s^t(t-r)^2dr + (t-s)^2\int_0^sdr\right)R^{2d-\beta}
    \lesssim_T (t-s)^2R^{2d-\beta},
\end{equation*}
which completes the proof.
\end{proof}

As noted at the beginning of this section, the combination of Propositions \ref{Prop conv of fdd} and \ref{Prop tightness} completes the proof of Theorem \ref{Thm m2}.

\bibliographystyle{plain}
\bibliography{CLTsSWE}

\begin{thebibliography}{10}

\bibitem{MR4385408}
Obayda Assaad, David Nualart, Ciprian~A. Tudor, and Lauri Viitasaari.
\newblock Quantitative normal approximations for the stochastic fractional heat equation.
\newblock {\em Stoch. Partial Differ. Equ. Anal. Comput.}, 10(1):223--254, 2022.

\bibitem{BalanZhengLevyHAM}
Raluca Balan and Guangqu Zheng.
\newblock Hyperbolic {A}nderson model with {L}\'evy white noise: spatial ergodicity and fluctuation.
\newblock {\em Trans. Amer. Math. Soc.}, 377(6):4171--4221, 2024.

\bibitem{MR4491503}
Raluca~M. Balan, David Nualart, Llu\'{\i}s Quer-Sardanyons, and Guangqu Zheng.
\newblock The hyperbolic {A}nderson model: moment estimates of the {M}alliavin derivatives and applications.
\newblock {\em Stoch. Partial Differ. Equ. Anal. Comput.}, 10(3):757--827, 2022.

\bibitem{MR4017124}
Raluca~M. Balan, Llu\'{\i}s Quer-Sardanyons, and Jian Song.
\newblock Existence of density for the stochastic wave equation with space-time homogeneous {G}aussian noise.
\newblock {\em Electron. J. Probab.}, 24:Paper No. 106, 43, 2019.

\bibitem{BalanYuanHAM-with-time-independent-noise}
Raluca~M. Balan and Wangjun Yuan.
\newblock Spatial integral of the solution to hyperbolic {A}nderson model with time-independent noise.
\newblock {\em Stochastic Process. Appl.}, 152:177--207, 2022.

\bibitem{CLTforHEwithtimeindependentnoise}
Raluca~M. Balan and Wangjun Yuan.
\newblock Central limit theorems for heat equation with time-independent noise: the regular and rough cases.
\newblock {\em Infin. Dimens. Anal. Quantum Probab. Relat. Top.}, 26(2):Paper No. 2250029, 47, 2023.

\bibitem{MR2267655}
V.~I. Bogachev.
\newblock {\em Measure theory. {V}ol. {I}, {II}}.
\newblock Springer-Verlag, Berlin, 2007.

\bibitem{MR4290504}
Raul Bola\~{n}os Guerrero, David Nualart, and Guangqu Zheng.
\newblock Averaging 2d stochastic wave equation.
\newblock {\em Electron. J. Probab.}, 26:Paper No. 102, 32, 2021.

\bibitem{MR4346664}
Le~Chen, Davar Khoshnevisan, David Nualart, and Fei Pu.
\newblock Spatial ergodicity for {SPDE}s via {P}oincar\'{e}-type inequalities.
\newblock {\em Electron. J. Probab.}, 26:Paper No. 140, 37, 2021.

\bibitem{MR4563698}
Le~Chen, Davar Khoshnevisan, David Nualart, and Fei Pu.
\newblock Central limit theorems for spatial averages of the stochastic heat equation via {M}alliavin-{S}tein's method.
\newblock {\em Stoch. Partial Differ. Equ. Anal. Comput.}, 11(1):122--176, 2023.

\bibitem{MR2399293}
Daniel Conus and Robert~C. Dalang.
\newblock The non-linear stochastic wave equation in high dimensions.
\newblock {\em Electron. J. Probab.}, 13:no. 22, 629--670, 2008.

\bibitem{MR1684157}
Robert~C. Dalang.
\newblock Extending the martingale measure stochastic integral with applications to spatially homogeneous s.p.d.e.'s.
\newblock {\em Electron. J. Probab.}, 4:no. 6, 29, 1999.

\bibitem{MR1825714}
Robert~C. Dalang.
\newblock Corrections to: ``{E}xtending the martingale measure stochastic integral with applications to spatially homogeneous s.p.d.e.'s'' [{E}lectron {J}. {P}robab. {\bf 4} (1999), no. 6, 29 pp. (electronic); {MR}1684157 (2000b:60132)].
\newblock {\em Electron. J. Probab.}, 6:no. 6, 5, 2001.

\bibitem{MR1961163}
Robert~C. Dalang and Carl Mueller.
\newblock Some non-linear {S}.{P}.{D}.{E}.'s that are second order in time.
\newblock {\em Electron. J. Probab.}, 8:no. 1, 21, 2003.

\bibitem{DalangQuel}
Robert~C. Dalang and Lluís Quer-Sardanyons.
\newblock Stochastic integrals for spde’s: A comparison.
\newblock {\em Expositiones Mathematicae}, 29(1):67--109, 2011.

\bibitem{MR4164864}
Francisco Delgado-Vences, David Nualart, and Guangqu Zheng.
\newblock A central limit theorem for the stochastic wave equation with fractional noise.
\newblock {\em Ann. Inst. Henri Poincar\'{e} Probab. Stat.}, 56(4):3020--3042, 2020.

\bibitem{ebina2023central}
Masahisa Ebina.
\newblock Central limit theorems for nonlinear stochastic wave equations in dimension three.
\newblock {\em Stoch. Partial Differ. Equ. Anal. Comput.}, 12(2):1141--1200, 2024.

\bibitem{MR4167203}
Jingyu Huang, David Nualart, and Lauri Viitasaari.
\newblock A central limit theorem for the stochastic heat equation.
\newblock {\em Stochastic Process. Appl.}, 130(12):7170--7184, 2020.

\bibitem{MR4098872}
Jingyu Huang, David Nualart, Lauri Viitasaari, and Guangqu Zheng.
\newblock Gaussian fluctuations for the stochastic heat equation with colored noise.
\newblock {\em Stoch. Partial Differ. Equ. Anal. Comput.}, 8(2):402--421, 2020.

\bibitem{MR3617205}
Tuomas Hyt\"{o}nen, Jan van Neerven, Mark Veraar, and Lutz Weis.
\newblock {\em Analysis in {B}anach spaces. {V}ol. {I}. {M}artingales and {L}ittlewood-{P}aley theory}, volume~63 of {\em Ergebnisse der Mathematik und ihrer Grenzgebiete. 3. Folge. A Series of Modern Surveys in Mathematics [Results in Mathematics and Related Areas. 3rd Series. A Series of Modern Surveys in Mathematics]}.
\newblock Springer, Cham, 2016.

\bibitem{Kallenberg}
Olav Kallenberg.
\newblock {\em Foundations of modern probability}, volume~99 of {\em Probability Theory and Stochastic Modelling}.
\newblock Springer, Cham, 2021.
\newblock Third edition.

\bibitem{KaratzasShreve}
Ioannis Karatzas and Steven~E. Shreve.
\newblock {\em Brownian motion and stochastic calculus}, volume 113 of {\em Graduate Texts in Mathematics}.
\newblock Springer-Verlag, New York, second edition, 1991.

\bibitem{KarczewskaZabczyk}
Anna Karczewska and Jerzy Zabczyk.
\newblock Stochastic {PDE}'s with function-valued solutions.
\newblock In {\em Infinite dimensional stochastic analysis ({A}msterdam, 1999)}, volume~52 of {\em Verh. Afd. Natuurkd. 1. Reeks. K. Ned. Akad. Wet.}, pages 197--216. R. Neth. Acad. Arts Sci., Amsterdam, 2000.

\bibitem{MR4242879}
Davar Khoshnevisan, David Nualart, and Fei Pu.
\newblock Spatial stationarity, ergodicity, and {CLT} for parabolic {A}nderson model with delta initial condition in dimension {$d\geq1$}.
\newblock {\em SIAM J. Math. Anal.}, 53(2):2084--2133, 2021.

\bibitem{MR4337703}
Kunwoo Kim and Jaeyun Yi.
\newblock Limit theorems for time-dependent averages of nonlinear stochastic heat equations.
\newblock {\em Bernoulli}, 28(1):214--238, 2022.

\bibitem{MR2399277}
Jan Maas and Jan van Neerven.
\newblock A {C}lark-{O}cone formula in {UMD} {B}anach spaces.
\newblock {\em Electron. Commun. Probab.}, 13:151--164, 2008.

\bibitem{Mizohata}
Sigeru Mizohata.
\newblock {\em The theory of partial differential equations}.
\newblock Cambridge University Press, New York, 1973.
\newblock Translated from the Japanese by Katsumi Miyahara.

\bibitem{nourdin_peccati_2012}
Ivan Nourdin and Giovanni Peccati.
\newblock {\em Normal Approximations with Malliavin Calculus: From Stein's Method to Universality}.
\newblock Cambridge Tracts in Mathematics. Cambridge University Press, 2012.

\bibitem{MR2200233}
David Nualart.
\newblock {\em The {M}alliavin calculus and related topics}.
\newblock Probability and its Applications (New York). Springer-Verlag, Berlin, second edition, 2006.

\bibitem{nualart_nualart_2018}
David Nualart and Eulalia Nualart.
\newblock {\em Introduction to Malliavin Calculus}.
\newblock Institute of Mathematical Statistics Textbooks. Cambridge University Press, 2018.

\bibitem{MR4479916}
David Nualart, Panqiu Xia, and Guangqu Zheng.
\newblock Quantitative central limit theorems for the parabolic {A}nderson model driven by colored noises.
\newblock {\em Electron. J. Probab.}, 27:Paper No. 120, 43, 2022.

\bibitem{averagingGaussianfunctionals}
David Nualart and Guangqu Zheng.
\newblock Averaging {G}aussian functionals.
\newblock {\em Electron. J. Probab.}, 25:Paper No. 48, 54, 2020.

\bibitem{MR4187721}
David Nualart and Guangqu Zheng.
\newblock Spatial ergodicity of stochastic wave equations in dimensions 1, 2 and 3.
\newblock {\em Electron. Commun. Probab.}, 25:Paper No. 80, 11, 2020.

\bibitem{MR4439987}
David Nualart and Guangqu Zheng.
\newblock Central limit theorems for stochastic wave equations in dimensions one and two.
\newblock {\em Stoch. Partial Differ. Equ. Anal. Comput.}, 10(2):392--418, 2022.

\bibitem{MR1930613}
Szymon Peszat.
\newblock The {C}auchy problem for a nonlinear stochastic wave equation in any dimension.
\newblock {\em J. Evol. Equ.}, 2(3):383--394, 2002.

\bibitem{MR4391725}
Fei Pu.
\newblock Gaussian fluctuation for spatial average of parabolic {A}nderson model with {N}eumann/{D}irichlet/periodic boundary conditions.
\newblock {\em Trans. Amer. Math. Soc.}, 375(4):2481--2509, 2022.

\bibitem{MR2024344}
L.~Quer-Sardanyons and M.~Sanz-Sol\'{e}.
\newblock Absolute continuity of the law of the solution to the 3-dimensional stochastic wave equation.
\newblock {\em J. Funct. Anal.}, 206(1):1--32, 2004.

\bibitem{MR3078023}
Marta Sanz-Sol\'{e} and Andr\'{e} S\"{u}\ss.
\newblock The stochastic wave equation in high dimensions: {M}alliavin differentiability and absolute continuity.
\newblock {\em Electron. J. Probab.}, 18:no. 64, 28, 2013.

\bibitem{Schwartz}
Laurent Schwartz.
\newblock {\em Th\'{e}orie des distributions}.
\newblock Publications de l'Institut de Math\'{e}matique de l'Universit\'{e} de Strasbourg, IX-X. Hermann, Paris, 1966.
\newblock Nouvelle \'{e}dition, enti\'{e}rement corrig\'{e}e, refondue et augment\'{e}e.

\bibitem{singularintegrals}
Elias~M. Stein.
\newblock {\em Singular integrals and differentiability properties of functions}.
\newblock Princeton Mathematical Series, No. 30. Princeton University Press, Princeton, N.J., 1970.

\bibitem{Walsh}
John~B. Walsh.
\newblock An introduction to stochastic partial differential equations.
\newblock In P.~L. Hennequin, editor, {\em \'{E}cole d'\'Et\'e de Probabilit\'es de Saint Flour XIV - 1984}, pages 265--439, Berlin, Heidelberg, 1986. Springer Berlin Heidelberg.

\bibitem{MR1546100}
Norbert Wiener.
\newblock The ergodic theorem.
\newblock {\em Duke Math. J.}, 5(1):1--18, 1939.

\end{thebibliography}

\end{document}